\documentclass[a4paper,reqno]{amsart}
\usepackage{fullpage}
\usepackage{amsmath,amsthm,amssymb}
\usepackage{hyperref}
\usepackage[capitalize]{cleveref}
\usepackage{stmaryrd}
\usepackage[all]{xy}

\theoremstyle{plain}
\newtheorem{thm}{Theorem}[section]
\newtheorem{lem}[thm]{Lemma}
\newtheorem{cor}[thm]{Corollary}
\newtheorem{prop}[thm]{Proposition}
\theoremstyle{definition}
\newtheorem{df}[thm]{Definition}
\newtheorem{thm-df}[thm]{Theorem-Definition}
\newtheorem{cond}[thm]{Condition}
\theoremstyle{remark}
\newtheorem{rem}[thm]{Remark}
\newtheorem{eg}[thm]{Example}

\numberwithin{equation}{thm}
\crefname{thm}{Theorem}{Theorems}
\crefname{lem}{Lemma}{Lemmas}
\crefname{cor}{Corollary}{Corollaries}
\crefname{prop}{Proposition}{Propositions}
\crefname{df}{Definition}{Definitions}
\crefname{cond}{Condition}{Conditions}
\crefname{rem}{Remark}{Remarks}
\crefname{eg}{Example}{Examples}
\crefname{equation}{}{}

\newcommand{\bC}{\mathbb{C}}

\newcommand{\bZ}{\mathbb{Z}}
\newcommand{\bfF}{\mathbf{F}}
\newcommand{\bfG}{\mathbf{G}}
\newcommand{\bfH}{\mathbf{H}}
\newcommand{\bfJ}{\mathbf{J}}
\newcommand{\cA}{\mathcal{A}}
\newcommand{\cB}{\mathcal{B}}
\newcommand{\cC}{\mathcal{C}}
\newcommand{\cE}{\mathcal{E}}
\newcommand{\cF}{\mathcal{F}}
\newcommand{\cH}{\mathcal{H}}
\newcommand{\cI}{\mathcal{I}}
\newcommand{\cJ}{\mathcal{J}}
\newcommand{\cK}{\mathcal{K}}
\newcommand{\cL}{\mathcal{L}}
\newcommand{\cM}{\mathcal{M}}
\newcommand{\cN}{\mathcal{N}}
\newcommand{\cO}{\mathcal{O}}
\newcommand{\cS}{\mathcal{S}}
\newcommand{\cT}{\mathcal{T}}
\newcommand{\cU}{\mathcal{U}}
\newcommand{\cZ}{\mathcal{Z}}
\newcommand{\fh}{\mathfrak{h}}
\newcommand{\bx}{\mathbf{x}}
\newcommand{\by}{\mathbf{y}}
\newcommand{\sD}{\mathsf{D}}
\newcommand{\mrq}{\mathrm{q}}
\newcommand{\sm}{\mathsf{m}}

\providecommand{\tensor}[1]{\underset{#1}{\mathbin{\otimes}}}
\providecommand{\outensor}[1]{\underset{#1}{\mathbin{\boxtimes}}}
\providecommand{\cotensor}[1]{\underset{#1}{\mathbin{\square}}}
\providecommand{\dblcross}[1]{\underset{#1}{\mathbin{\bowtie}}}
\providecommand{\rbicross}[1]{\underset{#1}{\mathbin{\rtimes}}}
\providecommand{\redltimes}{\mathbin{\underset{r}{\ltimes}}}
\providecommand{\barltimes}{\mathbin{\bar{\ltimes}}}
\providecommand{\barotimes}{\mathbin{\bar{\otimes}}}
\newcommand{\bra}{\langle}
\newcommand{\ket}{\rangle}

\newcommand{\wh}{\widehat}
\newcommand{\wt}{\widetilde}
\newcommand{\rpb}{\hspace{0.1em}{}^*\!}
\DeclareMathOperator{\id}{id}
\DeclareMathOperator{\Ker}{Ker}
\DeclareMathOperator{\Hom}{Hom}
\DeclareMathOperator{\Aut}{Aut}
\DeclareMathOperator{\Rep}{Rep}

\DeclareMathOperator{\KK}{KK}
\DeclareMathOperator{\Calg}{C*alg}
\DeclareMathOperator{\Ad}{Ad} %adjoint by unitary
\DeclareMathOperator{\Ind}{Ind} %induction
\DeclareMathOperator{\Res}{Res} %restriction
\DeclareMathOperator{\op}{op} %opposite
\DeclareMathOperator{\cop}{cop} %co-opposite
\DeclareMathOperator{\bi}{bi}
\DeclareMathOperator{\loc}{loc}

\begin{document}
	\title[partial Pontryagin duality]{Partial Pontryagin duality for actions of quantum groups on C*-algebras}
	\author[K. Kitamura]{Kan Kitamura}
	\date{}
	\address{Graduate School of Mathematical Sciences, the University of Tokyo, 3-8-1 Komaba Meguro-ku Tokyo 153-8914, Japan}
	\email{kankita@ms.u-tokyo.ac.jp}
	\date{}
	
	\begin{abstract}
		We compare actions on C*-algebras of two constructions of locally compact quantum groups, the bicrossed product and the double crossed product. 
		We give a duality between them as a generalization of Baaj--Skandalis duality. 
		In the case of quantum doubles, this duality also preserves monoidal structures given by twisted tensor products. 
		We also discuss its consequences for equivariant Kasparov theories in relation to the quantum analogue of the Baum--Connes conjecture. 
	\end{abstract}
	\maketitle
	\setcounter{tocdepth}{1}
	\tableofcontents
	
	\section{Introduction}\label{sec:intro}
	
	For a given locally compact abelian group $G$, we can associate another locally compact abelian group $\wh{G}$ called the Pontryagin dual of $G$. 
	In the theory of operator algebras, Pontryagin duality and its generalizations for noncommutative situations have been investigated. 
	As one of such developments, Takesaki~\cite{Takesaki:dualitycrossprod} used crossed products and gave a duality between actions on von Neumann algebras of a locally compact abelian group and its Pontryagin dual. 
	The C*-algebraic counterpart of this duality was due to Takai~\cite{Takai:duality} and called Takesaki--Takai duality. 
	Another development was to seek the noncommutative framework for the group itself, which led to the operator-algebraic formulations of quantum groups. 
	There are also generalizations of Takesaki and Takesaki--Takai dualities for actions of quantum groups, such as the duality for actions of Kac algebras by Enock and Schwartz~\cite{Enock:crossprod,Enock-Schwartz:crossprod} and that of Hopf C*-algebras coming from a certain class of multiplicative unitaries by Baaj and Skandalis~\cite{Baaj-Skandalis:mltu}. 
	
	In this article, we investigate a generalization of these dualities for actions on C*-algebras of quantum groups $G^{\op}$ and $\wh{G}$ with an additional component of a quantum group $H$ in action. 
	We use the formulation of a quantum group called a locally compact quantum group due to Kustermans--Vaes~\cite{Kustermans-Vaes:lcqg,Kustermans-Vaes:lcqgvN}. 
	From given locally compact quantum groups $G$ and $H$ and a datum called a matching, 
	their bicrossed product was constructed by \cite{Vaes-Vainerman:bicrossed} and their double crossed product was constructed by \cite{Baaj-Vaes:doublecrossed} as locally compact quantum groups. 
	%The double crossed product consists of $G^{\op}$ and $H$ as quantum subgroups and he bicrossed product involves $\wh{G}$ as a quantum subgroups and $H$ as a quotient. 
	The first result in this article is that actions of such a pair of quantum groups are in one-to-one correspondence up to equivariant Morita equivalences under reasonable conditions. 
	See \cref{thm:BSTTdualC*} for the precise statement. 
	We show this by examining unitaries associated with these constructions calculated in \cite{Vaes-Vainerman:bicrossed,Baaj-Vaes:doublecrossed}. 
	
	We study this duality in the case of (generalized) quantum doubles as a particularly interesting case. 
	Then we get a duality between actions of the direct product $G\times G^{\op}$ and the quantum double $\sD(G)$. 
	For a compact quantum group $G$, we write $\Rep(G)$ for the rigid C*-tensor category of finite dimensional unitary representations of $G$. 
	When the compact quantum group $G$ associates a finite dimensional Hopf $*$-algebra, 
	this duality morally corresponds to the Morita equivalence of fusion categories, the Deligne tensor product $\Rep(G)\boxtimes \Rep(G)^{\mathrm{rev}}$ and the Drinfeld center $\cZ(\Rep(G))$, 
	where $\Rep(G)^{\mathrm{rev}}$ is the C*-tensor category $\Rep(G)$ with the reversed monoidal structure. 
	We still have such a phenomenon in the locally compact case. 
	
	We also study the monoidal structures at the level of algebras with actions. 
	Unlike the group case, for C*-algebras $A$ and $B$ with actions of a general locally compact quantum group $G$, the tensor product $A\otimes B$ does not admit a ``diagonal" $G$-action canonically. 
	However, when we consider actions of the quantum double $\sD(G)$, there is a construction to give a C*-algebra $A\outensor{G}B$ with a canonical $G$-action as investigated by \cite{Nest-Voigt:eqpd,Meyer-Roy-Woronowicz:twiten,Meyer-Roy-Woronowicz:twiten2}. 
	We would like to call it a twisted tensor product. 
	The second result in this article is that our duality procedure is compatible with twisted tensor products in the case of (generalized) quantum doubles. 
	%See \cref{prop:monbimod} for the precise statement. 
	
	Although many of the results can also be given in the setting of von Neumann algebras, 
	we will consider them totally in the setting of C*-algebras because one of our motivations comes from equivariant $KK$-theories. 
	Associated with a regular locally compact quantum group $G$, we have a category $\KK^G$ called the equivariant Kasparov category, consisting of separable C*-algebras with $G$-actions as objects and homotopy classes of equivariant Kasparov bimodules as morphisms. 
	It has a canonical structure of a triangulated category. 
	Also, when $G$ is a group, tensor products give a monoidal structure on $\KK^G$. 
	These structures on $\KK^G$ are studied in the context of tensor triangular geometry, such as \cite{dellAmbrogio:ttKK,dellAmbrogio-Meyer:primecycKK}. 
	Similarly, for a regular locally compact quantum group $G$, twisted tensor products give a monoidal structure on $\KK^{\sD(G)}$. 
	The second result for quantum doubles says $\KK^{\sD(G)}$ is equivalent to $\KK^{G\times \wh{G}}$ as monoidal triangulated categories in the sense of \cite{Nakano-Vashaw-Yakimov:nctentrigeom}, 
	where the associators in the latter category are trivial. 
	See \cref{thm:moneqPDYD} and \cref{cor:moneqPDqd} for the precise statements. 
	
	The triangulated structure of an equivariant Kasparov category was the key to the reformulation of the Baum--Connes conjecture by \cite{Meyer-Nest:bctri}. 
	From this viewpoint, several authors have investigated the quantum analogue of the conjecture \cite{Meyer-Nest:dualcpt,Voigt:bcfo,Vergnioux-Voigt:bcfu,Martos:semidirectBC,Voigt:cpxss,deCommer-Martos-Nest:projective}. 
	In the quantum analogue of the Baum--Connes conjecture for a discrete quantum group $G$, 
	it is sometimes convenient to consider $\sD(G)$. 
	One of the reasons is that we can utilize the monoidal structure of $\KK^{\sD(G)}$. 
	Indeed in \cite{Voigt:bcfo}, with the aid of twisted tensor products, an analogue of the Baum--Connes conjecture for $\wh{SU_q(2)}$ was deduced 
	from the result that $C(SU_q(2)/T)$ is $\KK^{\sD(SU_q(2))}$-equivalent to $\bC^{\oplus2}$ for $0<q<1$, where $T\leq SU(2)$ is a maximal torus, 
	and a similar result for $-1<q<0$. 
	Also, an analogue of the Baum--Connes conjecture for $q$-deformations complex semisimple Lie groups was considered in \cite{Voigt:cpxss} 
	by making use of the duality between actions of $\sD(G)$ and a semi-direct product for a compact group $G$. 
	Thus by resuming this strategy, we can hope our duality result will give a tool for computations of equivariant $K$-theory. 
	
	As an example, we determine the equivariant $KK$-theoretic classes of some quantum homogeneous spaces with actions of quantum doubles. 
	More precisely, we consider the graded twisting $G^\tau$ of 
	a connected simply connected compact Lie group $G$ constructed by \cite{Neshveyev-Yamashita:twilie}. 
	We show that $C(G^\tau/T)$ and $\bC^{\oplus n}$ are $\KK^{\sD(G^{\tau})}$-equivalent for some $n\in \bZ_{>0}$ (\cref{thm:grtwLiehmg}). 
	We do this by comparing actions of $G^{\tau}\times (G^{\tau})^{\op}$ and actions of an extension of $G\times G^{\op}$ with the aid of our duality result and an analogue of the Baum--Connes conjecture for $\sD(G^\tau)$. 
	
	We explain the organization of this article. 
	In \cref{sec:preliminaries}, we recall and fix the terminologies and notation on the concept of locally compact quantum groups. 
	We recall in \cref{sec:matching} the concepts of matching on locally compact quantum groups and its associated bicrossed product and double crossed product. 
	\cref{sec:procedure} contains the first main result of the duality between actions of the bicrossed product and the double crossed product (\cref{thm:BSTTdualC*}). 
	\cref{sec:qdmon} focuses on the case of generalized quantum doubles. 
	We show the second result that the duality procedure is compatible with twisted tensor products (\cref{prop:monbimod}). 
	We will make the statements for these results without languages of $K$-theories so far. 
	Then in \cref{sec:KK}, we explain the consequences for equivariant Kasparov theories. 
	The duality and its compatibility with twisted tensor products are restated as a categorical equivalence (\cref{thm:BSTTdualcat}) and a monoidal equivalence (\cref{thm:moneqPDYD}) of equivariant Kasparov categories, respectively. 
	Also, we briefly discuss a quantum analogue of the Baum--Connes conjecture in \cref{ssec:gamma}. 
	Finally, \cref{sec:grtwhmg} is about the equivariant Kasparov theory of quantum homogeneous spaces of graded twisting with the action of the quantum double (\cref{thm:grtwLiehmg}). 
	
	\subsection*{Acknowledgments}
	The author would like to express deep gratitude to his advisor Yasuyuki Kawahigashi for his invaluable support. 
	Also, he would like to thank Yuki Arano and Yosuke Kubota for the discussions which led his attention to the partial Pontryagin duality of quantum groups, 
	and Christian Voigt for stimulating comments on the early version of the article. 
	Finally, the author would like to thank the anonymous referee for careful reading and comments, which improved the exposition.
	This work was supported by JSPS KAKENHI Grant Number JP21J21283, JST CREST program JPMJCR18T6, JSPS Overseas Challenge Program for Young Researchers, and the WINGS-FMSP program at the University of Tokyo. 
	
	\section{Preliminaries}\label{sec:preliminaries}
	
	We fix general notation throughout the article. 
	Let $A,B$ be C*-algebras. 
	For a bounded linear map $f\colon A\to \cM(B)$, 
	we abuse notation and simply write $f\colon \cM(A)\to\cM(B)$ for 
	the unique extension of $f$ to a linear map $\cM(A)\to\cM(B)$ 
	that is strictly continuous on the unit ball of $\cM(A)$ if it exists, 
	and similarly for the normal extension to the map between von Neumann algebras. 
	We use the symbol $\otimes$ for the spatial tensor products of C*-algebras and those of Hilbert spaces, and $\barotimes$ for von Neumann algebraic tensor products. 
	The symbol $\sigma$ is used for the flipping map $\sigma\colon A\otimes B\cong B\otimes A$. 
	We also use the so-called leg-numbering notation. 
	When $Y$ is a subset of a Banach space $X$, 
	we write $[Y]$ for the norm closure of the linear span of $Y$ in $X$. 
	Note that our notation is mostly compatible with that in \cite{Kitamura:indhomlcqg}, 
	except that we use the terminology ``action" instead of ``coaction" 
	and different legs in the definition of the twisted tensor products (\cref{def:twiten}). 
	
	\subsection{Locally compact quantum groups}\label{ssec:lcqg}
	
	We briefly recall and fix the notation about locally compact quantum groups and multiplicative unitaries. For details, we refer to \cite{Kustermans-Vaes:lcqg,Kustermans-Vaes:lcqgvN,Vaes:thesis} for locally compact quantum groups and \cite{Baaj-Skandalis:mltu,Woronowicz:mltu} for multiplicative unitaries. 
	
	\begin{df}
		Let $G=(L^\infty(G),\Delta_G)$ be a pair of a von Neumann algebra $L^\infty(G)$ 
		and a normal unital $*$-homomorphism 
		$\Delta_G\colon L^\infty(G)\to L^\infty(G)\barotimes L^\infty(G)$ satisfying \emph{coassociativity} in the sense of 
		$(\Delta_G\otimes\id)\Delta_G = (\id\otimes\Delta_G)\Delta_G$. 
		\begin{enumerate}
			\item
			A normal semi-finite weight $\varphi$ 
			on $L^\infty(G)$ 
			is called \emph{left} [resp.~\emph{right}] \emph{invariant} if 
			for any $x\in \cM^+_\varphi:=\{ x\in L^\infty(G)_+\,|\, \varphi(x)<\infty \}$ and any positive normal linear functional $\omega\colon  L^\infty(G)\to\bC$ 
			it holds that 
			\begin{align*}
				&
				\varphi(x)\omega(1) = \varphi((\omega\otimes \id)\Delta_G(x))
				\quad 
				\text{[resp.~$\varphi(x)\omega(1) = \varphi((\id \otimes \omega)\Delta_G(x))$\ ]. }
			\end{align*}
			\item
			The pair $G=(L^\infty(G),\Delta_G)$ is called a 
			\emph{locally compact quantum group} if 
			there exist a faithful left invariant weight and a faithful right invariant weight. 
			Then we call its faithful left [resp.~right] invariant weight a \emph{left} [resp.~\emph{right}] \emph{Haar weight}. 
		\end{enumerate}
	\end{df}
	
	Let $G=(L^{\infty}(G),\Delta_G)$ be a locally compact quantum group. 
	Take the left Haar weight $\varphi$ of $G$, which is unique up to a scalar multiplication. 
	Via the GNS construction of $\varphi$, denoted by 
	\begin{align*}
		\Lambda \colon \cN_\varphi := \{x\in L^{\infty}(G) \,|\, \varphi(x^*x)<\infty\}\to L^2(G), 
	\end{align*}
	we identify $L^\infty(G)$ as a von Neumann subalgebra of $\cB(L^2(G))$. 
	We have a well-defined unitary $W^G \in \cU(L^2(G)^{\otimes 2})$ characterized by 
	\begin{align*}
		W^{G*}(\Lambda(x)\otimes \Lambda(y)) 
		= (\Lambda\otimes \Lambda)\bigl(\Delta_G(y)(x\otimes 1)\bigr) 
	\end{align*}
	for all $x,y\in \cN_\varphi$, 
	which turns out to be a \emph{multiplicative unitary} in the sense of 
	\begin{align*}
		W^G_{23}W^G_{12}=W^G_{12}W^G_{13}W^G_{23}. 
	\end{align*}
	It holds $\Delta_G=\Ad W^G{}^*(1\otimes -)$ on $L^\infty(G)$. 
	
	Associated with $G$, there is another locally compact quantum group $\wh{G}$, which we call the (Pontryagin) dual of $G$, such that 
	the GNS construction of the left Haar weight of $\wh{G}$ can be canonically identified with $L^2(G)$ 
	and $W^{\wh{G}}=W^{G*}_{21}$ under this identification. 
	We write $\wh{W}^{G}:=W^{\wh{G}}$. 
	
	We write $J^G$ and $\wh{J}^{G}:=J^{\wh{G}}$ for the modular conjugations of the left Haar weight of $G$ and $\wh{G}$, respectively. 
	Sometimes we use the fact that $U^G:=J^G\wh{J}^G$ is a scalar multiplication of $\wh{J}^GJ^G$. 
	Then we have another multiplicative unitary 
	\begin{align*}
		&
		V^G := (\wh{J}^{G}\otimes \wh{J}^{G})
		\wh{W}^{G}(\wh{J}^{G}\otimes \wh{J}^{G}) 
		\in \cU(L^2(G)^{\otimes 2}) , 
	\end{align*}
	which satisfies $\Delta_G=\Ad V^G(-\otimes 1)$. 
	When we write $\Delta_G^{\cop}:=\sigma \Delta_G$, then
	$G^{\op}:=(L^{\infty}(G),\Delta_G^{\cop})$ is also a locally compact quantum group 
	and we can canonically identify $W^{G^{\op}}=V^{G*}_{21}$. 
	We have an anti-$*$-isomorphism of the von Neumann algebra called the unitary antipode, 
	$R^G:=\wh{J}^{G}(-)^*\wh{J}^{G}\colon L^\infty(G)\to L^\infty(G)$ 
	with $\Delta_G^{\cop} R^G = (R^G\otimes R^G)\Delta_G$. 
	We also write $\wh{\Delta}_{G}:=\Delta_{\wh{G}}$ and $\wh{V}^G:=V^{\wh{G}}$ for readability. 
	
	We shall write $\cK(G)=\cK(L^2(G))$ for short, and often identify $\cK(G)^*$ as $\cB(L^2(G))_*$. 
	We also write $C^r_0(G):=[(\id_{\cK(G)}\otimes\cK(G)^*) (W^G)] \subset \cB(L^2(G))$, which is a well-defined C*-algebra with $C^r_0(G)''=L^{\infty}(G)$. 
	Note that $C^r_0(G^{\op})=C^r_0(G)$ as C*-subalgebras of $\cB(L^2(G))$. 
	The restriction of $\Delta_G$ to $C^r_0(G)$ gives a well-defined non-degenerate $*$-homomorphism $\Delta_G\colon C^r_0(G)\to \cM(C^r_0(G)\otimes C^r_0(G))$, which satisfies 
	\begin{align*}
		[\Delta_G(C^r_0(G)) (1\otimes C^r_0(G))] 
		= C^r_0(G)\otimes C^r_0(G) 
		= [\Delta_G(C^r_0(G)) (C^r_0(G)\otimes 1)] . 
	\end{align*}
	It holds $W^G\in \cU\cM(C^r_0(G)\otimes C^r_0(\wh{G}))$. 
	A locally compact quantum group $G$ is said to be \emph{regular} if $[C^r_0(G)C^r_0(\wh{G})]=\cK(G)$, see also \cite{Baaj-Skandalis-Vaes:nonsemireg}. 
	When $G$ is regular, $\wh{G}$ and $G^{\op}$ are also regular and it holds 
	\begin{align*}
		&
		[C^r_0(G)_1 W^G C^r_0(\wh{G})_2] 
		= C^r_0(G)\otimes C^r_0(\wh{G}) 
		= [C^r_0(G)_1 W^{G*} C^r_0(\wh{G})_2] . 
	\end{align*}
	
	\subsection{Actions and homomorphisms}\label{ssec:homomorphism}
	We say a pair $(\cH,X)$ of a Hilbert space $\cH$ and a unitary $X\in L^\infty(G)\barotimes \cB(\cH)$ is a \emph{unitary representation} of $G$ if $(\Delta_G\otimes\id)(U)=U_{13}U_{23}$. 
	We introduce actions of locally compact quantum groups on operator algebras. 
	\begin{df}
		Let $G$ be a locally compact quantum group and $A$ be a C*-algebra. 
		Consider a non-degenerate injective $*$-homomorphism $\alpha\colon A\to  \cM(C^r_0(G)\otimes A)$. We say that $\alpha$ is a \emph{continuous left $G$-action} on $A$, or that $(A,\alpha)$ is a \emph{left $G$-C*-algebra}, 
		if 
		\begin{align*}
			(\id\otimes\alpha)\alpha=(\Delta_G\otimes\id)\alpha 
			\quad\text{and}\quad 
			\left[ (C^r_0(G)\otimes 1) \alpha(A) \right] = C^r_0(G)\otimes A . 
		\end{align*}
		Often we omit $\alpha$ and say $A$ is a left $G$-C*-algebra. 
		For a $*$-homomorphism $\alpha\colon A\to\cM(A\otimes C^r_0(G))$ with $(A,\sigma\alpha)$ being a left $G^{\op}$-C*-algebra, 
		we say that $\alpha$ is a \emph{continuous right $G$-action} on $A$, or that $(A,\alpha)$ is a \emph{right $G$-C*-algebra}. 
	\end{df}
	
	\begin{df}
		Let $G$ be a locally compact quantum group. 
		For left $G$-C*-algebras $(A,\alpha)$, $(B,\beta)$, 
		we say a $*$-homomorphism $f\colon A\to \cM(B)$ 
		(which need not be non-degenerate) is a \emph{left $G$-$*$-homomorphism} 
		if for all $a\in A$ and $x\in C^r_0(G)$, 
		\begin{align*}
			(x\otimes 1)(\beta f(a))=(\id\otimes f)((x\otimes 1)\alpha(a)) . 
		\end{align*}	
		Similarly we define a right $G$-$*$-homomorphism of right $G$-C*-algebras. 
	\end{df}
	
	For a left $G$-C*-algebra $(A,\alpha)$, we have its \emph{reduced crossed product} $G\redltimes A:= [C^r_0(\wh{G})_1 \alpha(A)]\subset \cM(\cK(G)\otimes A)$. 
	This is a left $\wh{G}^{\op}$-C*-algebra with $\Ad \wh{V}^G_{21}(-)_{23}$, called the dual action. 
	
	We also treat actions of quantum groups on Hilbert modules by following \cite[Section~2]{Baaj-Skandalis:eqkk}, \cite[Appendix]{Vaes:impr}. 
	We briefly note that for a left $G$-C*-algebra $(B,\beta)$, 
	giving a continuous left $G$-action on a right Hilbert $B$-module $\cE$ is equivalent to 
	giving a continuous left $G$-action $\kappa$ on $\cK_B(\cE\oplus B)$ such that $\kappa\restriction_B=\beta$ 
	via restriction to the corner. 
	For left $G$-C*-algebras $A$ and $B$, 
	a \emph{left $G$-$(A,B)$-correspondence} is a pair $(\cE,\phi)$ of a left $G$-right Hilbert $B$-module $\cE$ (i.e.~a right Hilbert $B$-module with a continuous left $G$-action) 
	and a left $G$-$*$-homomorphism $\phi\colon A\to \cM(\cK_B(\cE))=\cL_B(\cE)$. 
	When additionally $\phi(A)\subset \cK_B(\cE)$, the correspondence $(\cE,\phi)$ is called \emph{proper}. 
	Often we abuse notation to omit $\phi$. 
	For a left $G$-$(A,B)$-correspondence $\cE$ and a left $G$-$(B,D)$-correspondence $\cF$, the tensor product $\cE\tensor{B}\cF$ has the canonical structure of a left $G$-$(A,D)$-correspondence. 
	More precisely, when we write $\kappa_\cE$ and $\kappa_\cF$ for the continuous left $G$-actions on $\cE$ and $\cF$ respectively, the left $G$-action $\kappa$ on $\cE\tensor{B}\cF$ is determined by 
	\begin{align*}
		&
		(x\otimes1_{\cE\otimes\cF})\kappa(\xi \tensor{B} \eta) (y\otimes1_B)
		= 
		\bigl((x\otimes1_\cE)\kappa_\cE(\xi)\bigr)\tensor{C^r_0(G)\otimes B}\bigl(\kappa_\cF(\eta)(y\otimes1_B)\bigr) 
	\end{align*}
	for $x,y\in C^r_0(G)$, $\xi\in\cE$, and $\eta\in\cF$. 
	A left $G$-$(A,B)$-correspondence $\cE$ with a left $G$-$(B,A)$-correspondence $\cF$ such that $\cE\tensor{B}\cF\cong A$ and $\cF\tensor{A}\cE\cong B$ as $G$-C*-correspondences is called a \emph{left $G$-$(A,B)$-imprimitivity bimodule}. In this situation, we say $A$ is \emph{left $G$-Morita equivalent} to $B$. 
	
	Next we introduce homomorphisms of quantum groups by using bicharacters. See \cite{Meyer-Roy-Woronowicz:qghom} for details. 
	\begin{df}
		Let $G,H$ be locally compact quantum groups. 
		A unitary $X\in \cU(L^2(H)\otimes L^2(G))$ is called a \emph{bicharacter} from $H$ to $G$ if 
		\begin{align*}
			W^G_{23}X_{12} = X_{12}X_{13}W^{G}_{23}
			\quad\text{and}\quad
			X_{23}W^H_{12}=W^H_{12}X_{13}X_{23} .
		\end{align*}
		We define $\Hom_{\mathrm{qg}}(H,G)$ as the index set such that 
		$\{ W^{\phi} \}_{\phi\in \Hom_{\mathrm{qg}}(H,G)}$ is the set of bicharacters from $H$ to $G$, and 
		we call an element of $\Hom_{\mathrm{qg}}(H,G)$ a \emph{homomorphism} $\phi\colon H\to G$ of locally compact quantum groups. 
	\end{df}
	
	To understand the definition, it is good to keep in mind that a non-degenerate $*$-homomorphism $f\colon C^r_0(G)\to \cM(C^r_0(H))$ with $\Delta_H f=(f\otimes f)\Delta_G$ gives a bicharacter $(f\otimes\id)(W^G)\in\cU\cM(C^r_0(H)\otimes C^r_0(\wh{G}))$. 
	Let $\phi\colon H\to G$, $\psi\colon G\to F$ be homomorphisms of locally compact quantum groups. 
	Then $W^{\phi}\in\cU\cM(C^r_0(H)\otimes C^r_0(\wh{G}))$ and $(R^H\otimes R^{\wh{G}}) (W^{\phi}) = W^{\phi}$. 
	We put the unitary 
	\begin{align*}
		V^\phi := (\wh{J}^{G}\otimes \wh{J}^{H})W^{\phi*}_{21}(\wh{J}^{G}\otimes \wh{J}^{H}) \in\cU\cM(\wh{J}^{G}C^r_0(\wh{G})\wh{J}^{G}\otimes C^r_0(H)) . 
	\end{align*}
	There are well-defined homomorphisms 
	$\wh{\phi}\colon \wh{G}\to \wh{H}$ and $\phi^{\op}\colon H^{\op}\to G^{\op}$
	with $W^{\wh{\phi}} = \wh{W}^{\phi} := W^{\phi*}_{21}$ 
	and $W^{\phi^{\op}} := V^{\phi*}_{21}$. 
	We also put $\wh{V}^{\phi}:=V^{\wh{\phi}}$. 
	We define the composition $\psi\phi\colon H\to F$ by the corresponding well-defined bicharacter 
	$W^{\psi\phi}$ with the relation $W^{\psi\phi}_{13} = W^{\phi*}_{12}W^{\psi}_{23}W^{\phi}_{12}W^{\psi*}_{23}$. 
	Then we get the category of locally compact quantum groups and their homomorphisms, 
	where the identity $\id_G\colon G\to G$ is given by $W^{\id_G}:=W^G$. 
	We have the trivial homomorphism $1_{H\to G}\colon H\to G$ 
	determined by the bicharacter $1\in\cU\cM(C^r_0(H)\otimes C^r_0(\wh{G}))$. 
	
	For a left [resp.~right] $G$-C*-algebra $(A,\alpha)$, we can define the \emph{restriction} along $\phi$ as the continuous left [resp.~right] $H$-action 
	$\phi^*\alpha$ [resp.~$\alpha\rpb\phi$] on $A$ such that 
	\begin{align*}
		&
		(\id\otimes\alpha)(\phi^*\alpha)=\Ad W^{\phi*}_{12}(\alpha(-)_{23}) \quad
		\text{[resp. $(\alpha\otimes\id)(\alpha\rpb\phi)=\Ad V^{\phi}_{23}(\alpha(-)_{12})$\ ]}, 
	\end{align*}
	which uniquely exists. 
	Note that when $W^\phi=(f\otimes\id)(W^G)$ for some non-degenerate $*$-homomorphism $f\colon C^r_0(G)\to \cM(C^r_0(H))$ with $\Delta_H f=(f\otimes f)\Delta_G$, it holds $\phi^*\alpha=(f\otimes\id)\alpha$ [resp.~$\alpha\rpb\phi=(\id\otimes f)\alpha$]. 
	When we consider $A$ with this restriction we write $\phi^*A$ or $A\rpb \phi$, but sometimes we suppress $\phi$ and write $\Res^G_H A$ or more simply $A$ when no confusion is likely to occur. 
	By regarding $\Delta_G$ as a continuous left or right $G$-action on $C^r_0(G)$, 
	it holds that 
	\begin{align*}
		&
		\phi^*\Delta_G=\Ad W^{\phi*}(-)_2, 
		\quad
		\Delta_G\rpb\phi=\Ad V^{\phi}(-)_1, 
		\quad\text{and}\quad
		\sigma(\Delta_G\rpb \phi)R^G = (R^H\otimes R^G)(\phi^*\Delta_G). 
	\end{align*}
	
	\section{Bicrossed products and double crossed products}\label{sec:matching}
	
	In this section, we recall the concept of a matching and its associated two constructions of locally compact quantum groups by following \cite{Vaes-Vainerman:bicrossed,Baaj-Vaes:doublecrossed}. 
	
	\begin{df}\label{def:matching}
		Let $G$ and $H$ be locally compact quantum groups. 
		Then a $*$-automorphism $\sm\in\Aut(L^\infty(G)\barotimes L^\infty(H))$ is said to be a \emph{matching} on $G$ and $H$ if 
		\begin{align*}
			&
			(\Delta_G\otimes \id)\sm=\sm_{23}\sm_{13}(\Delta_G\otimes \id), 
			\quad\mathrm{and}\quad
			(\id\otimes \Delta_H)\sm=\sm_{13}\sm_{12}(\id\otimes \Delta_H). 
		\end{align*}
	\end{df}
	
	Note that the restriction of $\sm$ to $1\otimes L^\infty(H)$ gives a von Neumann algebraic left $G$-action on $L^\infty(H)$. Similarly, $\sm_{21}(1\otimes-)$ is a von Neumann algebraic left $H$-action on $L^\infty(G)$. 
	
	\begin{thm-df}[{\cite[Theorem~2.13]{Vaes-Vainerman:bicrossed}}, {\cite[Theorem~5.3]{Baaj-Vaes:doublecrossed}}]
		Let $\sm$ be a matching on locally compact quantum groups $G$ and $H$. 
		\begin{enumerate}
			\item
			We have a well-defined locally compact quantum group $(G\barltimes L^\infty(H), \Delta_{\rtimes})$, 
			where $G\barltimes L^\infty(H) = ( L^\infty(\wh{G})_1 \sm(L^\infty(H)_2) )''$ 
			and $\Delta_\rtimes \colon G\barltimes L^\infty(H)\to (G\barltimes L^\infty(H))^{\barotimes2}$ is characterized by 
			\begin{align*}
				&
				\Delta_{\rtimes}\sm(-)_2 
				= (\sm\otimes\sm)(\Delta_H(-)_{24})
				\text{\ \ on\ \ }L^\infty(H), 
				\\&
				(\id\otimes\Delta_{\rtimes})(W^G_{12}) 
				= W^{G}_{14} (\sm_{45}\sm_{15}W^G_{12}). 
			\end{align*}
			This is called the \emph{bicrossed product}, and 
			we will write $\wh{G}\rbicross{\sm}H$ for it. 
			\item
			We have a well-defined locally compact quantum group $\bigl( L^\infty(G)\barotimes L^\infty(H),  \sigma_{23}\sm_{23}(\Delta_G^{\cop}\otimes\Delta_H) \bigr)$.
			This is called the \emph{double crossed product}, and 
			we will write $G^{\op}\dblcross{\sm}H$ for it. 
		\end{enumerate}
	\end{thm-df}
	
	Like $\Delta_{\rtimes}$, throughout this article we shall 
	use the symbol $\rtimes$ instead of $\wh{G}\rbicross{\sm}H$ when it is used as indices, 
	such as $R^{\rtimes}$, $T^{\rtimes}$, $\wh{T}^{\rtimes}$ for a symbol $T\in \{W,V,J\}$, and $U^{\rtimes}$ to simplify notation. 
	We also put $U^l:=J^{\rtimes}(\wh{J}^G\otimes J^H)$, 
	$U^r:=\wh{J}^{\rtimes}(J^G\otimes \wh{J}^H)$, and 
	$Z^{\sm}:= U^{\rtimes}(U^{G*}\otimes U^{H*})
	%=J^{\rtimes}\wh{J}^{\rtimes}(\wh{J}^GJ^G)_1(\wh{J}^HJ^H)_2
	$. 
	
	We extract from \cite{Vaes-Vainerman:bicrossed,Baaj-Vaes:doublecrossed} the relations of unitaries we will use later. 
	\begin{thm}[\cite{Vaes-Vainerman:bicrossed,Baaj-Vaes:doublecrossed}]\label{thm:relations}
		For a matching $\sm$ on locally compact quantum groups $G$ and $H$, we have the following. 
		\begin{enumerate}
			\item
			$W^{\rtimes}=(\sm_{12} W^H_{24}) (\sm_{34} \wh{W}^G_{13})$. 
			\item
			$U^l\in \cM(C^r_0(G)\otimes \cK(H))$ and 
			$(\Delta_G\otimes\id)U^l=U^l_{23} U^l_{13}$. 
			Also, $\Ad U^l(-)_2=\sm (-)_2$ on $L^\infty(H)$. 
			\item
			$U^r\in \cM(\cK(G)\otimes C^r_0(H))$ and 
			$(\id\otimes\Delta_H)U^r=U^r_{13}U^r_{12}$. 
			Also, $\Ad U^r(-)_1=\sm (-)_1$ on $L^\infty(G)$. 
			\item
			$\Ad Z^{\sm}=\sm$ on $L^\infty(G)\barotimes L^\infty(H)$. 
			\item
			%$W^{G^{\op}\dblcross{\sm}H} = V^{G*}_{31}Z^{\sm*}_{34}W^H_{24}Z^{\sm}_{34} = V^{G*}_{31}U^{r*}_{32} (\sm_{32}W^H_{24})$. Especially, $Z^{\sm*}_{34}W^H_{24}Z^{\sm}_{34} = U^{r*}_{32} (\sm_{32}W^H_{24})$. 
			$V^{G^{\op}\dblcross{\sm}H} = Z^{\sm*}_{12}\wh{W}^G_{13}Z^{\sm}_{12}V^H_{24} = \sm_{32}(\wh{W}^G_{13}) U^{l}_{32} V^H_{24}$. 
			Especially, $Z^{\sm*}_{12}\wh{W}^G_{13}Z^{\sm}_{12} = \sm_{32}(\wh{W}^G_{13}) U^{l}_{32}$. 
			\item
			$U^{\rtimes} = J^{\rtimes}\wh{J}^{\rtimes}$ and $J^G\otimes J^H$ commutes up to a scalar multiplication. 
			\item
			$\sm_{23}(U^r_{13}V^G_{12}) = V^G_{12}U^r_{13}$. 
		\end{enumerate}
	\end{thm}
	
	(1) is due to \cite[Definition~2.2, Theorem~2.13]{Vaes-Vainerman:bicrossed}. 
	(2) follows from \cite[Proposition~3.7, Proposition~3.12, Theorem~4.4]{Vaes:unitaryimplement} with the fact that the left Haar weight of $\wh{G}\rbicross{\sm}H$ is the dual weight of $G\barltimes L^\infty(H)$ by \cite[Proposition~2.8]{Vaes-Vainerman:bicrossed}. 
	Since $\wh{H}\rbicross{\sigma\sm\sigma}G$ is isomorphic to the dual of $\wh{G}\rbicross{\sm}H$ 
	via $\sigma\colon L^\infty\Bigl(\wh{\wh{G}\rbicross{\sm}H}\Bigr) = (\sm(L^\infty(G)_1)L^\infty(\wh{H})_2)'' \to H\barltimes L^\infty(G)$
	by the expression of $W^{\rtimes}$, 
	we see $U^r_{21}$ is the $U^l$ for $\sigma\sm\sigma$ and thus (3) holds. 
	For (4), see \cite[Lemma~3.3]{Baaj-Vaes:doublecrossed}. 
	(5), (6), and (7) are due to \cite[Theorem~5.3]{Baaj-Vaes:doublecrossed}. 
	
	We also note that $\sigma\sm\sigma$ is a matching on $H$ and $G$ 
	and we have $H^{\op}\dblcross{\sigma\sm\sigma}G\cong (G^{\op}\dblcross{\sm}H)^{\op}$ 
	via $\sigma\colon L^\infty(G)\barotimes L^\infty(H)\to L^\infty(H)\barotimes L^\infty(G)$. 
	
	\begin{rem}\label{rem:componentbidbl}
		For a matching $\sm$ on locally compact quantum groups $G$ and $H$, we have the following homomorphisms. 
		\begin{enumerate}
			\item
			We have a homomorphism $\pi_H\colon \wh{G}\rbicross{\sm} H\to H$ given by 
			$\sm(-)_2 \colon L^\infty(H)\to G\barltimes L^\infty(H)$. 
			More precisely, its corresponding bicharacter is $W^{\pi_H}:=\sm_{12}W^H_{23}$. 
			We have another homomorphism 
			$\wh{\pi}_G\colon \wh{G}\to \wh{G}\rbicross{\sm}H$ corresponding to the bicharacter 
			$W^{\wh{\pi}_G} := \sm_{23}\wh{W}^G_{12}$, 
			which is well-defined as the dual of $\pi_G\colon \wh{H}\rbicross{\sigma\sm\sigma}G\to G$. 
			Note that 
			\begin{align}\label{eq:rem:componentbidbl}
				&
				V^{\wh{\pi}_G} 
				=
				(\wh{J}^{\rtimes}J^{\rtimes})_{12} W^{\wh{\pi}_G}_{312} (\wh{J}^{\rtimes}J^{\rtimes})_{12}^* 
				= 
				\wh{V}^G_{13} . 
			\end{align}
			\item
			We have two homomorphisms 
			$\imath_{G} \colon G^{\op}\to G^{\op}\dblcross{\sm}H$ and 
			$\jmath_{H} \colon H\to G^{\op}\dblcross{\sm}H$ such that 
			\begin{align*}
				&
				\imath_{G}^*\Delta_{G^{\op}\dblcross{\sm}H}=\Delta_G^{\cop}\otimes\id \quad \text{and}\quad \jmath_{H}^*\Delta_{G^{\op}\dblcross{\sm}H}=(\sigma\sm\otimes\id)(\id\otimes\Delta_H) , 
			\end{align*}
			which are well-defined by \cite[Proposition~A.3]{Kitamura:indhomlcqg}, for example. 
		\end{enumerate}
	\end{rem}
	
	In order to treat continuous actions on C*-algebras, we need some continuity on the matching as follows. 
	
	\begin{df}
		We say a matching $\sm$ of locally compact quantum groups $G$ and $H$ is \emph{continuous} if the restriction of $\sm$ gives a $*$-automorphism on $C^r_0(G)\otimes C^r_0(H)$. 
	\end{df}
	
	For a continuous matching $\sm$ on regular locally compact quantum groups $G$ and $H$, 
	it follows from \cite[Section~9]{Baaj-Vaes:doublecrossed} 
	that $\wh{G}\rbicross{\sm}H$ and $G^{\op}\dblcross{\sm}H$ are also regular 
	and that 
	\begin{align*}
		&
		C^r_0(\wh{G}\rbicross{\sm}H) = [C^r_0(\wh{G})_1 \sm(C^r_0(H)_2)]  
		\quad \text{and}\quad
		C^r_0(G^{\op}\dblcross{\sm}H)=C^r_0(G)\otimes C^r_0(H) 
	\end{align*}
	as C*-subalgebras of $L^\infty(\wh{G}\rbicross{\sm}H)$ and $L^\infty(G^{\op}\dblcross{\sm}H)$, respectively. 
	
	\begin{eg}[Generalized quantum double]\label{eg:matchingYD}
		Let $\phi\colon H\to G$ be a homomorphism of locally compact quantum groups. 
		Then $\sm:=\Ad W^{\phi}$ is a continuous matching on $H^{\op}$ and $\wh{G}$ and we write $\sD(\phi):= H\dblcross{\Ad W^{\phi}} \wh{G} 
		= ( L^\infty(H)\barotimes L^\infty(\wh{G}),\, \sigma_{23} (\Ad W^{\phi}_{23}) (\Delta_H\otimes\wh{\Delta}_{G}) )$, 
		which is a locally compact quantum group called the \emph{generalized quantum double}. See \cite[Section~8]{Baaj-Vaes:doublecrossed} for more details. 
		When $\phi=\id_G$, this construction gives the quantum double $\sD(G):=\sD(\id_G)$ of $G$. 
		Also, note $\sD(1_{H\to G})=H\times \wh{G}$. 
	\end{eg}
	
	\begin{eg}[Semi-direct product]\label{eg:matchingsemidirect}
		For a discrete group $G$ and a locally compact quantum group $H$, consider a group homomorphism $\theta\colon G\to \Aut(L^\infty(H))$ with $\Delta_H \theta_t = (\theta_t\otimes\theta_t)\Delta_H$ for all $t\in G$. 
		Then a continuous matching $\sm$ on $G^{\op}$ and $H$ can be given by 
		$\sm(\delta_t\otimes x):=\delta_t\otimes \theta_t(x)$ for $t\in\Gamma$, $x\in L^\infty(H)$. 
		Here $\delta_t(s):=\delta_{t,s}$ for $s,t\in\Gamma$ by using the Kronecker delta. 
		When $H$ is a locally compact group, $G\dblcross{\sm}H$ is nothing but the semi-direct product $G\ltimes H$ for the corresponding action of $G$ on $H$ as homeomorphic group automorphisms. 
	\end{eg}
	
	When the result is a compact quantum group, the bicrossed product construction was studied in detail by \cite{Fima:compact}. 
	Also, when $H$ is a compact quantum group in the situation of \cref{eg:matchingsemidirect}, the bicrossed product $\wh{G}\rbicross{\sm}H$ coincides with the construction of \cite[Theorem~1.5]{Wang:tensorprod}. For $K$-theoretic properties of such compact quantum groups, see \cite{Martos:semidirectBC}. 
	
	\subsection{Decomposition of actions of double crossed product}
	
	It is convenient to decompose actions of double crossed products into smaller pieces with a relation that is analogous to the Yetter--Drinfeld condition. 
	
	\begin{df}\label{def:YDcond}
		Consider a continuous matching $\sm$ on regular locally compact quantum groups $G$ and $H$. 
		Let $A$ be a C*-algebra, $\alpha$ 
		be a continuous left $G^{\op}$-action on $A$, and $\lambda$ 
		be a continuous left $H$-action on $A$. 
		We say $(A,\alpha,\lambda)$ is a \emph{left $\sm$-YD C*-algebra} 
		if 
		\begin{align}\label{eq:prop:YDcond}
			&
			\sm_{13}(\id\otimes\sigma\lambda)\alpha=(\alpha\otimes\id)\sigma\lambda, 
		\end{align}
		or equivalently, if the following diagram commutes, 
		\begin{align*}
			&
			\begin{aligned}
				\xymatrix@R=1em{
					& \cM(C^r_0(G)\otimes A) \ar[r]^-{\id\otimes\sigma\lambda} &\cM(C^r_0(G)\otimes A\otimes C^r_0(H)) \ar[dd]^-{\sm_{13}}\,\\
					A \ar[dr]_-{\sigma\lambda}\ar[ur]^-{\alpha}&& \\
					& \cM(A\otimes C^r_0(H)) \ar[r]^-{\alpha\otimes\id} &\cM(C^r_0(G)\otimes A\otimes C^r_0(H)).
				}
			\end{aligned}
		\end{align*}
	\end{df}
	
	\begin{eg}\label{eg:YDcond}
		For a homomorphism $\phi\colon H \to G$ of regular locally compact quantum groups, we call a left $\Ad W^\phi$-YD C*-algebra a left $\phi$-YD C*-algebra, whose notion has appeared in \cite{Nest-Voigt:eqpd,Roy:double}. 
	\end{eg}
	
	The following proposition is a straightforward generalization of \cite[Proposition~3.2]{Nest-Voigt:eqpd} and \cite[Proposition~7.6]{Roy:double}. 
	
	\begin{prop}\label{prop:YDcond}
		Let $\sm$ be a continuous matching on regular locally compact quantum groups $G$ and $H$. 
		For each C*-algebra $A$, we put 
		$X_A$ as the set of pairs $(\alpha,\lambda)$ with $(A,\alpha,\lambda)$ being a left $\sm$-YD C*-algebra, and 
		$Y_A$ as the set of continuous left $G^{\op}\dblcross{\sm}H$-actions $\wt{\alpha}$ on $A$. 
		
		\begin{enumerate}
			\item
			Then we have a well-defined mutually inverse bijective maps 
			\begin{align*}
				&
				\Phi_A \colon X_A\ni (\alpha,\lambda) \mapsto (\id\otimes\lambda)\alpha \in Y_A, 
				\\&
				\Psi_A \colon Y_A\ni \wt{\alpha} \mapsto (\imath_{G}^*\wt{\alpha}, \jmath_{H}^*\wt{\alpha}) \in X_A . 
			\end{align*}
			\item
			For left $\sm$-YD C*-algebras $B$ and $D$, a $*$-homomorphism $f\colon B\to \cM(D)$ 
			is a left $G^{\op}$- and left $H$-$*$-homomorphism if and only if 
			it is a left $G^{\op}\dblcross{\sm}H$-$*$-homomorphism for the corresponding $G^{\op}\dblcross{\sm}H$-actions. 
		\end{enumerate}
	\end{prop}
	
	\begin{proof}
		First we consider $\Phi_A$. 
		Let $(A,\alpha,\lambda)$ be a left $\sm$-YD C*-algebra. 
		We put $\wt{\alpha}=(\id\otimes\lambda)\alpha$, which is injective on $A$, and its comodule property follows from 
		\begin{align*}
			&
			\begin{aligned}
				&
				\bigl( \id\otimes ( (\id\otimes\sigma\lambda)\alpha) \otimes \id \bigr) 
				(\id\otimes\sigma\lambda) \alpha 
				=
				\sm_{25} \bigl( \id^{\otimes2}\otimes ((\sigma\lambda\otimes\id)\sigma\lambda) \bigr) 
				(\id\otimes\alpha) \alpha
				\\&
				=
				\sm_{25} (\Delta_G^{\cop}\otimes\id\otimes\Delta_H^{\cop}) (\id\otimes\sigma\lambda)\alpha. 
			\end{aligned}
		\end{align*}
		Also, $\wt{\alpha}$ satisfies 
		$[C^r_0(G)_1 C^r_0(H)_2 \wt{\alpha}(A)] = C^r_0(G)\otimes C^r_0(H)\otimes A$. 
		Thus $\Phi_A$ is well-defined. 
		Moreover, 
		\begin{align*}
			&
			(\Delta_{G}^{\cop}\otimes\id^{\otimes2}) (\id\otimes\lambda) \alpha 
			=
			(\id^{\otimes 2}\otimes\lambda) (\id\otimes\alpha) \alpha 
		\end{align*}
		and
		\begin{align*}
			\sm_{41} (\Delta_{H}\otimes\id^{\otimes2}) (\lambda\otimes\id) \sigma\alpha 
			&=
			\sm_{41} (\id\otimes\lambda\otimes\id) (\lambda\otimes\id) \sigma\alpha 
			=
			(\id\otimes\lambda\otimes\id) (\id\otimes\sigma\alpha) \lambda 
		\end{align*}
		show $\alpha=\imath_{G}^*\wt{\alpha}$ and $\lambda=\jmath_{H}^*\wt{\alpha}$, respectively. 
		Thus $\Psi_A\Phi_A = \id$ if $\Psi_A$ is well-defined. 
		
		We put $\alpha_0 := \imath_{G}^*\Delta_{G^{\op}\dblcross{\sm}H}$ and 
		$\lambda_0 := \jmath_{H}^*\Delta_{G^{\op}\dblcross{\sm}H}$. 
		The well-definedness of $\Psi_A$ and $\Phi_A\Psi_A=\id$ follows from 
		$\sm_{14}(\id\otimes\sigma\lambda_0)\alpha_0 = (\alpha_0\otimes\id)\sigma\lambda_0$ and 
		$(\id\otimes\lambda_0)\alpha_0 = \Delta_{G^{\op}\dblcross{\sm}H}$, respectively, 
		because for any left $G^{\op}\dblcross{\sm}H$-C*-algebra $(A,\wt{\alpha})$, 
		we can consider $(\wt{\alpha}(A), \Delta_{G^{\op}\dblcross{\sm}H}\otimes\id_A)$ 
		and reduce the situations to these relations of $(\alpha_0,\lambda_0)$. 
		
		Now (1) is shown and it is easy to see (2) by the constructions of $\Phi_A$ and $\Psi_A$ by using the fact that the equivariance of a $*$-homomorphism is compatible with restrictions of actions. 
	\end{proof}
	
	The notion of a left $\sm$-YD structure on a Hilbert module can be defined by considering the corner of a $\sm$-YD structure on the linking algebra. 
	Also, a left $\sm$-YD C*-correspondence is defined to be a pair of a left $\sm$-YD Hilbert module and a $G^{\op}\dblcross{\sm}H$-$*$-homomorphism. 
	We can show the same bijections as (1) of \cref{prop:YDcond} for Hilbert modules by considering the corner and for C*-correspondences with the aid of (2).
	
	\subsection{Action of bicrossed product on its component}\label{ssec:actbi}
	
	We investigate the unitary representation of $\wh{G}\rbicross{\sm}H$ on $L^2(G)$ denoted by $X^{\sm}$ below. 
	This unitary also induces an action of $\wh{G}\rbicross{\sm}H$ on $\wh{G}$ 
	and will play an essential role in the duality. 
	
	\begin{lem}\label{lem:bicrossell2}
		For a matching $\sm$ of locally compact quantum groups $G$ and $H$, the following hold. 
		\begin{enumerate}
			\item
			There is a well-defined unitary representation of $\wh{G}\rbicross{\sm}H$ on $L^2(G)$ defined by 
			\begin{align*}
				&
				X^{\sm} := U^G_3 Z^{\sm}_{12}U^{r*}_{32}Z^{\sm*}_{12} \wh{W}^{G}_{13} U^{G*}_3 \in \cU\cM(C^r_0(\wh{G}\rbicross{\sm}H)\otimes \cK(G)). 
			\end{align*}
			\item
			The restriction of the unitary representation $X^{\sm}$ via the homomorphism $\wh{\pi}_G\colon \wh{G}\to \wh{G}\rbicross{\sm}H$ is $\wh{V}^G_{21}$. 
			More precisely, 
			$((\Delta_{\rtimes}\rpb\wh{\pi}_G)\otimes\id)(X^{\sm}) = X^{\sm}_{124}\wh{V}^G_{43}$. 
			\item
			Up to a scalar multiplication, $X^{\sm}$ is equal to 
			$Z^{\sm}_{12} \wh{V}^{G}_{31} Z^{\sm*}_{32} Z^{\sm*}_{12}$. 
		\end{enumerate}
	\end{lem}
	
	\begin{proof}
		For (1), it suffices to show $X':=U^{G*}_3 X^{\sm}U^G_3$ is a unitary representation. 
		We can check this as 
		\begin{align*}
			&
			(\Delta_{\rtimes}\otimes\id)(X')
			=
			(\Delta_{\rtimes}\otimes\id)(Z^{\sm}_{12}U^{r*}_{32}Z^{\sm*}_{12} W^{G*}_{31})
			\\&
			=
			(\sm\otimes \sm\otimes\id)\bigl((\Delta_{H}\otimes\id)(U^{r*}_{21})\bigr)_{245}
			(\sm_{34}\sm_{54}W^{G*}_{51}) W^{G*}_{53}
			\\&
			=
			Z^{\sm}_{34}Z^{\sm}_{12} (U^{r*}_{52}U^{r*}_{54}) Z^{\sm*}_{12}Z^{\sm*}_{34}
			Z^{\sm}_{34}U^{r}_{54} (W^{G*}_{51}) U^{r*}_{54}Z^{\sm*}_{34} W^{G*}_{53}
			\\&
			=
			Z^{\sm}_{34}Z^{\sm}_{12} U^{r*}_{52}U^{r*}_{54} Z^{\sm*}_{12}
			U^{r}_{54} \wh{W}^{G}_{15} U^{r*}_{54}Z^{\sm*}_{34} \wh{W}^{G}_{35}
			\\&
			=
			Z^{\sm}_{12}U^{r*}_{52}Z^{\sm*}_{12} \wh{W}^{G}_{15} Z^{\sm}_{34}U^{r*}_{54}Z^{\sm*}_{34} \wh{W}^{G}_{35}
			\\&
			=
			X'_{125}X'_{345}. 
		\end{align*}
		
		(2) holds since by \cref{eq:rem:componentbidbl}, 
		\begin{align*}
			&
			((\Delta_{\rtimes}\rpb\wh{\pi}_G)\otimes\id)(X^{\sm}) 
			= \wh{V}^{G}_{13} X^{\sm}_{124} \wh{V}^{G*}_{13}
			= U^G_4 (\sm_{12}U^{r*}_{42}) \wh{V}^{G}_{13} \wh{W}^G_{14} \wh{V}^{G*}_{13} U^{G*}_4
			\\&
			= U^G_4 (\sm_{12}U^{r*}_{42}) \wh{W}^G_{14} U^{G*}_4 U^G_4 \wh{W}^G_{34} U^{G*}_4 
			= X^{\sm}_{124} \wh{V}^{G}_{43}. 
		\end{align*}
		
		(3) holds by using (5), (6) of \cref{thm:relations} to see for some $c_1, c_2\in\bC$ with $|c_1|=|c_2|=1$,  
		\begin{align*}
			&
			X^{\sm*} = 
			U^{G}_3 Z^{\sm}_{12} Z^{\sm*}_{12} \wh{W}^{G*}_{13} Z^{\sm}_{12} U^r_{32} Z^{\sm*}_{12} U^{G*}_3
			\\&
			=
			U^{G}_3 Z^{\sm}_{12} U^{l*}_{32} U^r_{32} W^{G}_{31} U^{r*}_{32} U^r_{32} Z^{\sm*}_{12} U^{G*}_3
			\\&
			=
			Z^{\sm}_{12} (J^{G}\wh{J}^{G})_3 \bigl((\wh{J}^G\otimes J^H) J^{\rtimes} \wh{J}^{\rtimes} (J^G\otimes \wh{J}^{H})\bigr)_{32} W^{G}_{31} U^{G*}_3 Z^{\sm*}_{12}
			\\&
			=
			c_1 Z^{\sm}_{12} U^{\rtimes}_{32} \bigl((J^G\otimes J^H) (J^G\otimes \wh{J}^{H})\bigr)_{32} W^{G}_{31} U^{G*}_3 Z^{\sm*}_{12}
			\\&
			=
			c_2 Z^{\sm}_{12} Z^{\sm}_{32} \wh{V}^{G*}_{31} Z^{\sm*}_{12} . 
		\qedhere\end{align*}
	\end{proof}
	
	\begin{prop}\label{prop:bicrosscr0}
		For a continuous matching $\sm$ of regular locally compact quantum groups $G$ and $H$, 
		there is a well-defined continuous right $\wh{G}\rbicross{\sm}H$-action $\delta_{\rtimes}$ 
		on $C^r_0(\wh{G})$ such that 
		for any $y\in C^r_0(\wh{G})$, 
		\begin{align}\label{eq:lem:bicrosscr01}
			&
			\Delta_{\rtimes}(y\otimes1) = \delta_{\rtimes}(y)_{134}
			\in \cM(C^r_0(\wh{G}) \otimes \bC\otimes C^r_0(\wh{G}\rbicross{\sm}H)), 
		\end{align}
		and 
		\begin{align}\label{eq:lem:bicrosscr02}
			\Ad X^{\sm*}(\wh{J}^Gy\wh{J}^G)_3 
			= (\wh{J}^{\rtimes}\otimes \wh{J}^G) \delta_{\rtimes}(y)_{312} (\wh{J}^{\rtimes}\otimes \wh{J}^G). 
		\end{align}
	\end{prop}
	
	Note that it follows $(\wh{J}^GC^r_0(\wh{G})\wh{J}^G,\, \Ad X^{\sm*}(-)_3)$ is a well-defined left $\wh{G}\rbicross{\sm}H$-C*-algebra. 
	
	\begin{proof}
		We note that $\Delta_{\rtimes}(y\otimes 1)\in L^\infty(\wh{G})\otimes1\otimes L^\infty(\wh{G}\rbicross{\sm}H)$ for any $y\in L^\infty(\wh{G})$ by the expression of $W^{\rtimes}$ 
		and get a normal $*$-homomorphism 
		$\delta_{\rtimes}\colon L^\infty(\wh{G})\to L^\infty(\wh{G})\barotimes L^\infty(\wh{G}\rbicross{\sm}H)$ with \cref{eq:lem:bicrosscr01}. 
		Clearly $(\delta_{\rtimes}\otimes\id^{\otimes2})\delta_{\rtimes}=(\id\otimes\Delta_{\rtimes})\delta_{\rtimes}$. 
		We check 
		\begin{align*}
			&
			\bigl[ (\cK(G)\otimes\cK(H))^*_{34} 
			\bigl( \Delta_{\rtimes} (C^r_0(\wh{G})\otimes1) \bigr) \bigr]
			\\& 
			= 
			\bigl[ (\cK(G)\otimes\cK(G)\otimes\cK(H))^*_{145} 
			\bigl( (\id\otimes \Delta_{\rtimes}) 
			( W^G_{12} ) \bigr) \bigr]
			\\&
			=
			\bigl[ (\cK(G)\otimes\cK(G)\otimes\cK(H))^*_{145} 
			\bigl( W^G_{14} Z^{\sm}_{45}Z^{\sm}_{15} W^{G}_{12} Z^{\sm*}_{15}Z^{\sm*}_{45} \bigr) \bigr]
			\\&
			=
			\bigl[\cK(G)^*_{1}(W^{G}_{12})\bigr]
			=C^r_0(\wh{G})\otimes1 , 
		\end{align*}
		and we see 
		$[C^r_0(\wh{G}\rbicross{\sm}H)_{23}\delta_{\rtimes}(C^r_0(\wh{G}))]=C^r_0(\wh{G})\otimes C^r_0(\wh{G}\rbicross{\sm}H)$ 
		by an argument using the regularity of $\wh{G}\rbicross{\sm}H$ and the comodule property, 
		such as (the proof of) \cite[Proposition~5.8]{Baaj-Skandalis-Vaes:nonsemireg}. 
		
		We calculate by using (1) of \cref{lem:bicrossell2} that 
		\begin{align*}
			&
			\Ad X^{\sm*} (\wh{J}^G y \wh{J}^G)_3
			\\&
			=
			\Ad (U^G_3 \wh{W}^{G*}_{13} Z^{\sm}_{12} U^{r}_{32}) 
			(J^G y J^G)_3
			\\&
			=
			\Ad (U^G_3 \wh{W}^{G*}_{13} Z^{\sm}_{12} U^{r}_{32} W^{H*}_{24} U^{r*}_{34}W^H_{24}W^{H*}_{24}U^{r}_{34} Z^{\sm*}_{12}) 
			(J^G y J^G)_3
			\\&
			=
			\Ad (U^G_3 \wh{W}^{G*}_{13} Z^{\sm}_{12} U^{r*}_{34} W^{H*}_{24} Z^{\sm*}_{12} U^r_{34}) 
			(J^G y J^G)_3
			\\&
			=
			\Ad ((J^G\wh{J}^G)_3 U^{r*}_{34} W^{\rtimes*}) 
			(\wh{J}^{\rtimes} (y\otimes1) \wh{J}^{\rtimes})_{34}
			\\&
			=
			\Ad ((\wh{J}^G\otimes\wh{J}^H)\wh{J}^{\rtimes})_{34} 
			\bigl((R^{\rtimes}\otimes R^{\rtimes})\Delta_{\rtimes}(y^*\otimes1)\bigr)_{3412}
			\\&
			=
			(\wh{J}^{\rtimes}\otimes \wh{J}^G\otimes\wh{J}^H) 
			\Delta_{\rtimes}(y\otimes1)_{3412}
			(\wh{J}^{\rtimes}\otimes \wh{J}^G\otimes\wh{J}^H) . 
		\qedhere\end{align*}
	\end{proof}
	
	\section{Duality procedures}\label{sec:procedure}
	
	\subsection{From double crossed product to bicrossed product}\label{ssec:YDtobicross}
	
	First we recall the procedure \cite[Proposition~6.1]{Baaj-Vaes:doublecrossed}, 
	which gives a $\wh{G}\rbicross{\sm}H$-action on the crossed product $G^{\op}\barltimes M$ for a von Neumann algebra $M$ with a $G^{\op}\dblcross{\sm}H$-action. 
	We would like to give a proof of its C*-algebraic version. 
	
	\begin{prop}[{cf.~\cite[Proposition~6.1]{Baaj-Vaes:doublecrossed}}]\label{prop:YDtobicrossC*}
		Let $\sm$ be a continuous matching on regular locally compact quantum groups $G$ and $H$, 
		and $(A,\alpha,\lambda)$ be a left $\sm$-YD C*-algebra. 
		Then $\bfJ^\sm_G(A):= (G^{\op}\mathbin{\underset{\alpha,r}{\ltimes}} A, \wt{\delta})$ is a well-defined left $\wh{G}\rbicross{\sm}H$-C*-algebra, 
		where $\wt{\delta}\colon G^{\op}\mathbin{\underset{\alpha,r}{\ltimes}} A\to \cM\bigl(C^r_0(\wh{G}\rbicross{\sm}H)\otimes (G^{\op}\mathbin{\underset{\alpha,r}{\ltimes}} A)\bigr)$ satisfies for all $a\in A$ and $y\in C^r_0(\wh{G})$, 
		\begin{align}\label{eq:prop:YDtobicrossC*}
			&
			\wt{\delta}(\alpha(a)) = (\sm\otimes\alpha)(1\otimes\lambda(a))
			\quad\text{and}\quad
			\wt{\delta}(\wh{J}^G y\wh{J}^G \otimes 1) 
			= \Ad X^{\sm*}_{123} (\wh{J}^G y \wh{J}^G)_{3}. 
		\end{align}
	\end{prop}
	
	Note that $\wt{\delta}$ acts on $\alpha(A)\cong A$ by $\pi_H^*\lambda$. 
	
	\begin{proof}
		We define $\wt{\delta}$ as the following composition of injective non-degenerate $*$-homomorphisms 
		\begin{align}\label{eq:prf:prop:YDtobicrossC*}
			&
			\begin{aligned}
				&
				\wt{\delta}\colon 
				G^{\op}\redltimes A \subset \cM\bigl( \cK(G)\otimes A \bigr) 
				\xrightarrow{\sm_{12}\sigma_{23}(1\otimes\id\otimes\lambda)} 
				\cM\bigl( \sm(1\otimes C^r_0(H))\otimes \cK(G) \otimes A \bigr) 
				\\&
				\xrightarrow{\Ad X^{\sm*}_{123}} 
				\cM\bigl( C^r_0(\wh{G}\rbicross{\sm}H) \otimes \cK(G) \otimes A \bigr). 
			\end{aligned}
		\end{align}
		Clearly this map satisfies \cref{eq:prop:YDtobicrossC*} for any $y\in C^r_0(\wh{G})$. 
		Also, for any $a\in A$ by using (3) of \cref{lem:bicrossell2} we have 
		\begin{align*}
			&
			\Ad X^{\sm*}_{123} \sm_{12} \sigma_{23} \bigl((\id\otimes\lambda)\alpha(a)\bigr)_{234} 
			\\&
			=
			\Ad (Z^{\sm}_{12}Z^{\sm}_{32}\wh{V}^{G*}_{31}Z^{\sm*}_{12}Z^{\sm}_{12}) 
			\sigma_{23} \bigl((\id\otimes\lambda)\alpha(a)\bigr)_{234} 
			\\&
			=
			\sm_{12}\sm_{32} \sigma_{23} \bigl((\id\otimes\lambda)\alpha(a)\bigr)_{234} 
			\\&
			=
			\sm_{12} \bigl((\id\otimes\alpha)\lambda(a)\bigr)_{234}. 
		\end{align*}
		
		By \cref{prop:bicrosscr0}, we see 
		\begin{align*}
			&
			\bigl[ C^r_0(\wh{G}\rbicross{\sm}H)_{12} \wt{\delta} (G^{\op}\redltimes A) \bigr]
			\\&
			=
			\bigl[ (C^r_0(\wh{G}\rbicross{\sm}H)\otimes \wh{J}^GC^r_0(\wh{G})\wh{J}^G)_{123} (\sm\otimes\alpha)(1\otimes\lambda)(A) \bigr]
			\\&
			= 
			C^r_0(\wh{G}\rbicross{\sm}H)\otimes (G^{\op}\redltimes A). 
		\end{align*}
		
		It holds 
		$(\Delta_{\rtimes}\otimes\id^{\otimes2})\wt{\delta}
		=(\id^{\otimes2}\otimes\wt{\delta})\wt{\delta}$ 
		on $(\wh{J}^GC^r_0(\wh{G})\wh{J}^G)_1$ again by \cref{prop:bicrosscr0} 
		and on $\alpha(A)$ 
		by the remark before the proof. 
	\end{proof}
	
	As expected, this procedure recovers the dual action of $\wh{G}$ as follows. 
	
	\begin{cor}\label{cor:YDtobicross}
		Let $\sm$ be a continuous matching on regular locally compact quantum groups $G$ and $H$. 
		For a left $\sm$-YD C*-algebra $(A,\alpha,\lambda)$ and 
		the continuous left $\wh{G}\rbicross{\sm}H$-action $\wt{\delta}$ on $G^{\op}\redltimes A$ given by \cref{prop:YDtobicrossC*}, it holds 
		$\wh{\pi}_G^*\wt{\delta}=\Ad \wh{V}^{G*}_{21}(-)_{23}$. 
	\end{cor}
	
	\begin{proof}
		Using \cref{eq:rem:componentbidbl}, (2) of \cref{lem:bicrossell2}, and \cref{prop:bicrosscr0}, we check that 
		\begin{align*}
			&
			((\Delta_{\rtimes}\rpb\wh{\pi}_G)\otimes\id^{\otimes2})\wt{\delta} 
			= \Ad (\wh{V}^{G}_{13} X^{\sm*}_{124}) \sm_{12}\sigma_{24} (1\otimes\id\otimes1\otimes\lambda) 
			\\&
			= \Ad (\wh{V}^{G*}_{43} X^{\sm*}_{124} \wh{V}^{G}_{13}) \sm_{12} \sigma_{24} (1\otimes\id\otimes1\otimes\lambda) 
			\\&
			= \Ad (\wh{V}^{G*}_{43} X^{\sm*}_{124}) \sm_{12} \sigma_{24} (1\otimes\id\otimes1\otimes\lambda) 
			= \Ad \wh{V}^{G*}_{43} (\wt{\delta}(-))_{1245} . 
		\qedhere\end{align*}
	\end{proof}

	\subsection{From bicrossed product to double crossed product}\label{ssec:bicrosstoYD}
	
	\begin{rem}
		To proceed, we need the following condition on a continuous matching $\sm$ on locally compact quantum groups $G$ and $H$, 
		\begin{align}\label{eq:strongcontimatching}
			&
			C^r_0(G)\otimes C^r_0(H)=[C^r_0(H)_2\sm(C^r_0(G)_1)] , 
		\end{align}
		which says the continuity of the left $H$-action $\sm_{21}(-)_2$ on $C^r_0(G)$. 
		In this article, we say such a continuous matching $\sm$ is \emph{continuous in the strict sense}. 
		Then the continuity of the left $G$-action $\sm(-)_2$ follows from (6) of \cref{thm:relations} as 
		\begin{align*}
			&
			C^r_0(G)\otimes C^r_0(H) 
			=
			\bigl[ \sm(R^G\otimes R^H) \bigl( C^r_0(H)_2 \sm(C^r_0(G)_1) \bigr) \bigr]
			\\&
			= 
			\bigl[ U^{\rtimes}(J^G\otimes J^H) C^r_0(H)_2 
			U^{\rtimes} (U^{G*}\otimes U^{H*}) C^r_0(G)_1 (U^{G}\otimes U^{H})U^{\rtimes*} 
			(J^G\otimes J^H) U^{\rtimes*} \bigr]
			\\&
			=
			\bigl[ U^{\rtimes} (J^G\otimes J^H) C^r_0(H)_2 (J^G\otimes J^H) 
			U^{\rtimes} (\wh{J}^{G}\otimes \wh{J}^{H}) C^r_0(G)_1 (\wh{J}^{G}\otimes \wh{J}^{H}) 
			U^{\rtimes*} U^{\rtimes*} \bigr]
			\\&
			=
			[ \sm(C^r_0(H)_2) C^r_0(G)_1 ]. 
		\end{align*}
	\end{rem}
	
	Next we give a structure of a left $\sm$-YD C*-algebra to the crossed product $\wh{G}\redltimes B$ for a left $\wh{G}\rbicross{\sm}H$-C*-algebra $B$. 
	Note that in particular this will give a left $H$-action on $\wh{G}\redltimes B$, but this $H$-action need not preserve $B\subset\cM(\wh{G}\redltimes B)$. 
	This situation makes the proof a bit more complicated. 
	
	\begin{prop}\label{prop:bicrosstoYDC*}
		Let $\sm$ be a matching on regular locally compact quantum groups $G$ and $H$ that is continuous in the strict sense. 
		Consider a left $\wh{G}\rbicross{\sm}H$-C*-algebra $(B,\wt{\beta})$ and put $\beta:=\wh{\pi}_{G}^*\wt{\beta}$. 
		Then $\wh{\bfJ}^{\sm}_G(B) := (\wh{G}\mathbin{\underset{\beta,r}{\ltimes}} B, \Ad V^G_{21}(-)_{23}, \gamma)$ is a well-defined left $\sm$-YD C*-algebra, where 
		$\gamma\colon \wh{G}\mathbin{\underset{\beta,r}{\ltimes}} B\to \cM\bigl(C^r_0(H)\otimes (\wh{G}\mathbin{\underset{\beta,r}{\ltimes}} B)\bigr)$ is a continuous left $H$-action such that for any $x\in C^r_0(G)$ and $b\in B$, 
		\begin{align}\label{eq:prop:bicrosstoYDC*}
			&
			\gamma(x\otimes 1) = \sigma\sm(x\otimes1) \quad\text{and}\quad 
			\gamma(\beta(b)) =  \wt{\beta}(b)_{213} . 
		\end{align} 
	\end{prop}
	
	In the proof, we use the fact that $\wh{\Delta}_\rtimes$ on $\bigl( \sm(L^\infty(G)_1) L^\infty(\wh{H})_2 \bigr)''$ satisfies 
	\begin{align*}
		&
		\wh{\Delta}_{\rtimes}\sm(-)_1 = (\sm\otimes\sm)\Delta_G(-)_{13} 
		\text{\ \ on\ \ }L^\infty(G), 
		\\&
		(\id\otimes\wh{\Delta}_{\rtimes})(W^H_{13}) = W^H_{15} (\sm_{45}\sm_{41}W^H_{13}), 
	\end{align*}
	which can be checked by using the isomorphism of  
	$\wh{H}\rbicross{\sigma\sm\sigma}G$ and the dual of $\wh{G}\rbicross{\sm}H$ (see the remark after \cref{thm:relations}). 
	\begin{proof}
		Suppose $\gamma$ is a well-defined continuous left $H$-action. Then we can check $(\wh{G}\redltimes B, \Ad V^G_{21}(-)_{23}, \gamma)$ is a left $\sm$-YD C*-algebra. 
		Indeed, the commutativity of the diagram \cref{eq:prop:YDcond} on $C^r_0(G)_1$ is easy and for any $b\in B$ it holds 
		\begin{align*}
			&
			\Ad V^G_{24} (\gamma\beta(b))_{123}
			=
			\Ad V^G_{24} (\wt{\beta}(b))_{213}
			\\&
			=
			\sm_{41} (\wt{\beta}(b))_{213}
			=
			\sm_{41} (\gamma\otimes\id) (\Ad V^G_{13}) \beta(b)_{12}, 
		\end{align*}
		where we have used 
		$\Ad V^{G}_{13}(-)_{12}=\sm_{32}(-)_{12}$ on $C^r_0(\wh{G}\rbicross{\sm}H)$, 
		since they are identity maps on $C^r_0(\wh{G})_1$ and for $x\in C^r_0(H)$ it holds 
		$\Ad V^G_{13}\sm_{12}(x_2) = \sm_{32}\sm_{12}(x_2)$. 
		
		Therefore it suffices to show $\gamma$ is a well-defined continuous left $H$-action. 
		We consider the following composition 
		\begin{align*}
			\begin{aligned}
				&
				\wt{\gamma}\colon 
				\wh{G}\redltimes B 
				\xrightarrow{\id\otimes\wt{\beta}} 
				\Bigl[ C^r_0(G)_1 \bigl( ((\wh{\pi}_{G}^*\Delta_{\rtimes})\otimes\id) \wt{\beta}(B) \bigr) \Bigr]
				\\&
				\xrightarrow{\Ad \sm_{23}(\wh{W}^G_{12})}
				\Bigl[ \bigl(\sm_{23}(\Delta_G(C^r_0(G))_{21})\bigr) \wt{\beta}(B)_{234} \Bigr]
				\\&
				\xrightarrow{\sm_{15}(-)_{1234}}
				\Bigl[ \bigl((\sm\otimes\sm)(\Delta_G(C^r_0(G))_{13})\bigr)_{2315} \wt{\beta}(B)_{234} \Bigr]
				\\&
				\xrightarrow{\Ad \wh{V}^{\rtimes *}_{2315}}
				\bigl[ \sm_{23}(C^r_0(G)_{2}) \wt{\beta}(B)_{234} \bigr]
				\\&
				\subset 
				\cM(\cK(G)\otimes\cK(G)\otimes\cK(H)\otimes B\otimes\cK(H)) , 
			\end{aligned}
		\end{align*}
		and we check $(\omega\otimes\sigma\otimes\id\otimes \wt{\omega})\wt{\gamma}$ satisfies \cref{eq:prop:bicrosstoYDC*} for states $\omega$ on $\cK(G)$ and $\wt{\omega}$ on $\cK(H)$. 
		Thus $\gamma$ is an injective $*$-homomorphism $\wh{G}\redltimes B\to \cM(\cK(H)\otimes\cK(G)\otimes B)$. 
		
		Using the relation 
		\begin{align*}
			\begin{aligned}
				&
				V^{\rtimes} V^{\wh{\pi}_G*}_{123} 
				= 
				U^{\rtimes}_{12} W^{\rtimes}_{3412} U^{\rtimes*}_{12} \wh{V}^{G*}_{13} 
				\\&
				= 
				U^{\rtimes}_{12} (\sm_{34}W^H_{42}) (Z^{\sm}_{12}W^{G*}_{13}Z^{\sm*}_{12}) U^{\rtimes *}_{12} U^{G}_{1} W^{G}_{13} U^{G*}_{1}
				\\&
				=
				U^{\rtimes}_{12} (\sm_{34}W^H_{42}) U^{\rtimes*}_{12} U^{G}_1 U^{H*}_2 W^{G*}_{13} U^{H}_{2} W^{G}_{13} U^{G*}_{1}
				\\&
				=
				U^{\rtimes}_{12} (\sm_{34}W^H_{42}) U^{\rtimes *}_{12} 
			\end{aligned}
		\end{align*}
		and \cref{eq:strongcontimatching}, we see 
		\begin{align*}
			&
			\cK(G)\otimes\cK(H)\otimes 
			\bigl[ C^r_0(H)_2 \sm_{12}(C^r_0(G)_1)\wt{\beta}(B) \bigr] 
			\\&
			=
			\Bigl[ (\cK(G)\otimes\cK(H)\otimes C^r_0(G)\otimes C^r_0(H))_{1234} V^{\rtimes}_{1234}\wt{\beta}(B)_{125}V^{\rtimes *}_{1234} \Bigr] 
			\\&
			=
			\Bigl[ (\cK(G)\otimes\cK(H)\otimes C^r_0(G)\otimes C^r_0(H))_{1234} 
			U^{\rtimes}_{12} (\sm_{34}W^H_{42}) U^{\rtimes *}_{12} \bigl((\id\otimes\beta)\wt{\beta}(B)_{1235}\bigr) V^{\wh{\pi}_G}_{123} V^{\rtimes *}_{1234} \Bigr] 
			\\&
			=
			\Bigl[ (\wh{G}\redltimes B)_{35} (\cK(G)\otimes\cK(H)\otimes C^r_0(G)\otimes C^r_0(H))_{1234} 
			U^{\rtimes}_{12} (\sm_{34}W^{H*}_{42}) U^{\rtimes *}_{12} \Bigr] 
			\\&
			= 
			\cK(G)\otimes \cK(H)\otimes \bigl[ (\wh{G}\redltimes B)_{13} C^r_0(H)_2 \bigr], 
		\end{align*}
		and thus $[C^r_0(H)_1 \gamma(\wh{G}\redltimes B)]=C^r_0(H)\otimes (\wh{G}\redltimes B)$. 
		
		Now it remains to show $(\id\otimes\gamma)\gamma=(\Delta_H\otimes\id)\gamma$. 
		Clearly this holds on $C^r_0(G)_1$. 
		Also, for any $b\in B$, by applying $\wt{\gamma}$ to $\wt{\beta}(b)$ we get 
		\begin{align*}
			&
			\Ad \bigl(\wh{V}^{\rtimes*}_{3416}Z^{\sm}_{16}(\sm_{34}\wh{W}^{G}_{13})\bigr) (\id\otimes\wt{\beta})\wt{\beta}(b) 
			\\&
			= 
			\Ad \bigl( \wh{V}^{\rtimes*}_{3416} Z^{\sm}_{16} (\sm_{12}W^{H*}_{24}) W^{\rtimes}_{1234} \bigr) (\Delta_{\rtimes}\otimes\id)\wt{\beta}(b) 
			\\&
			=
			\Ad \bigl(\wh{V}^{\rtimes*}_{3416} Z^{\sm}_{16} (\sm_{12}W^{H*}_{24}) Z^{\sm*}_{16} W^{H*}_{26} \bigr)  (\wt{\beta}(b)_{345}) 
			\\&
			=
			\Ad \bigl(\wh{V}^{\rtimes*}_{3416} ((\id\otimes\wh{\Delta}_{\rtimes})(W^{H*}_{13}))_{23416} \bigr)  (\wt{\beta}(b)_{345}) 
			\\&
			= \Ad W^{H*}_{24} \wt{\beta}(b)_{345}. 
		\qedhere\end{align*}
	\end{proof}
	
	\subsection{Duality result}\label{ssec:duality}
	
	\begin{thm}\label{thm:BSTTdualC*}
		Let $\sm$ be a matching on regular locally compact quantum groups $G$ and $H$ that is continuous in the strict sense. 
		Consider a left $\sm$-YD C*-algebra $(A,\alpha,\lambda)$ and a left $\wh{G}\rbicross{\sm}H$-C*-algebra $(B,\wt{\beta})$, 
		and put $\beta:=\wh{\pi}_G^*\wt{\beta}$. 
		Then the following hold. 
		\begin{enumerate}
			\item
			We have a well-defined left $G^{\op}$- and left $H$-$*$-isomorphism
			\begin{align*}
				&
				\eta_A := (\id\otimes \alpha)^{-1} 
				\Ad ((J^G\otimes\wh{J}^G) W^{G*} (J^G\otimes\wh{J}^G))_{12} \colon  
				\wh{\bfJ}^{\sm}_G\bfJ^{\sm}_G(A) 
				\to \cK_A(L^2(G)\tensor{\bC} A), 
			\end{align*}
			where we have equipped $L^2(G)$ with the left $G^{\op}\dblcross{\sm}H$-action by $U^r_{32}V^G_{31}$. 
			\item
			We have a left $\wh{G}\rbicross{\sm}H$-$*$-isomorphism 
			\begin{align*}
				&
				\wh{\eta}_B := 
				(\id\otimes\beta)^{-1} \Ad W^G_{21} \colon 
				\bfJ^{\sm}_G\wh{\bfJ}^{\sm}_G(B)
				\to \cK_B(L^2(G)\tensor{\bC}B) , 
			\end{align*}
			where we have equipped $L^2(G)$ with the left $\wh{G}\rbicross{\sm}H$-action by $X^{\sm*}$. 
		\end{enumerate}
	\end{thm}
	
	In the statement, we equip the right Hilbert $A$-module $L^2(G)\tensor{\bC}A$ with the continuous left actions given by the tensor product of $L^2(G)$ and $A$ regarded as an equivariant $(\bC,A)$-correspondence, 
	and similarly for $L^2(G)\tensor{\bC}B$. 
	Also, we recall that thanks to (7) of \cref{thm:relations}, the triple 
	$(L^2(G), V^G_{21}, U^r_{21})$ is a well-defined left $\sm$-YD Hilbert space (regarded as a right Hilbert $\bC$-module). 
	Due to \cite[Remark~12.5]{Vaes:impr}, these actions on $L^2(G)\tensor{\bC}A$ and $L^2(G)\tensor{\bC}B$ are continuous and give equivariant Morita equivalences. 
	
	\begin{proof}
		Thanks to Baaj--Skandalis duality \cite{Baaj-Skandalis:mltu}, we know $\eta_A$ is a well-defined left $G^{\op}$-$*$-isomorphism and $\wh{\eta}_B$ is a well-defined left $\wh{G}$-$*$-isomorphism with the restricted actions by \cref{cor:YDtobicross}. 
		
		We show (1). It suffices to show $\eta_A$ is a left $H$-$*$-homomorphism. 
		We write $\gamma$ for the left $H$-action on 
		\begin{align*}
			&
			\wh{\bfJ}^{\sm}_G\bfJ^{\sm}_G(A) 
			= \bigl[ C^r_0(G)_1 U^G_2\wh{\Delta}_{G}(C^r_0(\wh{G}))_{12}U^{G*}_2 \alpha(A)_{23} \bigr] . 
		\end{align*}
		Then we check for $x\in C^r_0(G)$, 
		\begin{align*}
			&
			(\id\otimes\eta_A)\gamma\eta_A^{-1}(x_1) 
			=
			(\id\otimes\eta_A)\gamma(x_1) 
			\\&
			=
			(\id\otimes\eta_A)(\Ad U^{r}_{21}) (x_2)
			=
			\Ad U^{r}_{21} (x_2), 
		\end{align*}
		for $y\in C^r_0(\wh{G})$ by using \cref{prop:bicrosscr0}, 
		\begin{align*}
			&
			(\id\otimes\eta_A)\gamma\eta_A^{-1}(y_1) 
			=
			(\id\otimes\eta_A)\gamma(U_2^G \wh{\Delta}_G(y)_{12} U_2^{G*})
			\\&
			=
			(\id\otimes\eta_A)(\Ad X^{\sm*}_{213})(U^GyU^{G*})_{3} 
			\\&
			=
			(\id\otimes\eta_A)
			\bigl((\wh{J}^{\rtimes}\otimes \wh{J}^G) \delta_{\rtimes}(J^G yJ^G)_{312} (\wh{J}^{\rtimes}\otimes \wh{J}^G)\bigr)_{213}
			\\&
			=
			(\id\otimes\eta_A) 
			\Ad\bigl( (\wh{J}^{\rtimes}\otimes \wh{J}^G)
			U^{r}_{12} W^{G}_{13} U^{r*}_{12} (J^G\otimes \wh{J}^H\otimes \wh{J}^G)\bigr)_{213} (y_2)
			\\&
			=
			(\id\otimes\eta_A) 
			\Ad\bigl( (J^G\otimes \wh{J}^H\otimes \wh{J}^G) W^{G}_{13} (J^G\otimes \wh{J}^H\otimes \wh{J}^G) U^r_{12} \bigr)_{213} 
			(y_2)
			\\&
			= \Ad U^r_{21} (y_2), 
		\end{align*}
		and for $a\in A$, 
		\begin{align*}
			&
			(\id\otimes\eta_A) \gamma\eta_A^{-1}(U^G_{1}\alpha(a)U^{G*}_1) 
			=
			(\id\otimes\eta_A) \gamma(\alpha(a)_{23}) 
			\\&
			=
			(\id^{\otimes2}\otimes\alpha)^{-1} 
			\Ad\bigl( U^G_2 V^{G}_{32} U^{G*}_2 U^{l}_{21} \bigr) 
			\bigl((\id\otimes\alpha)\lambda(a)\bigr)_{134}
			\\&
			=
			(\id^{\otimes2}\otimes\alpha)^{-1} 
			\Ad\bigl( U^{l}_{21} U^G_2 V^{G}_{32} U^{G*}_2 \bigr) 
			\bigl((\id\otimes\alpha)\lambda(a)\bigr)_{134}
			\\&
			=
			(\id^{\otimes2}\otimes\alpha)^{-1} 
			\Ad\bigl( J^{\rtimes}_{21} (J^G\otimes J^H)_{21} \bigr) 
			\bigl((\id^{\otimes2}\otimes\alpha)(\id\otimes\alpha)\lambda(a)\bigr)
			\\&
			=
			\Ad\bigl( J^{\rtimes}_{21} (J^G\otimes J^H)_{21} Z^{\sm}_{21} \bigr) 
			\bigl((\id\otimes\lambda)\alpha(a)\bigr)_{213}
			\\&
			=
			\Ad U^r_{21} 
			\bigl( (\id\otimes\lambda) (U^G_1 \alpha(a) U^{G*}_1) \bigr)_{213} , 
		\end{align*}
		where the last equality holds since by (6) of \cref{thm:relations}, 
		up to scalar multiplications we have 
		\begin{align*}
			&
			J^{\rtimes} (J^G\otimes J^H) Z^{\sm}
			=
			\wh{J}^{\rtimes} (\wh{J}^G\otimes \wh{J}^H)
			=
			U^r U^G_1 . 
		\end{align*}
		
		We show (2). It suffices to show $\wh{\eta}_B$ is a left $\wh{G}\rbicross{\sm}H$-$*$-homomorphism. 
		We write $\wt{\delta}$ for the left $\wh{G}\rbicross{\sm}H$-action on  
		\begin{align*}
			&
			\bfJ^{\sm}_G\wh{\bfJ}^{\sm}_G (B) 
			= 
			\bigl[ (\wh{J}^G C^r_0(\wh{G}) \wh{J}^G)_1 \Delta_G(C^r_0(G))_{21} \beta(B)_{23} \bigr]. 
		\end{align*}
		Then we check for $y\in C^r_0(\wh{G})$, 
		\begin{align*}
			&
			(\id^{\otimes2}\otimes\wh{\eta}_B) \wt{\delta} \wh{\eta}_B^{-1}(\wh{J}^Gy\wh{J}^G)_1
			=
			(\id^{\otimes2}\otimes\wh{\eta}_B) \wt{\delta} (\wh{J}^G_1 y_1 \wh{J}^G_1)
			\\&
			=
			(\id^{\otimes2}\otimes\wh{\eta}_B) 
			\Ad X^{\sm*}_{123} (\wh{J}^G y \wh{J}^G)_{3} 
			=
			\Ad X^{\sm*}_{123} (\wh{J}^G y \wh{J}^G)_{3} , 
		\end{align*}
		for $x\in C^r_0(G)$ by using (3) of \cref{lem:bicrossell2}, 
		\begin{align*}
			&
			(\id^{\otimes2}\otimes\wh{\eta}_B) \wt{\delta} \wh{\eta}_B^{-1}(x_1)
			=
			(\id^{\otimes2}\otimes\wh{\eta}_B) \wt{\delta} (\Delta_G(x)_{21})
			\\&
			=
			(\id^{\otimes2}\otimes\wh{\eta}_B) (\Ad W^{G*}_{43}) \sm_{12} \sm_{32} (x_3)
			\\&
			=
			(\id^{\otimes2}\otimes\wh{\eta}_B) \Ad\bigl( W^{G*}_{43} Z^{\sm}_{12} Z^{\sm}_{32} \wh{V}^{G*}_{31} Z^{\sm*}_{12} \bigr) (x_3)
			\\&
			= \Ad X^{\sm*}_{123}(x_3) , 
		\end{align*}
		and for $b\in B$, 
		\begin{align*}
			&
			(\id^{\otimes2}\otimes\wh{\eta}_B) \wt{\delta} \wh{\eta}_B^{-1}\beta(b)
			=
			(\id^{\otimes2}\otimes\wh{\eta}_B) \wt{\delta} (\beta(b)_{23})
			\\&
			=
			(\id^{\otimes2}\otimes\wh{\eta}_B) \sm_{12} \Ad V^G_{43} (\wt{\beta}(b)_{425})
			\\&
			=
			(\id^{\otimes3}\otimes\beta)^{-1} 
			\Ad (W^G_{43}Z^{\sm}_{12}V^G_{43}) (\wt{\beta}(b)_{425})
			\\&
			\overset{\blacklozenge}{=}
			(\id^{\otimes3}\otimes\beta)^{-1} 
			\Ad (Z^{\sm}_{12}V^{\wh{\pi}_G}_{324}) (\wt{\beta}(b)_{325})
			\\&
			=
			(\id^{\otimes3}\otimes\beta)^{-1} 
			\Ad (Z^{\sm}_{12}) \bigl(((\Delta_{\rtimes}\rpb\wh{\pi}_G)\otimes\id)\wt{\beta}(b)\bigr)_{3245}
			\\&
			=
			\Ad Z^{\sm}_{12} \wt{\beta}(b)_{324} 
			=
			\Ad (U^{\rtimes}_{12}U^{H*}_{2}) \wt{\beta}(b)_{324} 
			\\&
			=
			\Ad \bigl( U^{\rtimes}_{12}U^{H*}_{2} U^{\rtimes}_{32} W^{G*}_{13} U^{\rtimes*}_{32} \bigr) \wt{\beta}(b)_{324} 
			\\&
			\overset{\bigstar}{=}
			\Ad (Z^{\sm}_{12}Z^{\sm}_{32}\wh{V}^{G*}_{31}Z^{\sm*}_{12} Z^{\sm}_{12}\wh{W}^{G*}_{31}Z^{\sm*}_{12}) \wt{\beta}(b)_{124} 
			\\&
			=
			\Ad (Z^{\sm}_{12}Z^{\sm}_{32}\wh{V}^{G*}_{31}Z^{\sm*}_{12}) \bigl(((\wh{\pi}_G^*\Delta_{\rtimes})\otimes\id)\wt{\beta}(b)\bigr)_{3124} 
			\\&
			=
			\Ad X^{\sm*}_{123} ((\id\otimes\wt{\beta})\beta(b))_{3124} . 
		\end{align*}
		Here 
		for the last equality we have used (3) of \cref{lem:bicrossell2}, 
		and the equality $\overset{\blacklozenge}{=}$ holds by 
		\cref{eq:rem:componentbidbl} and 
		$W^GV^G=\Sigma\wh{V}^G U^G_2$ up to a scalar, 
		where $\Sigma\in\cU(L^2(G)^{\otimes2})$ is the flipping map (see \cite[(2.2)]{Baaj-Skandalis-Vaes:nonsemireg}, for example). 
		Also, the equality $\overset{\bigstar}{=}$ holds because up to scalars we have  
		\begin{align*}
			&
			U^{\rtimes}_{12}U^{H*}_{2} 
			U^{\rtimes}_{32} W^{G*}_{13} U^{\rtimes*}_{32}
			\\&
			=
			U^{\rtimes}_{12}U^{H*}_{2} 
			U^{\rtimes}_{32} U^{H*}_2 U^{G*}_3 V^{G*}_{31} U^{G}_3 U^{H}_2 U^{\rtimes*}_{32}
			\\&
			=
			U^{\rtimes}_{12}U^{H*}_{2} Z^{\sm}_{32} (U^{G*}_2 \wh{V}^{G*} \wh{W}^{G*} \Sigma)_{31} Z^{\sm*}_{32}
			\\&
			=
			Z^{\sm}_{12}Z^{\sm}_{32}\wh{V}^{G*}_{31}Z^{\sm*}_{12} Z^{\sm}_{12}\wh{W}^{G*}_{31}Z^{\sm*}_{12} \Sigma_{13} . 
		\qedhere\end{align*}
	\end{proof}
	
	\begin{rem}\label{rem:dualityopcop}
		Note that for a locally compact quantum group $F$ and a left $F$-C*-algebra $(A,\alpha)$, 
		the opposite C*-algebra $A^{\op}$ has the canonical continuous left $F^{\op}$-action $(R^F\otimes\id)\alpha$. 
		Thus \cref{thm:BSTTdualC*} gives the correspondence between continuous right actions of $G^{\op}\dblcross{\sm}H$ and $\wh{G}\rbicross{\sm}H$. 
		Via the isomorphisms of locally compact quantum groups, we have a similar duality between left $G^{\op}\dblcross{\sm}H$-YD C*-algebras and right $\wh{H}\rbicross{\sigma\sm\sigma}G$-C*-algebras. 
	\end{rem}
	
	\begin{rem}\label{rem:BSTTdualbimod}
		In the situation of \cref{thm:BSTTdualC*}, consider left $\sm$-YD C*-algebras $(A_i,\alpha_i,\lambda_i)$ for $i=1,2$. 
		\begin{enumerate}
			\item
			For a left $G^{\op}\dblcross{\sm}H$-$*$-homomorphism $f\colon A_1\to A_2$, 
			we have a well-defined $*$-homomorphism 
			\begin{align*}
				\bfJ^{\sm}_{G}(f):=(G^{\op}\redltimes f) 
				\colon G^{\op}\redltimes A_1 \to& G^{\op}\redltimes A_2, \\
				\alpha_1(a) \mapsto& \alpha_2(f(a)) \\
				\wh{J}^Gy\wh{J}^G\otimes 1 \mapsto& \wh{J}^Gy\wh{J}^G\otimes 1 
			\end{align*}
			where $a\in A_1$ and $y\in C^r_0(\wh{G})$. 
			It is left $\wh{G}\rbicross{\sm}H$-equivariant by definition of $\bfJ^{\sm}_G$. 
			\item
			For a left $G^{\op}\dblcross{\sm}H$-$(A_1,A_2)$-imprimitivity bimodule $\cE$, 
			we have a well-defined left $\wh{G}\rbicross{\sm}H$-$(\bfJ^{\sm}_G(A_1), 
			\bfJ^{\sm}_G(A_2))$-imprimitivity bimodule $\bfJ^{\sm}_G(\cE)$ as the corner of $\bfJ^{\sm}_G(\cK_{A_2}(\cE\oplus A_2))$. 
			By combining (1), for a proper left $G\dblcross{\sm}H$-$(A_1,A_2)$-correspondence $\cE$, 
			we get a proper left $\wh{G}\rbicross{\sm}H$-$(\bfJ^{\sm}_G(A_1), 
			\bfJ^{\sm}_G(A_2))$-correspondence, denoted by $\bfJ^{\sm}_G(\cE)$. 
			\item
			Similarly, $\wh{\bfJ}^{\sm}_G$ produces a 
			$G^{\op}\dblcross{\sm}H$-$*$-homomorphism, proper correspondence, imprimitivity bimodule from $\wh{G}\rbicross{\sm}H$-equivariant ones. 
		\end{enumerate}
	\end{rem}

	\section{Generalized quantum doubles and monoidal structures}\label{sec:qdmon}
	
	In this section, we consider the particular case of \cref{eg:matchingYD}. 
	Let $\phi\colon H\to G$ be a homomorphism of locally compact quantum groups. 
	If $G$ and $H$ are regular the matching $\Ad W^{\phi}$ is continuous in the strict sense. 
	We note from \cite[Proposition~8.1]{Baaj-Vaes:doublecrossed} that in this case one has the isomorphism 
	$\wh{H^{\op}}\rbicross{\Ad W^{\phi}}\wh{G} \cong \wh{H^{\op}}\times \wh{G}$ given by 
	\begin{align}\label{eq:isomgenqd}
		&
		\Ad W^{\phi*} \colon (\wh{J}^H C^r_0(\wh{H}) \wh{J}^H)_1 W^{\phi} C^r_0(\wh{G})_2 W^{\phi*} \to (\wh{J}^H C^r_0(\wh{H}) \wh{J}^H) \otimes C^r_0(\wh{G}) . 
	\end{align}
	
	\subsection{Duality procedures as twisted tensor products}\label{ssec:twitendual}
	We recall the notion of twisted tensor products. 
	We refer to 
	\cite{Nest-Voigt:eqpd,Meyer-Roy-Woronowicz:twiten,Meyer-Roy-Woronowicz:twiten2,Kitamura:indhomlcqg}, 
	but we use a different convention for the legs from these references 
	because it would be useful for iterations of twisted tensor products. 
	\begin{df}\label{def:twiten}
		Let $\phi\colon H\to G$, $\psi\colon G\to F$ be homomorphisms of locally compact quantum groups. 
		For a left $\phi$-YD C*-algebra $(A,\alpha,\lambda)$ and a left $\psi$-YD C*-algebra $(B,\beta,\mu)$, 
		we define the \emph{twisted tensor product} 
		\begin{align*}
			A\outensor{G}B := [\lambda(A)_{21} \beta(B)_{23}] \subset \cM(A\otimes \cK(G)\otimes B) , 
		\end{align*}
		which is a well-defined left $\psi\phi$-YD C*-algebra with the left $H$-action $\Ad W^{\phi*}_{13} (\alpha\otimes\id^{\otimes2})$ 
		and the left $\wh{F}$-action $\sigma_{12}\sigma_{23}\Ad W^{\psi}_{23}(\id^{\otimes2}\otimes \mu)$. 
	\end{df}
	
	We also define twisted tensor products of YD C*-correspondences by taking the corner of the twisted tensor products of the linking algebras. 
	Twisted tensor products of YD C*-correspondences are compatible with inner tensor products by \cite[Lemma~5.12]{Meyer-Roy-Woronowicz:twiten} and its proof. 
	
	We note that restrictions of actions and dual actions are given by twisted tensor products of YD C*-algebras. 
	For a homomorphism $\phi\colon H\to G$ of regular locally compact quantum groups, 
	we write $\bC_\phi$ for $\bC$ with the trivial left action of $\sD(\phi)$, and unless stated otherwise, we regard $C^r_0(G)$ as a left $1_{G\to \wh{G^{\op}}}$-YD C*-algebra by $(C^r_0(G), \Delta_G, \Delta_G^{\cop})$. 
	Then for a left $G$-C*-algebra $(A,\alpha)$ regarded as a left $1_{G\to\{1\}}$-YD C*-algebra (where $\{1\}$ is the trivial group), 
	we have a left $H$-$*$-isomorphism $\alpha\colon \phi^*A\cong \bC_\phi\outensor{G}A$.
	Also, by regarding $A$ as a left $1_{\{1\}\to\wh{G}}$-YD C*-algebra, we have a left $\wh{G}^{\op}$-$*$-isomorphism 
	$\Ad \wh{V}^G_{23}(-)_{21}\colon G\redltimes A\xrightarrow{\sim} A\outensor{\wh{G}}C^r_0(\wh{G})$. 
	Next we see the action of $\bfJ^{\Ad W^{\phi}}_{H^{\op}}$ is expressed as a twisted tensor product by regarding $\wh{H^{\op}}\times \wh{G} \cong \sD(1_{\wh{G}\to H^{\op}})$. 
	
	\begin{prop}\label{prop:dualtwiten}
		For a homomorphism $\phi\colon H\to G$ of regular locally compact quantum groups and 
		a left $\phi$-YD C*-algebra $(A,\alpha,\lambda)$, it holds 
		\begin{align*}
			\Ad \wh{V}^H_{23}(-)_{21} \colon 
			\bfJ^{\Ad W^{\phi}}_{H^{\op}}(A) \cong A\outensor{\wh{H}}C^r_0(\wh{H}), 
		\end{align*}
		as left $\wh{H^{\op}}\times \wh{G}$-C*-algebras. 
		Here in the right hand side, we regarded $A$ as a left $\wh{\phi}$-YD C*-algebra $(A,\lambda,\alpha)$, and $C^r_0(\wh{H})$ as a left $1_{\wh{H}\to H^{\op}}$-YD C*-algebra. 
	\end{prop}
	
	\begin{proof}
		It suffices to show $\Ad \wh{V}^H_{23}(-)_{21}$ preserves the $\wh{H^{\op}}\times \wh{G}$-actions via the isomorphism of \cref{eq:isomgenqd}. 
		We check this separately on $A$ and $C^r_0(\wh{H})$ with the aid of \cref{eq:prop:YDtobicrossC*}. 
		Since the left $\wh{H^{\op}}\rbicross{\sm}\wh{G}$-action on $\alpha(A)\subset \cM(\bfJ^{\Ad W^{\phi}}_{H^{\op}}(A))$ is the restriction of left $\wh{G}$-action $\lambda$ on $A\cong \alpha(A)$, 
		the corresponding $\wh{H^{\op}}\times\wh{G}$-action is trivial for $\wh{H^{\op}}$ and given by $\lambda$ for $\wh{G}$. 
		
		Also, by \cite[Proposition~8.1]{Baaj-Vaes:doublecrossed}, we have 
		$\wh{J}^{\sm}=W^{\phi}(J^H\otimes J^G)W^{\phi*}$ 
		%and $Z^{\sm}=W^{\phi}(J^{H}\otimes \wh{J}^G)W^{\phi}(J^{H}\otimes \wh{J}^G)$ up to a scalar, 
		and by \cref{lem:bicrossell2}~(1) it holds 
		\begin{align*}
			&
			X^{\Ad W^{\phi}} 
			\\&
			= 
			U^H_3 
			W^{\phi}_{12} 
			\bigl(W^{\phi}(J^{H}\otimes J^G)W^{\phi*}(J^{H}\otimes J^G)\bigr)^*_{32} 
			W^{\phi*}_{12} 
			V^H_{13} U^{H*}_3 
			\\&
			= 
			U^H_3 W^{\phi}_{12}
			\wh{V}^{\phi}_{32} W^{\phi*}_{32} 
			U^H_1 W^{\phi}_{32}W^H_{31} U^{H*}_1 
			W^{\phi*}_{12} U^{H*}_3 
			\\&
			= 
			W^{\phi}_{12} 
			\wh{W}^{\phi}_{23} U^H_3 V^H_{13} U^{H*}_3 
			W^{\phi*}_{12} . 
		\end{align*}
		Thus $\Ad (W^{\phi*}_{12}(X^{\Ad W^{\phi}})^*)(-)_3=\Ad (U^H_3V^{H*}_{13}U^{H*}_3\wh{W}^{\phi*}_{23})(-)_3$ on $C^r_0(\wh{H})$.
		The claim follows from these expressions. 
	\end{proof}
	
	Take a homomorphism $\phi\colon H\to G$ of regular locally compact quantum groups and a left $\phi$-YD C*-algebra $(A,\alpha,\lambda)$. 
	By regarding $A$ as a left $\wh{\phi}$-YD C*-algebra $(A,\lambda,\alpha)$, \cref{prop:dualtwiten} applied to the homomorphism $\wh{\phi}$ instead of $\phi$ gives the procedure 
	\begin{align*}
		\bfJ^{\phi}(A):=
		A\outensor{G}C^r_0(G)
		\cong 
		\bfJ^{\Ad \wh{W}^\phi}_{\wh{G}^{\op}}(A) ,
		%\cong ([\lambda(A)C^r_0(G)_2], \Ad W^{G*}_{13}(\alpha\otimes\id), \Ad V^G_{23}(-)_{12}) 
	\end{align*}
	yielding a left $H\times G^{\op}$-C*-algebra. 
	Here, note $G^{\op}\rbicross{\Ad\wh{W}^{\phi}} H \cong H\times G^{\op}$ as a consequence of \cref{eq:isomgenqd} for $\wh{\phi}$. 
	Thanks to \cref{thm:BSTTdualC*} for the matching $\Ad \wh{W}^{\phi}$ on $\wh{G}^{\op}$ and $H$, the procedure $\bfJ^{\phi}$ from left $\phi$-YD C*-algebras to left $H\times G^{\op}$-C*-algebras is bijective up to equivariant Morita equivalence. 
	Especially, when $\phi=\id_G$, this gives a duality between left $\sD(G)$-C*-algebras and $G\times G^{\op}$-C*-algebras.
	Due to conventional convenience, we mainly consider the this correspondence. 
	We also put $\wh{\bfJ}^{1_{H\to G}}(A) := A\outensor{\wh{G}^{\op}}C^r_0(\wh{G}^{\op})$ for a left $1_{H\to \wh{G}^{\op}}$-YD C*-algebra $A$. 
	Note the canonical equivariant $*$-isomorphism $\wh{\bfJ}^{1_{H\to G}}(A)\xrightarrow{\sim} \wh{\bfJ}^{\id}_{\wh{G}^{\op}}(A)$ for the trivial matching on $\wh{G}^{\op}$ and $H$. 
	
	In this occasion, we note the relation of $\bfJ^{\phi}$ with crossed products. 
	\begin{prop}\label{prop:crossprodqd}
		Let $\phi\colon H\to G$ be a homomorphism of regular locally compact quantum groups and $(A,\alpha,\lambda)$ be a left $\phi$-YD C*-algebra. 
		Then we have a canonical $*$-isomorphism $\sD(\phi)\redltimes A\cong C^r_0(\wh{H}) \outensor{H\times G^{\op}} \bfJ^{\phi}(A)$ constructed in the proof. 
		Here we regarded $A$ as a left $1_{H\times G^{\op}\to \{1\}}$-YD C*-algebra and $C^r_0(\wh{H})$ as a left $1_{\{1\}\to H\times G^{\op}}$-YD C*-algebra with the left $\wh{H}\times\wh{G}^{\op}$-action $\sigma_{23}(\wh{\Delta}_H\otimes\id)(\wh{\Delta}_H\rpb\wh{\phi})$. 
	\end{prop}
	
	Actually, we do not use the regularity in the proof. 
	
	\begin{proof}
		By \cite[Proposition~8.1, Proposition~9.1]{Baaj-Vaes:doublecrossed}, we have the description 
		\begin{align*}
			&
			C^r_0(\wh{\sD(\phi)})
			=
			\bigl[ C^r_0(\wh{H})_1 \bigl( \Ad ( (J^{H}\otimes\wh{J}^{G})W^{\phi*} (J^{H}\otimes\wh{J}^{G}) W^{\phi*} ) C^r_0(G)_2 \bigr) \bigr] . 
		\end{align*}
		We write $\wt{\delta}$ for the left $H\times G^{\op}$-action of $\bfJ^{\phi}(A)$ and have the composition of $*$-isomorphisms as follows, 
		\begin{align*}
			&
			\sD(\phi)\redltimes A 
			= 
			[ C^r_0(\wh{\sD(\phi)})_{12} (\id\otimes\lambda)\alpha(A) ] 
			\\&
			\xrightarrow{\Ad (U^G_2 \wh{V}^{\phi}_{12} U^{G*}_2) } 
			\bigl[ U^G_2 \wh{V}^{\phi}_{12} C^r_0(\wh{H})_1 \wh{V}^{\phi*}_{12} U^{G*}_2 
			W^{\phi*}_{12} C^r_0(G)_2 W^{\phi}_{12} 
			(\id\otimes\lambda)\alpha(A) \bigr] 
			\\&
			\xrightarrow{ \Ad (V^{G}_{32} U^G_2 \wh{V}^{\phi}_{12})(-)_{134} } 
			\bigl[ \bigl( \Ad (V^{G}_{32} U^{G\otimes2}_{32} \wh{V}^{\phi}_{12} \wh{V}^{\phi}_{13}) C^r_0(\wh{H})_1 \bigr) 
			\bigl( \Ad (V^{G}_{32}W^{\phi*}_{13}) C^r_0(G)_3 \bigr) 
			(\id\otimes\lambda)\alpha(A)_{134} \bigr] 
			\\&
			\cong 
			\bigl[ \bigl( \Ad (U^{G\otimes2}_{32} \wh{V}^{G*}_{23} \wh{V}^{G}_{23} \wh{V}^{\phi}_{12}) C^r_0(\wh{H})_1 \bigr) 
			\wt{\delta}\bfJ^{\phi}(A) \bigr] 
			\\&
			\xrightarrow{\Ad \wh{V}^H_{21}(-)_{2345}} 
			C^r_0(\wh{H}) \outensor{H \times G^{\op}} \bfJ^{\phi}(A) . 
			\qedhere\end{align*}
	\end{proof}
	
	\subsection{Monoidal operations}\label{ssec:moneqC*}
	\begin{rem}
		For a compact quantum group $G$, we write $\cO(G)\subset C^r(G)$ for the dense $*$-subalgebra spanned by the matrix coefficients of finite dimensional unitary representations of $G$. 
		We refer to \cite{Neshveyev-Tuset:book} for representation theory of compact quantum groups. 
		For a right $G$-C*-algebra $(A,\alpha)$ and a left $G$-C*-algebra $(B,\beta)$, we define the spectral subalgebras $\cA=\alpha^{-1}(A\odot \cO(G))$, $\cB:=\beta^{-1}(\cO(G)\odot B)$, and 
		the cotensor product by 
		\begin{align*}
			&
			A\cotensor{G}B:=[ x\in \cA\odot \cB \,|\, (\alpha\otimes\id)(x)=(\id\otimes\beta)(x) ] 
			\subset A\otimes B . 
		\end{align*}
		See \cite{deRijdt-vanderVennet:moneq,Voigt:bcfo}. 
		Let $\phi\colon H\to G$, $\psi\colon G\to F$ be homomorphisms of compact quantum groups. 
		For a left $\phi$-YD C*-algebra $(A,\alpha,\lambda)$ and a left $\psi$-C*-algebra $(B,\beta,\mu)$, 
		we can show there is a left $H\times \wh{F}^{\op}$-$*$-isomorphism 
		\begin{align*}
			\bfJ^{\phi}(A)\cotensor{G}\bfJ^{\psi}(B)\cong \bfJ^{\psi\phi}(A\outensor{G}B) 
		\end{align*}
		with the aid of \cite[Lemma~4.5]{Kitamura:indhomlcqg}, for example. 
	\end{rem}
	
	Unfortunately, sometimes cotensor products do not preserve Morita equivalences. %Also, we currently do not have the notion of cotensor products when $G$ is not compact.  
	Alternatively, we can treat its stabilized version by using twisted tensor products. 
	Let $(A,\alpha,\lambda)$ be a left $\phi$-YD C*-algebra for a homomorphism $\phi\colon H\to G$ of regular locally compact quantum groups. 
	Unless stated otherwise, in this section we always equip $L^2(G)$ with the left $H$-action $W^{\phi*}(-)_2$ and the left $\wh{G}$-action $\wh{V}^{G*}_{21}(-)_2$. 
	These actions induce a left $H$-action and a left $\wh{G}$-action on the right Hilbert $A$-module $L^2(G)\tensor{\bC}A$ that are continuous. 
	Also, they give a left $1_{H\to G}$-YD structure on $\cK_A(L^2(G)\tensor{\bC} A)$, 
	which can be checked to be left $\sD(1_{H\to G})$-$*$-isomorphic to 
	$\wh{\bfJ}^{1_{H\to G}}\bfJ^{\phi}(A)$ by 
	\begin{align*}
		\Ad \wh{W}^G_{12} (\id\otimes\lambda)\colon 
		\cK_A(L^2(G)\tensor{\bC}A) &\xrightarrow{\sim} 
		[ U^G_1 C^r_0(\wh{G})_1 U^{G*}_1 \Delta_G(C^r_0(G))_{21} \lambda(A)_{23} ] 
		\cong \wh{\bfJ}^{1_{H\to G}}\bfJ^{\phi}(A) .  
	\end{align*}
	Note that typically these actions on $L^2(G)\tensor{\bC}A$ itself do not satisfy left $1_{H\to G}$- nor $\phi$-YD conditions, but we have the following. 
	
	\begin{lem}\label{lem:monbimod}
		Let $\phi\colon H\to G$, $\psi\colon G\to F$ be homomorphisms of regular locally compact quantum groups with $\psi\phi=1_{H\to F}$, 
		$(A,\alpha,\lambda)$ be a left $\phi$-YD C*-algebra, and 
		$(B,\beta,\mu)$ be a left $\psi$-YD C*-algebra. 
		Then 
		\begin{align*}
			\cI_{A,B}:=
			\bigl(\, &
			(L^2(G)\tensor{\bC} A)\outensor{G}B,\ 
			W^{\phi*}_{12} \bigl((\id\otimes\alpha\otimes\id^{\otimes2})(-)\bigr)_{21345} W^{\phi}_{14},\\& 
			W^{\psi}_{41}(\id^{\otimes3}\otimes\mu)(-)_{23415}W^{\psi*}_{41}
			\, \bigr) 
		\end{align*}
		is a well-defined left $1_{H\to F}$-YD $\bigl( \cK_A(L^2(G)\tensor{\bC} A)\outensor{G}B ,\, A\outensor{G}B \bigr)$-imprimitivity bimodule. 
	\end{lem}
	
	For a C*-algebra $C$ and a Hilbert $C$-module $\cE$, we put $\cM(\cE):=\cL_C(C,\cE)$ in the proof. 
	
	\begin{proof}
		More precisely, as a left $\wh{F}$-C*-correspondence, $\cI_{A,B}$ is defined by the twisted tensor product of a left $\wh{G}$-C*-correspondence and a left $\psi$-YD C*-algebra 
		\begin{align*}
			&
			(L^2(G)\tensor{\bC} A)\outensor{G}B 
			= 
			[ \wh{V}^{G*}_{13}L^2(G)_1 \lambda(A)_{32} \beta(B)_{34} ] 
		\end{align*}
		inside $\cM(L^2(G)\otimes A\otimes \cK(G)\otimes B)$. 
		Using $W^{\phi*}_{12} \wh{V}^{G*}_{23} 
		= \wh{V}^{G*}_{23} W^{\phi*}_{12} W^{\phi*}_{13}$, 
		it is not hard to see the left $H$-action is well-defined and continuous. 
		It commutes with the left $\wh{F}$-action by 
		$W^{\psi}_{23}W^{\phi}_{12} = W^{\phi}_{12}W^{\psi\phi}_{13}W^{\psi}_{23} = W^{\phi}_{12}W^{\psi}_{23}$. 
	\end{proof}
	
	Note that the procedure $\bfJ^{\phi}$ sends a left $\sD(\phi)$-imprimitivity bimodule to a $H\times G^{\op}$-equivariant one by \cref{rem:BSTTdualbimod} for $\sm=\Ad \wh{W}^{\phi}$. 
	By letting $(F,\psi,B)$ in \cref{lem:monbimod} be $(\wh{F}^{\op}, 1_{G\to \wh{F}^{\op}}, \bfJ^{\psi}(B))$ and identifying $\wh{\bfJ}^{1_{H\to G}}\bfJ^{\phi}(A)\cong \cK_A(L^2(G)\tensor{\bC} A)$, we get the following. 
	
	\begin{prop}\label{prop:monbimod}
		Let $\phi\colon H\to G$, $\psi\colon G\to F$ be homomorphisms of regular locally compact quantum groups, 
		$(A,\alpha,\lambda)$ be a left $\phi$-YD C*-algebra, and $(B,\beta,\mu)$ be a left $\psi$-YD C*-algebra. 
		Then 
		\begin{align*}
			&
			\cJ_{A,B} := 
			\cI_{A,\bfJ^{\psi}(B)} 
			\tensor{\Ad W^{\psi}_{24} } 
			\bfJ^{\psi\phi}(A\outensor{G}B) 
		\end{align*}
		is a well-defined left $1_{H\to \wh{F}^{\op}}$-YD 
		$\bigl( \wh{\bfJ}^{1_{H\to G}}\bfJ^{\phi}(A)\outensor{G}\bfJ^{\psi}(B) ,\, \bfJ^{\psi\phi}(A\outensor{G}B) \bigr)$-imprimitivity bimodule. 
	\qed\end{prop}
	
	Here, we used the left $H\times F^{\op}$-$*$-isomorphism 
	\[ \Ad W^{\psi}_{24}\colon A\outensor{G}\bfJ^{\psi}(B) = 
	A\outensor{G}(B\outensor{F}C^r_0(F))\xrightarrow{\sim}(A\outensor{G}B)\outensor{F}C^r_0(F) = \bfJ^{\psi\phi}(A\outensor{G}B) \]
	for the tensor product of equivariant C*-correspondences. 
	More generally, consider homomorphisms of locally compact quantum groups $\phi\colon H\to G$, $\psi\colon G\to F$, $\theta\colon F\to E$, a left $\phi$-YD C*-algebra $A$, a left $\psi$-YD C*-algebra $B$, and a left $\theta$-YD C*-algebra $D$. 
	Then we have $\Ad W^\psi_{24}\colon A\outensor{G}(B\outensor{F}D)\xrightarrow{\sim} (A\outensor{G}B)\outensor{F}D$ as left $\theta\psi\phi$-YD C*-algebras inside $\cM(A\otimes\cK(G)\otimes B\otimes \cK(F)\otimes D)$. 
	The coherence for such $*$-isomorphisms can be assured thanks to $W^\psi_{23}W^\phi_{12}=W^\phi_{12}W^{\psi\phi}_{13}W^\psi_{23}$. 
	When $\psi=1_{G\to F}$, this $*$-isomorphism is trivial, and we can safely write the iterated twisted tensor product as $A\outensor{G} B \outensor{F}D$. 
	The next lemma says the equivariant imprimitivity bimodule $\cJ_{A,B}$ is compatible with this associativity. 
	Similarly to \cref{prop:monbimod}, when a $*$-homomorphism of the form $\Ad W$ with some leg numbering for some unitary $W$ appears in the tensor product of C*-correspondences, we understand it is with respect to the legs coming from iterated twisted tensor products. 
	
	\begin{lem}\label{lem:monbimodass}
		Let $\phi\colon H\to G$, $\psi\colon G\to F$, $\theta\colon F\to E$ be homomorphisms of regular locally compact quantum groups, and $(A,\alpha,\lambda)$ be a left $\phi$-YD C*-algebra, $(B,\beta,\mu)$ be a left $\psi$-YD C*-algebra, $(D,\delta,\nu)$ be a left $\theta$-YD C*-algebra. 
		Then 
		\begin{align*}
			\bigl( \wh{\bfJ}^{1_{H\to F}}(\cJ_{A,B})\outensor{F}\bfJ^\theta(D) \bigr) \tensor{ \wh{\bfJ}^{1_{H\to F}}\bfJ^{\psi\phi}(A\outensor{G} B) \outensor{F} \bfJ^{\theta}(D)} \cJ_{A\outensor{G}B,D}
		\end{align*}
		and 
		\begin{align*}
			\bigl( \wh{\bfJ}^{1_{H\to G}}\bfJ^\phi(A)\outensor{G} \cJ_{B,D} \bigr) \tensor{ \wh{\bfJ}^{1_{H\to G}}\bfJ^{\phi}(A) \outensor{G} \bfJ^{\theta\psi}(B\outensor{F} D)} \cJ_{A,B\outensor{F}D} \tensor{\Ad W^{\psi}_{24}} \bfJ^{\theta\psi\phi}((A\outensor{G}B)\outensor{F}D)
		\end{align*}
		are equivariantly unitarily isomorphic left $1_{H\to \wh{E}^{\op}}$-YD $\bigl( \wh{\bfJ}^{1_{H\to G}}\bfJ^\phi(A) \outensor{G} \wh{\bfJ}^{1_{G\to F}} \bfJ^\psi(B) \outensor{F} \bfJ^{\theta}(D), \, \bfJ^{\theta\psi\phi}((A\outensor{G}B)\outensor{F}D) \bigr)$-imprimitivity bimodules. 
	\end{lem}
	
	\begin{proof}
		We write $\wh{\delta}$ for the left $F$-action on $\bfJ^\theta(D)$. 
		We put $\cE_1$ for the former imprimitivity bimodule in the statement and 
		\begin{align*}
			\cE_2:=
			\bigl( \wh{\bfJ}^{1_{H\to G}}\bfJ^\phi(A)\outensor{G} \cJ_{B,D} \bigr) \tensor{\wh{\bfJ}\bfJ(A)\boxtimes \bfJ(B\boxtimes D)} \cJ_{A,B\boxtimes D} . 
		\end{align*}
		
		Since $\bfJ^{\psi}(B)$ is a left $1_{H\to \wh{F}^{\op}}$-YD C*-algebra, 
		we see $\wh{\bfJ}^{1_{H\to F}}(C\outensor{G} \bfJ^{\psi}(B)) = C\outensor{G}\wh{\bfJ}^{1_{G\to F}}\bfJ^{\psi}(B)$ for any left $\phi$-YD C*-algebra $C$ 
		and $\wh{\bfJ}^{1_{H\to F}}(\cI_{A,\bfJ(B)}) \cong \cI_{A, \wh{\bfJ}\bfJ(B)}$ 
		as equivariant $\bigl( \cK_A(L^2(G)\tensor{\bC}A) \outensor{G}\wh{\bfJ}^{1_{G\to F}}\bfJ^{\psi}(B), A\outensor{G}\wh{\bfJ}^{1_{G\to F}}\bfJ^{\psi}(B) \bigr)$-imprimitivity bimodules. 
		Then we can check 
		\begin{align*}
			&
			\cE_1\tensor{ \Ad W^{\theta*}_{46} } (A\outensor{G}B)\outensor{F} \bfJ^\theta(D) 
			\\&
			\cong
			\bigl( \wh{\bfJ}^{1_{H\to F}}(\cI_{A,\bfJ(B)}) \outensor{F} \bfJ^\theta(D) \bigr) 
			\tensor{\Phi}
			\cI_{A\boxtimes B,\bfJ(D)}
			\\&
			\cong
			\Bigl[
			\wh{V}^{G*}_{14} L^2(G)_1 W^{\psi*}_{42} \wh{V}^{F*}_{26} L^2(F)_2 W^{\psi}_{46} \lambda(A)_{43} \bigl((\id\otimes\mu)\beta(B)\bigr)_{465} W^{\psi*}_{46} \wh{\delta}(\bfJ^\theta(D))_{6789}
			\Bigr]  \\&=:\cF_1
		\end{align*}
		inside $\cM(L^2(G)\otimes L^2(F)\otimes A\otimes \cK(G)\otimes B\otimes \cK(F)\otimes D\otimes \cK(E)\otimes C^r_0(E))$, where 
		\begin{align*}
			&
			\Psi\colon 
			A\outensor{G}\wh{\bfJ}^{1_{G\to F}}\bfJ^{\psi}(B)  
			= 
			\wh{\bfJ}^{1_{H\to F}} ( A\outensor{G}\bfJ^{\psi}(B) ) 
			\\&
			\xrightarrow{\Ad W^{\psi}_{24}} 
			\wh{\bfJ}^{1_{H\to F}} ( \bfJ^{\psi\phi}(A\outensor{G}B) ) 
			\cong 
			\cK_{A\boxtimes B}\bigl( L^2(F)\tensor{\bC} (A\outensor{G}B) \bigr) 
		\end{align*}
		is the left $H\times F^{\op}$-$*$-isomorphism and 
		\begin{align*}
			&
			\Phi:=\Psi\otimes\id^{\otimes4}\colon A\outensor{G}\wh{\bfJ}^{1_{G\to F}}\bfJ^{\psi}(B) 
			\outensor{F} \bfJ^{\theta}(D) 
			\xrightarrow{\sim} 
			\cK_{A\boxtimes B} \bigl( L^2(F)\tensor{\bC} (A\outensor{G}B) \bigr) \outensor{F} \bfJ^{\theta}(D) . 
		\end{align*}
		
		By unwinding definitions, the left $E^{\op}$-action on $\cF_1$ is $\Ad V^E_{01}(-)_{234567890}$ (where we promise the leg $0$ is placed on the right of $9$), and 
		the left $H$-action on $\cF_1$ coincide with 
		\begin{align*}
			W^{\phi*}_{12} \bigl( \sigma_{12}\sigma_{23} (\id^{\otimes2}\otimes\alpha\otimes\id^{\otimes6})(-) \bigr) W^{\phi}_{15} W^{\psi\phi}_{17} . 
		\end{align*}
		
		Next, we observe 
		\begin{align*}
			&
			\cI_{A, \bfJ(B\boxtimes D)} 
			\tensor{\Ad W^{\theta*}_{46}} 
			A\outensor{G} (B\outensor{F}\bfJ^{\theta}(D))
			\\&
			\cong 
			\bigl( \cK_{A} (L^2(G) \tensor{\bC} A) \outensor{G} \bfJ^{\theta\psi}(B\outensor{F} D) \bigr) \tensor{\Ad W^{\theta*}_{57}} \cI_{A, B\boxtimes \bfJ(D)} ,
		\end{align*}
		and 
		\begin{align*}
			&
			\cE_2\tensor{ \Ad (W^{\theta*}_{46}W^{\theta\psi*}_{26}) } A\outensor{G}(B\outensor{F}\bfJ^\theta(D)) 
			\\&
			\cong
			\bigl( \cK_A(L^2(G)\tensor{\bC}A) \outensor{G} \cI_{B, \bfJ(D)} \bigr) 
			\tensor{\cK( (L^2(G)\otimes A) \boxtimes (B\boxtimes \bfJ(D) ) )}
			\cI_{A, B\boxtimes \bfJ(D)} 
			\\&
			\cong
			\Bigl[
			W^{\psi*}_{42} \wh{V}^{F*}_{26} L^2(F)_2 W^{\psi}_{46} \wh{V}^{G*}_{14} L^2(G)_1 \lambda(A)_{43} W^{\psi*}_{46} \bigl((\id\otimes\beta)\mu(B)\bigr)_{645} \wh{\delta}(\bfJ^\theta(D))_{6789} W^{\psi}_{46}
			\Bigr] \\&=:\cF_2
		\end{align*}
		inside 
		$\cM(L^2(G)\otimes L^2(F)\otimes A\otimes \cK(G)\otimes B\otimes \cK(F)\otimes D\otimes \cK(E)\otimes C^r_0(E))$. 
		The left $E^{\op}$-action on $\cF_2$ is $\Ad V^E_{01}(-)_{234567890}$ and 
		the left $H$-action on $\cF_2$ is 
		\begin{align*}
			W^{\phi*}_{12}\bigl( \sigma_{12}\sigma_{23} (\id^{\otimes2}\otimes\alpha\otimes\id^{\otimes6})(-) \bigr) W^{\phi}_{15} . 
		\end{align*}
		
		From these expressions we can check 
		$\cF_1\tensor{\Ad W^{\psi*}_{24}} (A\outensor{G} B)\outensor{F} \bfJ^{\theta}(D) \cong \cF_2$ 
		with the aid of 
		$\wh{V}^{G*}_{14}W^{\psi*}_{42}\wh{V}^{F*}_{26}W^{\psi}_{46} = W^{\psi*}_{42}\wh{V}^{F*}_{26}W^{\psi}_{46}\wh{V}^{G*}_{14}\wh{V}^{G*}_{12}$. 
		Hence by letting $C:=\bfJ^{\theta\psi\phi}((A\outensor{G} B)\outensor{F} D)$, we conclude 
		\begin{align*}
			&
			\cE_1
			\cong 
			\cF_1 \tensor{\Ad W^{\theta}_{46}} 
			C
			\cong 
			\cF_1 \tensor{\Ad (W^{\psi}_{24}W^{\theta\psi}_{26}W^{\theta}_{46}W^{\psi*}_{24}) } 
			C
			\\&
			\cong 
			\cF_2 \tensor{\Ad (W^{\psi}_{24}W^{\theta\psi}_{26}W^{\theta}_{46}) } 
			C
			\cong 
			\cE_2 \tensor{\Ad W^{\psi}_{24} } 
			C . 
		\qedhere\end{align*}
	\end{proof}

	\section{Consequences for equivariant Kasparov theory}\label{sec:KK}
	
	Let $G$ be a regular locally compact quantum group with separable $L^2(G)$. 
	We write $\Calg^G$ for the category of separable left $G$-C*-algebras and left $G$-$*$-homomorphisms $A\to B$ for any separable left $G$-C*-algebras $A,B$. 
	We define the $G$-equivariant Kasparov category $\KK^G$ as 
	the category such that its objects are separable left $G$-C*-algebras, 
	and the set of morphisms from $A$ to $B$ is the $G$-equivariant KK-group $\KK^G(A,B)$ for objects $A,B\in\KK^G$. 
	Here we refer to \cite{Baaj-Skandalis:eqkk} for the definition of $\KK^G(A,B):=\KK^{C^r_0(G)}(A,B)$. 
	The category $\KK^G$ is additive and has arbitrary countable direct sums. 
	
	We briefly recall the Cuntz picture of $\KK^G$. See the argument of \cite[Theorem~6.5]{Meyer:genhom} and \cite[Theorem~4.3]{Nest-Voigt:eqpd} for details. 
	For a separable left $G$-C*-algebra $(A,\alpha)$, we take $A\ast_G A$ as the quotient of the free product $A\ast A$ by the kernel of the well-defined $*$-homomorphism 
	\begin{align*}
		&
		\wt{\alpha}:=((\id\otimes\iota_1)\alpha)\ast((\id\otimes\iota_2)\alpha)\colon A\ast A\to \cM(C^r_0(G)\otimes (A\ast A)) , 
	\end{align*}
	where $\iota_1, \iota_2\colon A\to A\ast A$ are the embedding of $A$ into the left and right component of $A\ast A$, respectively. 
	Then $\wt{\alpha}$ induces a well-defined continuous left $G$-action on $A\ast_G A$. 
	We define $\mrq A:=\Ker( \id_A\ast_G \id_A \colon A\ast_G A\to A )$, which is a left $G$-C*-subalgebra of $A\ast_G A$, and by putting 
	$A_{\cK}:=\cK_A(L^2(G)^{\oplus\infty}\tensor{\bC}A)$, 
	\begin{align*}
		&
		\wt{\mrq}A:= \cK_{\mrq A_{\cK}}\bigl( L^2(G)^{\oplus\infty}\tensor{\bC}\mrq A_{\cK} \bigr) . 
	\end{align*}
	Then the canonical map 
	$\Hom_{\Calg^{G}}(\wt{\mrq}A, \wt{\mrq}B){\slash\!\!\simeq} \to \KK^G(\wt{\mrq}A, \wt{\mrq}B)$ 
	is bijective, 
	where $\simeq$ is the $G$-equivalent homotopy equivalence of left $G$-$*$-homomorphisms. 
	Here we note the composition of left $G$-$*$-homomorphisms 
	$\id_A\ast_G 0\colon \mrq A\subset A\ast_G A\to A$ is a $\KK^G$-equivalence. 
	Thus as a composition of $G$-imprimitivity bimodules and $\id\ast_G0$, 
	we get a proper left $G$-$(\wt{\mrq}A, A)$-correspondence $\wt{\pi}_A$ 
	giving a $\KK^G$-equivalence. 
	Using this we can deduce the universality of the canonical functor $\Calg^G\to \KK^G$ among split exact stable homotopy functors from $\Calg^G$ to additive categories. 
	For the precise meaning of this statement, see \cite[Theorem~4.4]{Nest-Voigt:eqpd}. 
	
	Also, by using the Cuntz picture, we can equip $\KK^G$ with the canonical structure of a triangulated category given by mapping cone construction. 
	See \cite[Proposition~4.5]{Nest-Voigt:eqpd}, \cite[Appendix]{Meyer-Nest:bctri} for details, and \cite{Neeman:book} for the definition and basic properties of triangulated categories. 
	
	\subsection{Categorical equivalences}\label{ssec:catequiv}
	
	We can restate \cref{thm:BSTTdualC*} as follows. 
	\begin{thm}\label{thm:BSTTdualcat}
		Let $G$ and $H$ be regular locally compact quantum groups with separable $L^2(G)$ and $L^2(H)$, and 
		$\sm$ be a matching on $G$ and $H$ that is continuous in the strict sense. 
		Then there are mutually quasi-inverse equivalences of triangulated categories 
		$\bfJ^{\sm}_{G} \colon \KK^{G^{\op}\dblcross{\sm}H}\to \KK^{\wh{G}\rbicross{\sm} H}$
		and 
		$\wh{\bfJ}^{\sm}_{G} \colon \KK^{\wh{G}\rbicross{\sm} H} \to \KK^{G^{\op}\dblcross{\sm}H}$. 
		Moreover, $\bfJ^{\sm}_{G}$, $\wh{\bfJ}^{\sm}_{G}$ and the functors of Baaj--Skandalis duality $\KK^{G^{\op}}\simeq \KK^{\wh{G}}$ commute with the restrictions 
		$\imath_{G}^* \colon \KK^{G^{\op}\dblcross{\sm}H}\to \KK^{G^{\op}}$ and 
		$\wh{\pi}_G^* \colon \KK^{\wh{G}\rbicross{\sm}H}\to \KK^{\wh{G}}$ 
		up to natural isomorphisms. 
	\end{thm}
	
	\begin{proof}
		First, $\bfJ^{\sm}_G$ gives a well-defined functor of 
		$\Calg^{G^{\op}\dblcross{\sm}H}\to \Calg^{\wh{G}\rbicross{\sm}H}$ 
		as in (1) of \cref{rem:BSTTdualbimod} for morphisms. 
		We can check that this induces a well-defined triangulated functor $\KK^{G^{\op}\dblcross{\sm}H}\to \KK^{\wh{G}\rbicross{\sm}H}$ by using the universality. 
		Similarly, we get a well-defined triangulated functor $\wh{\bfJ}^{\sm}_G\colon \KK^{\wh{G}\rbicross{\sm}H}\to \KK^{G^{\op}\dblcross{\sm}H}$. 
		To show $\wh{\bfJ}^{\sm}_G\bfJ^{\sm}_G\cong \id_{\KK^{G^{\op}\dblcross{\sm}H}}$, 
		we see the equivariant $*$-isomorphism  
		$\wh{\bfJ}^{\sm}_{G}\bfJ^{\sm}_{G}(A)\cong \cK_A(L^2(G)\tensor{\bC}A)$
		in \cref{thm:BSTTdualC*} 
		is natural with $A\in \Calg^{G^{\op}\dblcross{\sm}H}$ 
		and the right hand side induces an endofunctor on $\KK^{G^{\op}\dblcross{\sm}H}$ that is naturally isomorphic to $\id$. 
		Similarly $\bfJ^{\sm}_G\wh{\bfJ}^{\sm}_G\cong \id_{\KK^{\wh{G}\rbicross{\sm}H}}$. 
	\end{proof}
	
	Using twisted tensor products, 
	we can equip $\KK^{\sD(G)}$ with a monoidal structure by following \cite{Nest-Voigt:eqpd}. 
	See \cite{Etingof-Gelaki-Nikshych-Ostrik:book} for the definition and basic properties of monoidal categories. 
	
	Consider homomorphisms $\phi\colon H\to G$, $\psi\colon G\to F$, $\theta\colon F\to E$, $\eta\colon K\to H$ of regular locally compact quantum groups with separable $L^2$-spaces. 
	We recall $\bC_\phi$ denotes $\bC$ with the trivial left $\sD(\phi)$-action. 
	For $B\in\Calg^{\sD(\psi)}$, we have a well-defined (covariant) functor 
	$-\outensor{G}B\colon \Calg^{\sD(\phi)}\to\Calg^{\sD(\psi\phi)}$ 
	and this induces a triangulated functor 
	$-\outensor{G}B\colon \KK^{\sD(\phi)}\to\KK^{\sD(\psi\phi)}$ by the universality and techniques from \cite[Section~5]{Meyer-Roy-Woronowicz:twiten}. 
	More precisely, for $\bx\in \KK^{\sD(\phi)}(A_1,A_2)$ 
	we take $f \in \Hom_{\Calg^{G}}(\wt{\mrq}A_1, \wt{\mrq}A_2)$ 
	representing $[\wt{\pi}_{A_2}]^{-1}\circ \bx\circ [\wt{\pi}_{A_1}] \in \KK^{\sD(\phi)}(\wt{\mrq}A_1, \wt{\mrq}A_2)$ and let 
	\begin{align*}
		\bx\outensor{G}B 
		:= 
		[\wt{\pi}_{A_2}\outensor{G}B] \circ [f\otimes\id^{\otimes2}] \circ [\wt{\pi}_{A_1}\outensor{G}B]^{-1}, 
	\end{align*}
	which is independent on the choice of $f$. 
	When $\bx$ is given by a proper left $G$-$(A_1,A_2)$-correspondence $\cE$, then $[\cE]\outensor{G}B$ is given by the left $G$-correspondence $\cE\outensor{G}B$. 
	We can similarly define a triangulated functor $A\outensor{G}-\colon \KK^{\sD(\psi)}\to \KK^{\sD(\psi\phi)}$ for $A\in \Calg^{\sD(\phi)}$. 
	By definition $\bC_{\id_G}\outensor{G}-$ and $-\outensor{G}\bC_{\id_H}$ are both naturally isomorphic to $\id$ on $\KK^{\sD(\phi)}$. 
	
	For any $f\in\Hom_{\Calg^{\sD(\phi)}}(\wt{\mrq}A_1, \wt{\mrq}A_2)$, $g\in\Hom_{\Calg^{\sD(\psi)}}(\wt{\mrq}B_1, \wt{\mrq}B_2)$, 
	$\bx := [\wt{\pi}_{A_2}] \circ [f] \circ [\wt{\pi}_{A_1}]^{-1} \in \KK^{\sD(\phi)}(A_1,A_2)$ and 
	$\by := [\wt{\pi}_{B_2}] \circ [g] \circ [\wt{\pi}_{B_1}]^{-1} \in  \KK^{\sD(\psi)}(B_1,B_2)$, 
	we have the following commutative diagram in $\KK^{\sD(\psi\phi)}$
	\begin{align}\label{diag:cubicbifunc}
		&
		\begin{aligned}
			\xymatrix{
				A_1\outensor{G}B_1
				\ar[rrr]^-{\bx\boxtimes B_1}
				\ar[ddd]^-{A_1\boxtimes \by}
				&&&
				A_2\outensor{G}B_1
				\ar[ddd]^-{A_2\boxtimes \by}
				\\&
				\wt{\mrq}A_1\outensor{G}\wt{\mrq}B_1
				\ar[ul]_-{[\wt{\pi}_{A_1}\boxtimes \wt{\pi}_{B_1}]}
				\ar[r]^-{[f\boxtimes \id]}
				\ar[d]^-{[\id\boxtimes g]}
				&
				\wt{\mrq}A_2\outensor{G}\wt{\mrq}B_1
				\ar[ur]^-{[\wt{\pi}_{A_2}\boxtimes \wt{\pi}_{B_1}]}
				\ar[d]^-{[\id\boxtimes g]}
				&\\&
				\wt{\mrq}A_1\outensor{G}\wt{\mrq}B_2
				\ar[dl]^-{[\wt{\pi}_{A_1}\boxtimes \wt{\pi}_{B_2}]}
				\ar[r]_-{[f\boxtimes \id]}
				&
				\wt{\mrq}A_2\outensor{G}\wt{\mrq}B_2
				\ar[dr]_-{[\wt{\pi}_{A_2}\boxtimes \wt{\pi}_{B_2}]}
				&\\
				A_1\outensor{G}B_2
				\ar[rrr]_-{\bx\boxtimes B_2}
				&&&
				A_2\outensor{G}B_2
			}
		\end{aligned}
	\end{align}
	with $\KK^{\sD(\psi\phi)}$-equivalences $[\wt{\pi}_{A_i}\outensor{G}\wt{\pi}_{B_j}]$ for $i,j=1,2$. 
	Therefore we have a well-defined bifunctor 
	$-\outensor{G}-\colon \KK^{\sD(\phi)}\times \KK^{\sD(\psi)} \to \KK^{\sD(\psi\phi)}$, 
	which sends $(\bx,\by)$ to 
	\begin{align*}
		&
		\bx\outensor{G}\by
		:=
		(\bx\outensor{G}B_2)\circ(A_1\outensor{G}\by) 
		= 
		(A_2\outensor{G}\by)\circ(\bx\outensor{G}B_1) . 
	\end{align*}
	We have a natural isomorphism 
	$\Ad W^{\psi}\colon -\outensor{G}(-\outensor{H}-)\xrightarrow{\sim}(-\outensor{G}-)\outensor{H}-$ 
	of the functors of 
	\begin{align*}
		&
		\Calg^{\sD(\phi)}\times \Calg^{\sD(\psi)}\times \Calg^{\sD(\theta)}
		\to 
		\Calg^{\sD(\theta\psi\phi)} , 
	\end{align*}
	which induces a natural isomorphism in the level of equivariant Kasparov categories 
	by an argument like \cref{diag:cubicbifunc} using $\wt{\pi}$. 
	Also, we have the ``pentagonal" equality of two natural isomorphisms 
	\begin{align*}
		&
		\Ad W^{\psi}_{46}\Ad W^{\phi}_{24} 
		= 
		\Ad W^{\phi}_{24} \Ad W^{\psi\phi}_{26} \Ad W^{\psi}_{46}
	\end{align*}
	of the functors 
	\begin{align*}
		&
		-\outensor{H}(-\outensor{G}(-\outensor{F}-)), 
		((-\outensor{H}-)\outensor{G}-)\outensor{F}- 
		\colon 
		\KK^{\sD(\eta)}\times\KK^{\sD(\phi)}\times\KK^{\sD(\psi)}\times\KK^{\sD(\theta)}
		\to \KK^{\sD(\theta\psi\phi\eta)} . 
	\end{align*}
	Especially, we have the following corollary. 
	As in \cite{Nakano-Vashaw-Yakimov:nctentrigeom}, 
	we say a category $\cC$ with a structure of a monoidal category and a triangulated category is a \emph{monoidal triangulated category} 
	if the bifunctor of the monoidal operation $-\otimes -\colon \cC\times\cC\to \cC$ is componentwise triangulated. 
	
	\begin{cor}[{\cite[Theorem~4.10]{Nest-Voigt:eqpd}}]\label{cor:tensortriangulatedKK}
		Let $G$ be a regular locally compact quantum group with separable $L^2(G)$. 
		Then $\KK^{\sD(G)}$ is a monoidal triangulated category with $-\outensor{G}-$ and the unit object $\bC_{\id_G}$. 
	\qed\end{cor}
	
	In the case of generalized quantum doubles, the results of \cref{ssec:moneqC*} can be restated as follows. 
	\begin{thm}\label{thm:moneqPDYD}
		Let $\phi\colon H\to G$, $\psi\colon G\to F$, $\theta\colon F\to E$ be homomorphisms of regular locally compact quantum groups with separable $L^2$ spaces. 
		Then we have the commutative diagram in $\KK^{G\times E^{\op}}$ 
		that are natural with $A\in\KK^{\sD(\phi)}$, $B\in\KK^{\sD(\psi)}$, and $D\in\KK^{\sD(\theta)}$ as follows, 
		\begin{align*}
			&
			\begin{aligned}
				&
				\xymatrix@C=-7.5em{
					&\bfJ^\phi(A) \outensor{\wh{G}^{\op}} C^r_0(\wh{G}) \outensor{G} 
					\bfJ^\psi(B) \outensor{\wh{F}^{\op}} C^r_0(\wh{F}) \outensor{F} \bfJ^\theta(D) 
					\ar[dl]_-{\bfJ^{\phi}(A) \outensor{} C^r_0(\wh{G}) \outensor{} \cJ_{B,D} \qquad} 
					\ar[dr]^-{\qquad \cJ_{A,B} \outensor{} C^r_0(\wh{F}) \outensor{} \bfJ^\theta(D) }
					&\\
					\bfJ^\phi(A) \outensor{\wh{G}^{\op}} C^r_0(\wh{G}) \outensor{G} 
					\bfJ^{\theta\psi}(B\outensor{F}D) 
					\ar[d]_-{\cJ_{A,B\outensor{}D}} &&
					\bfJ^{\psi\phi}(A\outensor{G}B) 
					\outensor{\wh{F}^{\op}} C^r_0(\wh{F}) \outensor{F}\bfJ^\theta(D) 
					\ar[d]^-{\cJ_{A\outensor{}B,D}} \\
					\bfJ^{\theta\psi\phi}\bigl( A\outensor{G}(B\outensor{F}D) \bigr)
					\ar[rr]^-{\bfJ^{\theta\psi\phi}(\Ad W^{\psi}_{24})}&& 
					\bfJ^{\theta\psi\phi}\bigl( (A\outensor{G}B)\outensor{F}D \bigr) , 
				}
			\end{aligned}
		\end{align*}
		where $\cJ$ is the equivariant imprimitivity bimodule from \cref{prop:monbimod}. 
	\end{thm}
	
	\begin{proof}
		By \cref{lem:monbimodass}, it is enough to show that $\cJ_{A,B}$ is natural with $A\in\KK^{\sD(\phi)}$ and $B\in\KK^{\sD(\psi)}$. 
		We take $A_1,A_2,B_1,B_2,f,g,\bx,\by$ as in \cref{diag:cubicbifunc}. 
		For a proper left $\sD(\phi)$-$(A_1,A_2)$-correspondence $\cE$ and a proper left $\sD(\psi)$-$(B_1,B_2)$-correspondence $\cF$, 
		it is not hard to see 
		\begin{align*}
			&
			\wh{\bfJ}^{1_{H\to G}}\bfJ^{\phi}(\cE)\outensor{G}\bfJ^{\psi}(\cF) 
			\tensor{ \wh{\bfJ}\bfJ(A_2)\boxtimes \bfJ(B_2) }
			\cJ_{A_2,B_2}
			\quad\text{and}
			\\&
			\cJ_{A_1,B_1} 
			\tensor{ \bfJ(A_1\boxtimes B_1) } 
			\bfJ^{\psi\phi}(\cE\outensor{G}\cF) , 
		\end{align*}
		give the same element in $\KK^{G\times F^{\op}}( \bfJ^\phi(A_1)\outensor{G}C^r_0(G)\outensor{\wh{G}^{\op}} \bfJ^{\psi}(B_1), \bfJ^{\psi\phi}(A_2\outensor{G}B_2) )$, 
		by considering $\cJ$ for the linking algebras $\cK_{A_2}(\cE\oplus A_2)$ and $\cK_{B_2}(\cF\oplus B_2)$. 
		Now the claim follows from an argument like \cref{diag:cubicbifunc} using $\wt{\pi}$. 
	\end{proof}
	
	Note that the argument of \cref{cor:tensortriangulatedKK} also shows that twisted tensor products give an associative bifunctor on $\KK^{\sD(\phi)}$ for a homomorphism $\phi\colon G\to G$ with $\phi\phi=\phi$. 
	Now the naturality of the diagram in \cref{thm:moneqPDYD} shows the following. 
	\begin{cor}\label{cor:moneqPDqd}
		Let $\phi\colon G\to G$ be a homomorphism on a regular locally compact quantum group with separable $L^2(G)$ that is idempotent in the sense of $\phi\phi=\phi$. 
		Then $\KK^{\sD(\phi)}$ is a monoidal triangulated category with $-\outensor{G}-$ 
		and the unit object $\wh{\bfJ}^{\Ad \wh{W}^{\phi}}_{\wh{G}^{\op}}\bfJ^{\id_G}(\bC_{\id_G})$ 
		which is equivalent to $\KK^{\sD(G)}$. 
		Especially, we have an equivalence of monoidal triangulated categories $\KK^{\sD(G)}\simeq \KK^{\sD(1_{G\to G})}$. 
	\qed\end{cor}
	
	\begin{rem}
		Note that $\sD(\phi)$ is a 2-cocycle twist of $H\times\wh{G}$ by $\Omega=W^{\phi}_{14} \in \cM((C^r_0(H)\otimes C^r_0(\wh{G}))^{\otimes2})$ in the sense of \cite{deCommer:Galois}, 
		and they are comonoidally equivalent in the sense of \cite[Proposition~4.3]{deCommer:comonoidal}. 
		Then it is shown in \cite{Crespo:moneqKK} that such a pair of regular locally compact quantum groups give equivalent equivariant Kasparov categories. 
		\cref{cor:moneqPDqd} can be regarded as its refinement in the case of quantum doubles. 
	\end{rem}
	
	We also note the consequence of \cref{prop:crossprodqd}. 
	
	\begin{cor}\label{cor:crossprodqd}
		Let $\phi\colon H\to G$ be a homomorphism of regular locally compact quantum groups with separable $L^2(G)$ and $L^2(H)$. 
		Then the triangulated functor $\sD(\phi)\redltimes -\colon \KK^{\sD(\phi)}\to\KK$ is naturally isomorphic to 
		the composition of $C^r_0(\wh{H})\outensor{H\times G^{\op}}- \colon \KK^{H\times G^{\op}}\to \KK$ 
		after $\bfJ^{\phi}\colon \KK^{\sD(\phi)}\xrightarrow{\sim} \KK^{H\times G^{\op}}$. 
	\qed\end{cor}
	
	\subsection{Quantum analogue of the property \texorpdfstring{$\gamma=1$}{γ=1}}\label{ssec:gamma}
	
	We indicate how we can apply this categorical equivalence to the quantum analogue of the Baum--Connes conjecture from the viewpoint of Meyer--Nest \cite{Meyer-Nest:bctri}. 
	For simplicity, we will restrict ourselves to the properties which can be hoped to be true at most for ``torsion-free" discrete quantum groups. 
	
	Before doing this, we recall the notion of induction procedures quantum group actions on C*-algebras to use here and in the next section. We refer to \cite{Vaes:impr,Nest-Voigt:eqpd,Kitamura:indhomlcqg} for details. 
	Consider a homomorphism $\phi\colon H\to G$ of regular locally compact quantum groups. 
	We say $\phi$ is \emph{proper} if 
	there is a unital normal $*$-homomorphism $\wh{\phi}^r\colon L^\infty(\wh{H})\to L^\infty(\wh{G})$ with $(\wh{\phi}^r\otimes\wh{\phi}^r) \wh{\Delta}_H = \wh{\Delta}_G\wh{\phi}^r$ such that $(\id\otimes \wh{\phi}^r)(W^H)=W^\phi$. 
	Moreover, when $\wh{\phi}^r$ is injective, we say $\phi$ gives a \emph{closed quantum subgroup}. 
	When $\phi$ is proper and $A$ is a left $H$-C*-algebra, we have an induced left $G$-C*-algebra $\Ind_{\phi} A$. 
	When $\phi$ gives a closed quantum subgroup, 
	$\Ind_\phi$ induces a triangulated functor $\Ind^G_H\colon \KK^H\to \KK^G$, and this is naturally isomorphic to 
	$\wh{G}^{\op}\redltimes (\wh{\phi}^{\op*}(H\redltimes -)) \colon \KK^H\to\KK^G$. 
	For a general proper homomorphism $\phi$, the procedure $\Ind_\phi$ does not behave well with equivariant Morita equivalences, and does not induce a triangulated functor $\KK^H\to \KK^G$. 
	When $\phi$ gives a closed quantum subgroup, 
	the left $G$-C*-algebra $C^r_0(G/H):=\Ind_{\phi} \bC$ is considered to realize the quantum homogeneous space $G/H$. 
	It also has a continuous left $\wh{G}$-action $\Ad W^G_{21}(-)_2$ by the proof of \cite[Theorem~8.2]{Vaes:impr}, and 
	we regard $C^r_0(G/H)$ as a left $G$-YD C*-algebra. 
	
	For a locally compact quantum group $G$, we put $G^{\bi}:=G\times G^{\op}$. 
	When we have a homomorphism $\phi \colon H\to G$ of locally compact quantum groups, 
	we have another homomorphism $\phi^{\bi}\colon H^{\bi}\to G^{\bi}$ with $W^{\phi^{\bi}} = W^{\phi}_{13}V^{\phi*}_{42}$. 
	\begin{prop}\label{prop:hmgvscotensor}
		Let $\phi\colon H\to G$ be a homomorphism of regular locally compact groups giving a closed quantum subgroup. 
		Then we have a $G^{\bi}$-$*$-isomorphism 
		$\bfJ^{\id_G} (C^r_0(G/H)) \cong \Ind^G_H(\phi^*C^r_0(G))$. 
	\end{prop}
	
	\begin{proof}
		We regard $C^r_0(G)$ as a left $G^{\bi}$-C*-algebra. 
		By \cite[Proposition~6.3]{Kitamura:indhomlcqg} we get left $G^{\bi}$-$*$-isomorphisms 
		\begin{align*}
			&
			(\Ind^G_H \bC)\outensor{G}C^r_0(G) \cong \Ind^G_H(\bC_\phi\outensor{G}C^r_0(G)) \cong \Ind^G_H(\phi^*C^r_0(G)) . 
		\qedhere\end{align*}
	\end{proof}
	
	For a triangulated category $\cT$ closed under direct sums of countably many objects, 
	a full subcategory $\cS$ of $\cT$ is called \emph{localizing} 
	if $\cS$ is closed under taking distinguished triangles, countable direct sums, and isomorphic objects in $\cT$. 
	For a class of objects $\cC$ in $\KK^G$, we write $\bra\, \cC \,\ket_{\loc}^{G} \subset \KK^G$ for the smallest localizing subcategory of $\KK^G$ containing $\cC$. 
	
	\begin{rem}
		For a homomorphism $\phi\colon H\to G$ of regular locally compact quantum group and a left $H$-C*-algebra $A$, it holds by \cite[Example~5.9]{Kitamura:indhomlcqg}, 
		\begin{align*}
			&
			\Ind^{\sD(\phi)}_H A 
			= 
			(\, C^r_0(\wh{G})\otimes A,\, \Ad W^{\phi*}_{12}\sigma_{12}(\id\otimes\alpha),\, \wh{\Delta}_G\otimes \id \,) . 
		\end{align*}
		We have a left $H\times G^{\op}$-$*$-isomorphism 
		\begin{align*}
			&
			\cK_A(L^2(G)\tensor{\bC}A)
			\xrightarrow{\Ad \wh{V}^{G}_{12}(-)_{13}}
			[ \wh{\Delta}_G(C^r_0(\wh{G}))_{12} A_{3} C^r_0(G)_{1} ] 
			\cong 
			\bfJ^{\phi}(\Ind^{\sD(\phi)}_H A) , 
		\end{align*}
		where in the left hand side we consider the left $H\times G^{\op}$-Hilbert space $(L^2(G), W^{\phi*}_{13}V^{G}_{32}(-)_3)$ 
		and we equip $A$ with the trivial left $G^{\op}$-action. 
		Thus $\bfJ^{\phi}(\Ind^{\sD(\phi)}_H A)$ is left $H\times G^{\op}$-Morita equivalent to $A$. 
	\end{rem}
	
	\begin{prop}\label{prop:gamma1}
		For a compact quantum group $G$ with separable $L^2(G)$, the following are equivalent. 
		\begin{enumerate}
			\item
			$C^r(G)\in \bra\, \KK \,\ket_{\loc}^{G^{\bi}}$. 
			\item[(1')]
			$C^r(G)\in \bra\, \Res^G_{G^{\bi}}\KK^G \,\ket_{\loc}^{G^{\bi}}$, 
			where we consider the projection onto the first component $G^{\bi}\to G$. 
			\item
			$\bC\in \bra\, \Ind^{\sD(G)}_{G}\KK^G \,\ket_{\loc}^{\sD(G)}$. 
			\item
			$\KK^{G^{\bi}} = \bra\, \KK \,\ket_{\loc}^{G^{\bi}}$. 
			\item[(3')]
			$\KK^{G^{\bi}} = \bra\, \Res^G_{G^{\bi}}\KK^G \,\ket_{\loc}^{G^{\bi}}$. 
			\item
			$\KK^{\sD(G)}= \bra\, \Ind^{\sD(G)}_{G}\KK^G \,\ket_{\loc}^{\sD(G)}$. 
		\end{enumerate}
		Either one of them implies $\KK^G=\bra\, \KK \,\ket_{\loc}^G$. 
	\end{prop}
	
	\begin{proof}
		We see $\bfJ^{\id_G}(\bC)=\bfJ^{\id_G}(C^r_0(G/G))\cong C^r(G)$ by \cref{prop:hmgvscotensor}, and 
		the equivalences (3')$\Leftrightarrow$(4) and (1')$\Leftrightarrow$(2) by the previous remark. 
		
		If (3') holds then we see $\KK^{G^{\op}}=\bra\KK\ket_{\loc}^{G^{\op}}$ by the restrictions via $G^{\op}\to G^{\bi}\to G^{\op}$. 
		Thus we have $\KK^{G}=\bra\KK\ket_{\loc}^{G}$ by taking opposite algebras (see \cref{rem:componentbidbl}) and (3) follows. 
		
		Clearly (3)$\Rightarrow$(1)$\Rightarrow$(1') 
		and it is enough to deduce (3') from (1'). 
		But for any $A\in \Calg^{G^{\bi}}$, it holds in $\KK^{G^{\bi}}$ that 
		\begin{align*}
			&
			A 
			\cong A \outensor{\wh{G}^{\op}} C^r_0(\wh{G}) \outensor{G} C^r(G) 
			\\&
			\subset \bra\, A \outensor{\wh{G}^{\op}} C^r_0(\wh{G}) \outensor{G} \KK^G \,\ket_{\loc}^{G^{\bi}}
			\subset 
			\bra\, \KK^G \,\ket_{\loc}^{G^{\bi}}. 
		\qedhere\end{align*}
	\end{proof}
	
	In this proposition when the discrete quantum group $\wh{G}$ is \emph{torsion-free} in the sense that $\Rep(G)$ has only trivial ergodic Q-systems (see \cite[Definition~3.7]{Arano-deCommer:torsion}), the last condition $\KK^G=\bra\, \KK \,\ket_{\loc}^G$ 
	is considered as a version of the property that a $\gamma$-element for $\wh{G}$ exists and equals to $1$, 
	via the reformulations of the Baum--Connes conjecture by \cite{Meyer-Nest:bctri} using the decomposition of the triangulated category $\KK^{G}\simeq \KK^{\wh{G}}$. 
	Also, (4) was considered in \cite{Voigt:cpxss} as a version of the Baum--Connes property for $\sD(G)$. 
	The author does not know the property $\KK^G=\bra \KK\ket_{\loc}^{G}$ is preserved under direct products, 
	although there are such results for other variants of the Baum--Connes properties, see \cite{Martos:semidirectBC} for example. 
	But by considering (1) of the previous proposition, we can easily see the following permanence for quantum doubles. 
	\begin{cor}\label{cor:prodgamma1}
		For compact quantum groups $G_1$, $G_2$ with separable $L^2$ spaces, 
		(4) of \cref{prop:gamma1} for $G=G_1\times G_2$ holds if and only if (4) for $G=G_1$ and $G_2$ hold. 
	\qed\end{cor}
	
	We use the following lemma using the techniques of \cite{Meyer-Nest:homological,Meyer:homological} later. 
	\begin{lem}\label{lem:gamma1}
		Let $G$ be a compact quantum group with sparable $L^2(G)$ such that 
		$\KK^{G^{\bi}} = \bra \KK \ket_{\loc}^{G^{\bi}}$, 
		$A_i\in\KK^{G^{\bi}}$, and $B_i\in \KK^{\sD(G)}$ for $i=1,2$. 
		Then $\bx\in\KK^{G^{\bi}}(A_1,A_2)$ is a $\KK^{G^{\bi}}$-equivalence if $G^{\op}\redltimes\bx$ is a $\KK^{G}$-equivalence. 
		Also, $\by\in\KK^{\sD(G)}(B_1,B_2)$ is a $\KK^{\sD(G)}$-equivalence if $\Res^{\sD(G)}_{G}\by$ is a $\KK^G$-equivalence. 
	\end{lem}
	
	\begin{proof}
		We can show for $A\in \KK^{\sD(G)}$, if $\Res^{\sD(G)}_{G} A\cong 0$ in $\KK^{G}$ then $A\cong 0$ in $\KK^{\sD(G)}$, 
		because of \cite[Proposition~5.1]{Voigt:cpxss} and (3)$\Leftrightarrow$(4) of \cref{prop:gamma1}. 
		%Alternatively, this can be directly checked by using the adjointness $\Ind^{\sD(G)}_{G} \dashv \Res^{\sD(G)}_{G}$ of the triangulated functors between $\KK^{\sD(G)}$ and $\KK^G$ preserving countable direct sums. 
		Thus the second claim follows from the fact that for any exact triangle $A\xrightarrow{f} B\to C\to A[1]$ in a triangulated category, 
		$C\cong 0$ if and only if $f$ is isomorphic. 
		The first claim holds because of the natural isomorphism of the functors of $\KK^{\sD(G)}\to \KK^G$, 
		\begin{align*}
			&
			G^{\op}\redltimes\bfJ^{\id_G}(-) 
			\cong -\outensor{G} C^r_0(G) \outensor{\wh{G}^{\op}} C^r_0(\wh{G}) 
			\cong -\outensor{G}\cK(L^2(G)) 
			\cong \Res^{\sD(G)}_{G} . 
		\qedhere\end{align*} 
	\end{proof}
	
	\begin{rem}\label{rem:BCSU-q(2)}
		Let $0<q< 1$. 
		We identify 
		$T:=\left\{ \left( \begin{array}{cc}
			z&0\\0&\overline{z}
		\end{array} \right)\in SU(2) \,\bigg|\, z\in \bC \right\}$ 
		as a closed subgroup of $SU_{\pm q}(2)$ canonically. 
		It is shown in \cite[Theorem~4.7]{Voigt:bcfo} that $C(SU_q(2)/T)$ and $\bC^{\oplus 2}$ are $\KK^{\sD(SU_q(2))}$-equivalent. 
		Also, it is shown in \cite[Theorem~4.6, Proposition~4.8]{Voigt:bcfo} that there is an element $\bx\in \KK^{\sD(SU_{-q}(2))}(C(SU_{-q}(2)/T),\bC^{\oplus2})$ giving a projection onto a direct summand in $\KK^{\sD(SU_{-q}(2))}$ that is a $\KK^{SU_{-q}(2)}$-equivalence. 
		In this occasion, we briefly remark that these facts are enough to show $\KK^{\sD(SU_{-q}(2))}$-equivalence of $C(SU_{-q}(2)/T)$ and $\bC^{\oplus 2}$. 
		
		We put $G_q:=SU_{-q}(2)$. 
		Note that it follows from the fact above that $\bC^{\oplus 2}$ is a direct summand of $C(G_q^{\op}/T^{\op})$ in $\KK^{\sD(G_q^{\op})}$ and they are $\KK^{G_q^{\op}}$-equivalent 
		by taking opposite algebras. 
		We can show $\KK^{G_q^{\bi}}=\bra \KK \ket_{\loc}^{G_q^{\bi}}$ in a similar way as \cite[Theorem~6.1]{Voigt:bcfo} by using $\KK^{T^{\bi}}=\bra \KK\ket_{\loc}^{T^{\bi}}$. 
		Now \cref{lem:gamma1} shows the desired claim. 
	\end{rem}
	
	\begin{rem}
		Let $G_0$ and $G_1$ be compact quantum groups with $\Gamma_i:=\wh{G}_i$ being torsion-free for $i=0,1$. 
		Note that torsion-freeness is preserved under direct products, free products by \cite[Theorem~3.16, Theorem~3.17]{Arano-deCommer:torsion}, and taking opposite quantum groups clearly. 
		\begin{enumerate}
			\item
			If $G_0$ and $G_1$ are monoidally equivalent, so are $G_0^{\bi}$ and $G_1^{\bi}$. 
			Thus if $G_0$ satisfies (3) of \cref{prop:gamma1}, so does $G_1$ by \cite[Theorem~8.6]{Voigt:bcfo}. 
			\item
			If $\Gamma_0$ is a divisible quantum subgroup of $\Gamma_1$, so is $\Gamma_0^{\bi}$ in $\Gamma_1^{\bi}$ (see \cite[Definition~4.1, Lemma~4.2]{Vergnioux-Voigt:bcfu}). 
			Thus if $G_1$ satisfies (3) of \cref{prop:gamma1}, so does $G_0$ by \cite[Lemma~6.7]{Vergnioux-Voigt:bcfu}. 
			\item
			The properties of \cref{prop:gamma1} are also preserved under free products. 
			To see this, assume $G_0$ and $G_1$ satisfy (4) of \cref{prop:gamma1}. 
			We put $\Gamma:=\wh{G}_0\ast \wh{G}_1$ and 
			let $\iota_j\colon \wh{G}_j\to \Gamma$ denote the canonical quantum subgroup for $j=0,1$. 
			By \cite[Theorem~5.5]{Vergnioux-Voigt:bcfu} and remarks after \cite[Definition~5.1]{Vergnioux-Voigt:bcfu}, there is a left $\sD(\Gamma)$-C*-algebra $\mathcal{P}$ that is contained in 
			\[ \bra A\otimes C \,|\, A\in \{c_0(\Gamma/\wh{G}_0), c_0(\Gamma/\wh{G}_1), c_0(\Gamma)\}, C\in \KK \ket_{\loc}^{\sD(\Gamma)}, \] 
			and contains $\bC$ as a direct summand in the triangulated category $\KK^{\sD(\Gamma)}$. 
			Thus (2) of \cref{prop:gamma1} for $\wh{\Gamma}$ holds because 
			$c_0(\Gamma)= \Ind^{\sD(\Gamma)}_{\wh{\Gamma}}\bC$ and 
			for $i=0,1$, 
			\begin{align*}
				&
				c_0(\Gamma/\Gamma_i) = \Ind^{\sD(\Gamma)}_{\sD(\iota_i)} \Res^{\sD(\Gamma_i)}_{\sD(\iota_i)} \bC_{\id_{\Gamma_i}} 
				\in 
				\bra \Ind^{\sD(\Gamma)}_{\sD(\iota_i)} \Res^{\sD(\Gamma_i)}_{\sD(\iota_i)} \Ind^{\sD(\Gamma_i)}_{G_i}\KK^{G_i} \ket_{\loc}^{\sD(\Gamma)}
				\\&
				= 
				\bra \Ind^{\sD(\Gamma)}_{\sD(\iota_i)} \Ind^{\sD(\iota_i)}_{\wh{\Gamma}} \Res^{G_i}_{\wh{\Gamma}}\KK^{G_i} \ket_{\loc}^{\sD(\Gamma)} 
				\subset 
				\bra \Ind^{\sD(\Gamma)}_{\wh{\Gamma}} \KK^{\wh{\Gamma}} \ket_{\loc}^{\sD(\Gamma)} , 
			\end{align*}
			with the aid of \cite[Theorem~5.3, Remark~5.7]{Kitamura:indhomlcqg}, where we considered restrictions via $\wh{\iota}_i\colon \wh{\Gamma}\to G_i$ and the homomorphism $\sD(\iota_i)\to \sD(\Gamma_i)$ induced by $\wh{\iota}_i$. 
		\end{enumerate}
		
		Therefore, by following the strategy of \cite{Voigt:bcfo,Vergnioux-Voigt:bcfu}, we obtain from the case of $SU_q(2)$ (cf.~\cref{rem:BCSU-q(2)}) that free orthogonal quantum groups and free unitary quantum groups satisfy the conditions of \cref{prop:gamma1}. 
	\end{rem}
	
	\section{Equivariant \texorpdfstring{$K$}{K}-theory of quantum homogeneous spaces}\label{sec:grtwhmg}
	
	We determine equivariant $KK$-theoretic classes of certain quantum homogeneous spaces of graded twisting in the sense of \cite{Neshveyev-Yamashita:twilie,Bichon-Neshveyev-Yamashita:grtwi,Bichon-Neshveyev-Yamashita:grmod}. 
	We mainly deal with the following situation. 
	\begin{cond}\label{cond:actionfinab}
		Let $F,G,H$ be compact quantum groups with separable $L^2$-spaces and $\Gamma$ be a finite abelian group. Consider the following sequences 
		\begin{align*}
			\xymatrix@R=0.4em{
				1\ar[r] & 
				\wh{\Gamma} \ar@<0.3ex>[r]^-{\imath} & 
				H \ar[r]^-{\phi}\ar@<0.3ex>[l]^-{\pi} & 
				G \ar[r] & 
				1, 
				\\
				1\ar[r] & 
				\wh{\Gamma} \ar@<0.3ex>[r]^-{\jmath} & 
				H \ar[r]^-{\varphi}\ar@<0.3ex>[l]^-{\pi} & 
				F \ar[r] & 
				1, 
			}
		\end{align*}
		which are ``split exact" (with common $\pi$) in the sense of satisfying the three conditions as follows. 
		\begin{itemize}
			\item
			$\imath, \wh{\phi}, \jmath, \wh{\varphi}$ give closed quantum subgroups. 
			Via $\phi$ and $\varphi$, we shall identify $C^r(G)$ and $C^r(F)$ as unital C*-subalgebras of $C^r(H)$ compatible with the comultiplications. 
			\item
			Under these identifications, 
			$(\imath^*C^r_0(H))^{\wh{\Gamma}}=C^r_0(G)$, 
			and $(\jmath^*C^r_0(H))^{\wh{\Gamma}}=C^r_0(F)$. 
			\item
			$\pi\imath=\id_{\wh{\Gamma}}=\pi\jmath$. 
		\end{itemize}
	\end{cond}
	
	\begin{rem}
		For $t\in\Gamma$, we write $u^t\in C(\wh{\Gamma})$ for the unitary $\ell^2(\Gamma)\ni \xi \mapsto \xi(t^{-1}(-))\in \ell^2(\Gamma)$. 
		Under the condition above, we have a $*$-homomorphism $\pi^r\colon C(\wh{\Gamma})\to C^r(H)$ such that $W^\pi=(\pi^r\otimes\id)(W^{\wh{\Gamma}})$ by coamenability and \cite{Meyer-Roy-Woronowicz:qghom}, 
		and by abusing notation to identify $u^t$ with $\pi^r(u^t)$, 
		we have $W^\pi=\sum\limits_{t\in\Gamma}u^{t*}\otimes\delta_t$. 
		It is not hard to see $\Ad W^{\pi}$ is a matching on $G^{\op}$ and $\Gamma$ (see \cite[Remark~7.2]{Kitamura:indhomlcqg}). 
		Especially for each $t\in \Gamma$, we see $\alpha_t:=\Ad u^t$ gives an $*$-automorphism of $C^r(G)$ preserving $\Delta_G$, and by comparing the Haar states of $G$ and $H$ we see $C^r(H)=\Gamma\redltimes C^r(G)$. 
		We let $p:=(R^{\wh{\Gamma}}\otimes\epsilon_H)(\jmath^*\Delta_H) \colon \cO(G)\subset \cO(H)\to C(\wh{\Gamma})$ be a unital $*$-homomorphism preserving comultiplications, where $\epsilon_H\colon \cO(H)\to \bC$ is the counit. 
		Then the pair $(p,\alpha)$ is an \emph{invariant cocentral action} of $\Gamma$ on $\cO(G)$ in the sense of the beginning of \cite[Section~2.2]{Bichon-Neshveyev-Yamashita:grtwi} preserving $*$-structures and 
		we can see that $F=G^{t,(p,\alpha)}$ is the \emph{graded twisting} in the sense of the beginning of \cite[Section~3.1]{Bichon-Neshveyev-Yamashita:grtwi}. 
		Conversely, a graded twisting of $G$ by an invariant cocentral action of $\Gamma$ fits the condition above thanks to \cite[Proposition~2.7, Proposition~3.1]{Bichon-Neshveyev-Yamashita:grtwi}. 
	\end{rem}
	
	\begin{rem}\label{rem:bfF}
		Under \cref{cond:actionfinab}, we see $\wh{\Gamma}\rbicross{\Ad W^{\pi}_{21}} G^{\op}\cong H^{\op}\cong \wh{\Gamma}\rbicross{\Ad W^{\pi}_{21}} F^{\op}$. 
		We shall write $\wt{G} := G\dblcross{\Ad W^\pi}\Gamma$ and $\wt{F}:=F\dblcross{\Ad W^{\pi}}\Gamma$. 
		Then we have equivalences of triangulated categories $\KK^{\wt{G}}\simeq \KK^H\simeq \KK^{\wt{F}}$ by using \cref{thm:BSTTdualcat} and \cref{rem:dualityopcop} twice. 
		But we would like to use the concrete construction of an equivalence 
		$\KK^{\wt{G}}\simeq\KK^{\wt{F}}$ described in 
		\cite[Theorem~7.4]{Kitamura:indhomlcqg}, a $KK$-theoretic version of \cite[Theorem~3.6]{Bichon-Neshveyev-Yamashita:grmod}. 
		We have a categorical equivalence 
		$\Calg^{\wt{G}} \simeq \Calg^{\wt{F}}$ that sends 
		a separable left $\wt{G}$-C*-algebra $A$ with its left $G$-action $\alpha$ and left $\Gamma$-action $\lambda$ to 
		\begin{align*}
			&
			\bfF(A):=
			\bigl[ u^{r*}_1 \lambda(a) \ \big|\ 
			r\in\Gamma, a\in A, 
			(\jmath^*\alpha)(a)\in u^{r}\otimes A \bigr] 
			\subset \Gamma\redltimes A \cong A\outensor{\Gamma} C(\wh{\Gamma}), 
		\end{align*}
		and the restriction of its left $\wt{F}$-action via the canonical homomorphism $\sD(\pi)\to\wt{F}$ is the left $\sD(\pi)$-action on $A\outensor{\wh{\Gamma}} C(\wh{\Gamma})$, 
		where we consider $C(\wh{\Gamma})$ as a left $\wh{\Gamma}$-YD C*-algebra with the trivial $\Gamma$-action. 
		This functor induces an equivalence 
		$\bfF\colon \KK^{\wt{G}}\xrightarrow{\sim} \KK^{\wt{F}}$ 
		of triangulated categories. 
	\end{rem}
	
	We want to see what data of $\KK^{\sD(F)}$ can be conveyed from the information about $G$. 
	Our strategy is to compare $\KK^{\wt{G}^{\bi}}$ and $\KK^{\wt{F}^{\bi}}$ instead of $\KK^{\sD(\wt{G})}$ and $\KK^{\sD(\wt{F})}$ directly. 
	We can do this since the sequences of \cref{cond:actionfinab} induces the sequences 
	\begin{align*}
		\xymatrix@R=0.4em{
			1\ar[r] & 
			(\wh{\Gamma})^{\bi} \ar@<0.3ex>[r]^-{\imath^{\bi}} & 
			H^{\bi} \ar[r]^-{\phi^{\bi}}\ar@<0.3ex>[l]^-{\pi^{\bi}} & 
			G^{\bi} \ar[r] & 
			1, 
			\\
			1\ar[r] & 
			(\wh{\Gamma})^{\bi} \ar@<0.3ex>[r]^-{\jmath^{\bi}} & 
			H^{\bi} \ar[r]^-{\varphi^{\bi}}\ar@<0.3ex>[l]^-{\pi^{\bi}} & 
			F^{\bi} \ar[r] & 
			1, 
		}
	\end{align*}
	which again satisfy \cref{cond:actionfinab} via the identification $\wh{\Gamma^{\bi}}\cong (\wh{\Gamma})^{\bi}$. 
	We note that $W^{\pi^{\op}}=W^{\pi} \in \cU\cM(C^r(H)\otimes C(\Gamma))$ because $\Gamma$ is an abelian group, and 
	$\wt{G}^{\op}\cong G^{\op}\dblcross{\Ad W^{\pi}}\Gamma^{\op}$ given by 
	\begin{align}\label{eq:wtGopvsrDpi}
		&
		\begin{aligned}
			&
			C^r(\wt{G}^{\op})
			= C^r(G)\otimes C(\Gamma)
			\ni x\otimes \delta_t 
			\\&
			\mapsto 
			\Ad W^{\pi*} (x\otimes \delta_{t^{-1}})
			\in C^r(G)\otimes C(\Gamma)
			= C^r(G^{\op}\dblcross{\Ad W^{\pi}}\Gamma^{\op}) . 
		\end{aligned}
	\end{align}
	By \cref{rem:bfF}, we have the categorical equivalence 
	$\Calg^{\wt{G}^{\bi}} \simeq \Calg^{\wt{F}^{\bi}}$ that sends 
	a separable left $(G\dblcross{\Ad W^{\pi}}\Gamma) \times (G^{\op}\dblcross{\Ad W^{\pi}}\Gamma^{\op})$-C*-algebra $A$ 
	with its left $G$-action $\alpha$, left $\Gamma$-action $\lambda$, left $G^{\op}$-action $\beta$, and left $\Gamma^{\op}$-action $\mu$ to 
	\begin{align}\label{eq:bfFbi}
		\bfF^{\bi}(A)
		:=&
		\left[ u^{r*}_1 u^{s*}_2 (\id\otimes\mu)\lambda(a) \ \left|\ 
		\begin{array}{c}
			r,s\in\Gamma,\ a\in A, 
			\\
			(\jmath^*\alpha)(a)\in u^{r}\otimes A,\ 
			(\jmath^{\op *}\beta)(a)\in u^{s}\otimes A
		\end{array}\right.
		\right] 
		\subset \Gamma^{\bi}\redltimes A , 
	\end{align}
	and the restriction of its left $\wt{F}^{\bi}$-action via the homomorphism $\sD(\pi^{\bi})\cong \sD(\pi)^{\bi}\to \wt{F}^{\bi}$ 
	is the left $\sD(\pi^{\bi})$-action of $\Gamma^{\bi}\redltimes A\cong A\outensor{\Gamma^{\bi}}C(\wh{\Gamma}^{\bi})$. 
	This functor induces an equivalence $\bfF^{\bi}\colon \KK^{\wt{G}^{\bi}}\xrightarrow{\sim} \KK^{\wt{F}^{\bi}}$ 
	of triangulated categories. 
	
	We define the compact quantum group 
	$G^{\bi}_{\Gamma}$ by 
	$C^r(G^{\bi}_{\Gamma}) := C^r(G)\otimes C(\Gamma)\otimes C^r(G^{\op})$ 
	and for $x\in C^r(G), y\in C^r(G^{\op}), s\in\Gamma$, 
	\begin{align*}
		&
		\Delta_{G^{\bi}_{\Gamma}} (x\otimes \delta_s \otimes y) 
		:= \sum\limits_{t\in\Gamma} 
		(u^{t}\otimes u^{s^{-1}t})_{43}^*
		\bigl( \Delta_G(x)\otimes \delta_t\otimes\delta_{st^{-1}}\otimes \Delta_G(y)
		\bigr)_{142563}
		(u^{t}\otimes u^{s^{-1}t})_{43} . 
	\end{align*}
	Then $C^r(G)$ is a left $G^{\bi}_{\Gamma}$-C*-algebra by sending $x\in C^r(G)$ to 
	\begin{align*}
		&
		\sum\limits_{t\in\Gamma} (\delta_t)_2 u^{t\otimes 2*}_{34} \bigl((\id\otimes\Delta_G)\Delta_G(x)\bigr)_{143} u^{t\otimes2}_{34} \in C^r(G^{\bi}_{\Gamma})\otimes C^r(G). 
	\end{align*}
	We can regard $G^{\bi}_{\Gamma}$ as a closed quantum subgroup of $\wt{G}^{\bi}$ via 
	\begin{align*}
		&
		C^r(\wt{G})\otimes C^r(\wt{G}^{\op})\ni 
		x\otimes \delta_s\otimes y\otimes \delta_t 
		\mapsto 
		\delta_{s,t^{-1}} x\otimes \delta_s\otimes y 
		\in C^r(G^{\bi}_{\Gamma}). 
	\end{align*}
	We have canonical closed quantum subgroups $G^{\bi}\subset G^{\bi}_{\Gamma}\subset \wt{G}^{\bi}$, and 
	projections $G^{\bi}_{\Gamma}\to \wt{G}$ and $G^{\bi}\to G$ that ignore the component $G^{\op}$. 
	We will use induction and restriction along these homomorphisms concerning $G$, and similarly for $F$. 
	Then we are going to show the following. 
	
	\begin{thm}\label{thm:grtwtransfer}
		Under \cref{cond:actionfinab}, take $A\in\Calg^{\wt{G}^{\bi}}$, $C\in \Calg$ such that $\Res^{\wt{G}^{\bi}}_{G^{\bi}_{\Gamma}}A \cong C\otimes C^r(G)$ in $\KK^{G^{\bi}_{\Gamma}}$. 
		If $\KK^{F^{\bi}}=\bra\, \KK \,\ket_{\loc}^{F^{\bi}}$, then 
		$\Res^{\wt{F}^{\bi}}_{F^{\bi}}\bfF^{\bi}(A)\cong C\otimes C^r(F)$ 
		in $\KK^{F^{\bi}}$. 
	\end{thm}
	
	We have a natural transformations 
	$\eta\colon \id_{\KK^{\wt{G}^{\bi}}}
	\to 
	\Ind^{\wt{G}^{\bi}}_{G^{\bi}_{\Gamma}} \Res^{\wt{G}^{\bi}}_{G^{\bi}_{\Gamma}}$ and 
	$\kappa\colon \Res^{\wt{G}^{\bi}}_{G^{\bi}_{\Gamma}} \Ind^{\wt{G}^{\bi}}_{G^{\bi}_{\Gamma}} 
	\to \id_{\KK^{G^{\bi}_{\Gamma}}}$
	corresponding to the adjunction of the functors 
	$\Res^{\wt{G}^{\bi}}_{G^{\bi}_{\Gamma}} \dashv \Ind^{\wt{G}^{\bi}}_{G^{\bi}_{\Gamma}}$ 
	by \cite[Proposition~6.6]{Kitamura:indhomlcqg}. 
	
	\begin{lem}\label{lem:grtwitransfer1}
		Under \cref{cond:actionfinab}, we have the following commutative diagram of functors up to some natural isomorphisms, 
		\begin{align*}
			&
			\xymatrix{
				\Calg^{\wt{G}^{\bi}} 
				\ar[d]_-{\Res^{\wt{G}^{\bi}}_{G^{\bi}_{\Gamma}}} \ar[r]^-{\bfF^{\bi}} &
				\Calg^{\wt{F}^{\bi}} 
				\ar[d]_-{\Res^{\wt{F}^{\bi}}_{F^{\bi}_{\Gamma}}} 
				\ar[dr]^-{\Res^{\wt{F}^{\bi}}_{F^{\bi}}}&
				\\
				\Calg^{G^{\bi}_{\Gamma}} 
				\ar[d]_-{\Ind^{\wt{G}}_{G^{\bi}_{\Gamma}}} & 
				\Calg^{F^{\bi}_{\Gamma}} 
				\ar[d]_-{\Ind^{\wt{F}}_{F^{\bi}_{\Gamma}}} &
				\Calg^{F^{\bi}} \ar[d]^-{\Ind^F_{F^{\bi}}}
				\\
				\Calg^{\wt{G}} \ar[r]^-{\bfF} & 
				\Calg^{\wt{F}} \ar[r]^-{\Res^{\wt{F}}_F}&
				\Calg^{F}. 
			}
		\end{align*}
		Moreover, when we put 
		\begin{align*}
			&
			\bfG:= F^{\op}\redltimes \Res^{\wt{F}^{\bi}}_{F^{\bi}} \bfF^{\bi} 
			\colon \KK^{\wt{G}^{\bi}}\to \KK^{F}, 
			\\&
			\bfH:=\Res^{\wt{F}}_{F} \bfF (\wh{\wt{G}}^{\op}\redltimes (G^{\bi}_{\Gamma})\redltimes-)
			\colon \KK^{G^{\bi}_{\Gamma}}\to \KK^F, 
		\end{align*}
		we have a natural isomorphism $\theta\colon  \bfG 
		\xrightarrow{\sim} 
		\bfH \Res^{\wt{G}^{\bi}}_{G^{\bi}_{\Gamma}}$ defined in the proof. 
	\end{lem}
	
	\begin{proof}
		First we consider the natural isomorphisms of the diagram. 
		Note that we have a left $F$-$*$-isomorphism 
		$C^r(\wt{F}) 
		\cong 
		\Ind^F_{F^{\bi}} C^r(F^{\bi}_{\Gamma})$ and 
		it follows 
		$\Res^{\wt{F}}_F \Ind^{\wt{F}}_{F^{\bi}_{\Gamma}} 
		\cong \Ind^F_{F^{\bi}} \Res^{F^{\bi}_{\Gamma}}_{F^{\bi}}$ 
		by \cite[Theorem~5.3, Remark~5.7]{Kitamura:indhomlcqg}. 
		From this we see 
		$\Ind^{F}_{F^{\bi}} \Res^{\wt{F}^{\bi}}_{F^{\bi}} 
		\cong 
		\Res^{\wt{F}}_{F} \Ind^{\wt{F}}_{F^{\bi}_{\Gamma}} \Res^{\wt{F}^{\bi}}_{F^{\bi}_{\Gamma}}$. 
		
		For $(A,\alpha,\lambda,\beta,\mu)$ as in \cref{eq:bfFbi}, 
		we have 
		\begin{align*}
			&
			\Ind^{\wt{F}}_{F^{\bi}_{\Gamma}} \Res^{{\wt{F}}^{\bi}}_{F^{\bi}_{\Gamma}} \bfF^{\bi}(A)
			\\&
			=
			\Ind^{\wt{F}}_{F^{\bi}_{\Gamma}}
			\left[ u^{r*}_1 u^{s*}_2 (\id\otimes\mu)\lambda(a) \ \left|\ 
			\begin{array}{c}
				r,s\in\Gamma,\ a\in A, 
				\\
				(\jmath^*\alpha)(a)\in u^{r}\otimes A,\ 
				(\jmath^{\op *}\beta)(a)\in u^{s}\otimes A
			\end{array}\right.
			\right] 
			\\&
			\cong
			(\id\otimes\mu)
			\bigl[ u^{r*}_1 \lambda(a) \ \big|\ 
			r\in\Gamma, a\in A, 
			(\jmath^*\alpha)(a)\in u^{r}\otimes A, \beta(a)\in 1\otimes A \bigr], 
		\end{align*}
		since as subspaces of $C^r(H)$ we have $C(\wh{\Gamma})\cap C^r(G)=\bC1$. 
		Similarly, we have
		\begin{align*}
			&
			\bfF (\Ind^{\wt{G}}_{G^{\bi}_{\Gamma}} \Res^{{\wt{G}}^{\bi}}_{G^{\bi}_{\Gamma}} A)
			\\&
			\cong
			\bigl[ u^{r*}_1 \lambda(a) \ \big|\ 
			r\in\Gamma, a\in A, 
			(\jmath^*\alpha)(a)\in u^{r}\otimes A, \beta(a)\in 1\otimes A \bigr], 
		\end{align*} 
		and $\id\otimes \mu$ gives their left $\wt{F}$-$*$-isomorphism. 
		It is easy to check the naturality. 
		By combining the constructions so far, we get the desired natural isomorphism 
		\begin{align}\label{prf:lem:grtwitransfer1}
			&
			\theta'\colon  
			\Ind^{F}_{F^{\bi}} \Res^{{\wt{F}}^{\bi}}_{F^{\bi}} \bfF^{\bi} 
			\xrightarrow{\sim} 
			\Res^{\wt{F}}_{F} \bfF \Ind^{\wt{G}}_{G^{\bi}_{\Gamma}} \Res^{{\wt{G}}^{\bi}}_{G^{\bi}_{\Gamma}}
		\end{align}
		as the functors of 
		$\Calg^{{\wt{G}}^{\bi}}\to \Calg^{F}$. 
		These functors give split exact homotopy functors to $\KK^F$. 
		By \cite[Remark~6.5]{Kitamura:indhomlcqg} and 
		$\Res^{\wt{G}}_{G}\Ind^{\wt{G}}_{G^{\bi}_{\Gamma}}\cong \Ind^{G}_{G^{\bi}}\Res^{G^{\bi}_{\Gamma}}_{G^{\bi}}$, 
		we see the functors in \cref{prf:lem:grtwitransfer1} preserve full corner embeddings on the full subcategory of $\Calg^{\wt{G}^{\bi}}$ of the objects whose left $G^{\op}$-actions are free. 
		
		Thus the compositions of the left and right hand sides of \cref{prf:lem:grtwitransfer1} after $\wh{\wt{G}^{\bi}}^{\op}\redltimes \wt{G}^{\bi}\redltimes -\colon \Calg^{\wt{G}^{\bi}}\to \Calg^{\wt{G}^{\bi}}$ 
		induce triangulated functors $\bfG',\bfH'\colon \KK^{\wt{G}^{\bi}}\to \KK^{F}$, respectively, and they are naturally isomorphic via $\theta'$. 
		It is not hard to see $\bfG'\cong \bfG$ and $\bfH'\cong\bfH\Res^{\wt{G}^{\bi}}_{G^{\bi}_{\Gamma}}$. 
		We define $\theta$ as the composition of these natural isomorphisms. 
	\end{proof}
	
	\begin{lem}\label{lem:grtwitransfer2}
		In the situation of \cref{cond:actionfinab}, let $C$ be a separable C*-algebra. 
		Then we have a left $F^{\bi}$-$\ast$-homomorphism 
		$\rho \colon  \Res^{\wt{F}^{\bi}}_{F^{\bi}} \bfF^{\bi} \Ind^{{\wt{G}}^{\bi}}_{G^{\bi}_{\Gamma}} (C^r(G))
		\to C^r(F) $ 
		which fits into the following commutative diagram in $\KK^{F}$, 
		\begin{align*}
			&\xymatrix@C=4.5em{
				\bfG \Ind^{{\wt{G}}^{\bi}}_{G^{\bi}_{\Gamma}} (C\otimes C^r(G)) 
				\ar[r]^-{\theta_{\Ind (C\otimes C^r(G))}}
				\ar[d]_-{F^{\op}\ltimes(\id_C\otimes \rho)}&
				\bfH \Res^{\wt{G}^{\bi}}_{G^{\bi}_{\Gamma}} \Ind^{\wt{G}^{\bi}}_{G^{\bi}_{\Gamma}} (C\otimes C^r(G)) 
				\ar[d]^-{ \bfH(\kappa_{C\otimes C^r(G)}) }
				\\
				C\otimes\cK(F) 
				\ar[r]^-{\sim}&
				\bfH (C\otimes C^r(G)) .
			}
		\end{align*}
	\end{lem}
	
	We will identify $\Ind^{\wt{G}^{\bi}}_{G^{\bi}_{\Gamma}} C^r(G)$ 
	as $C^r(\wt{G})$ via the left $\wt{G}^{\bi}$-$*$-isomorphisms 
	$\Ind^{\wt{G}^{\bi}}_{G^{\bi}_{\Gamma}} C^r(G) \cong C^r(\wt{G}^{\bi})\cotensor{G^{\bi}_{\Gamma}}C^r(G)$ by \cite[Theorem~4.7]{Kitamura:indhomlcqg} 
	and $C^r(\wt{G})\xrightarrow{\sim} C^r(\wt{G}^{\bi})\cotensor{G^{\bi}_{\Gamma}}C^r(G)$ 
	that sends $x\otimes\delta_s\in C^r(\wt{G})$ for each $x\in C^r(G)$ and $s\in \Gamma$ to 
	\begin{align*}
		&
		\sum\limits_{t\in\Gamma} (\delta_t\otimes\delta_{t^{-1}s})_{24} u^{t\otimes2*}_{35} 
		\bigl((\id\otimes\Delta_G)\Delta_G(x)\bigr)_{153} u^{t\otimes2}_{35} 
		\\&
		\in C^r(\wt{G}^{\bi})\cotensor{G^{\bi}_{\Gamma}}C^r(G) 
		\subset 
		C^r(\wt{G})\otimes C^r(\wt{G}^{\op})\otimes C^r(G). 
	\end{align*}
	
	\begin{proof}
		Via \cref{eq:wtGopvsrDpi} the left $G^{\op}\dblcross{\Ad W^{\pi}}\Gamma^{\op}$-action is 
		\begin{align*}
			&
			C^r(\wt{G}) 
			= C^r(G)\otimes C(\Gamma) 
			\ni x\otimes \delta_s
			\\&
			\mapsto 
			\sum\limits_{t\in\Gamma} 
			\Ad (W^{\pi*}_{12}W^{\pi}_{14}) \bigl(\Delta_{G}(x)\otimes \delta_{st}\otimes\delta_t\bigr)_{3142}
			\in C(G^{\op}\dblcross{\Ad W^{\pi}}\Gamma^{\op})\otimes C^r(\wt{G}) . 
		\end{align*}
		Since for each $r\in\Gamma$, 
		\begin{align*}
			&
			[ x\in C^r(H) \,|\, (\jmath^*\Delta_H)(x)\in u^{r*}\otimes C^r(H) ]
			\\&
			=
			u^{r*} [ x\in C^r(H) \,|\, (\jmath^*\Delta_H)(x)\in 1\otimes C^r(H) ] 
			=
			u^{r*} C^r(F) 
			\\&
			=
			[ x\in C^r(H) \,|\, (\Delta_H\rpb\jmath) (x)\in C^r(H)\otimes u^{r*} ], 
		\end{align*}
		we see 
		\begin{align*}
			&
			\bfF^{\bi}(C^r(\wt{G})) 
			=
			\biggl[\, 
			\sum\limits_{t,t'\in\Gamma} (u^r\delta_{t})\otimes (u^r\delta_{t'})\otimes (u^{t*} xu^{t})\otimes \delta_{t^{-1}st'} 
			\,\bigg|\,
			r,s\in \Gamma, x\in C^r(G)\cap u^{r*}C^r(F)
			\,\biggr], 
		\end{align*}
		and its left $F$-action and its left $F^{\op}$-action sends
		\begin{align*}
			&
			\sum\limits_{t,t'\in\Gamma} (u^r\delta_{t})\otimes (u^r\delta_{t'})\otimes (u^{t*} xu^{t})\otimes \delta_{t^{-1}st'} 
		\end{align*}
		to 
		\begin{align*}
			&
			u^r_1 \sum\limits_{t,t'\in\Gamma} u^{t}_1 
			\bigl( (u^r\delta_{t})\otimes (u^r\delta_{t'})\otimes \Delta_G(u^{t*} xu^{t})\otimes \delta_{t^{-1}st'} \bigr)_{23145} u^{t*}_1
			\\&
			=
			u^r_1 \sum\limits_{t,t'\in\Gamma} u^{t*}_4 
			\bigl( (u^r\delta_{t})\otimes (u^r\delta_{t'})\otimes \Delta_G(x)\otimes \delta_{t^{-1}st'} \bigr)_{23145} u^{t}_4
		\end{align*}
		and 
		\begin{align*}
			&
			u^r_1 \sum\limits_{t,t'\in\Gamma} u^{t'}_1 u^{t^{-1}st'*}_1 
			\bigl( (u^r\delta_{t})\otimes (u^r\delta_{t'})\otimes \Delta_G(u^{t*} xu^{t})\otimes \delta_{t^{-1}st'} \bigr)_{23415} u^{t^{-1}st'}_1 u^{t'*}_1 
			\\&
			=
			u^r_1 \sum\limits_{t,t'\in\Gamma} u^{s*}_1 u^{t*}_4
			\bigl( (u^r\delta_{t})\otimes (u^r\delta_{t'})\otimes \Delta_G( x )\otimes \delta_{t^{-1}st'} \bigr)_{23415} u^{t}_4 u^{s}_1 , 
		\end{align*}
		respectively. 
		Thus we have a $F^{\bi}$-invariant central projection 
		\begin{align*}
			&
			p := 
			\sum\limits_{t,t'\in\Gamma} \delta_t \otimes \delta_{t'} \otimes 1 \otimes \delta_{t^{-1}t'} 
			\in \cZ(\bfF^{\bi}(C^r(\wt{G}))^{F^{\bi}}), 
		\end{align*}
		and a left $F^{\bi}$-$*$-isomorphism $C^r(F)\to p\bfF^{\bi}(C^r(\wt{G}))p$ given by 
		the restriction of the well-defined injective $*$-homomorphism 
		$C^r(H)=\Gamma\redltimes C^r(G)\to \Gamma^{\bi}\redltimes C^r(\wt{G})$ 
		that sends $u^s x$ for $s\in \Gamma$ and $x\in C^r(G)$ to 
		\begin{align*}
			&
			\sum\limits_{t,t'\in\Gamma} (u^s\delta_{t})\otimes (u^s\delta_{t'})\otimes (u^{t*} xu^{t})\otimes \delta_{t^{-1}t'} . 
		\end{align*} 
		Thus we get a surjective left $F^{\bi}$-$*$-homomorphism $\rho\colon \bfF^{\bi}(C^r(\wt{G}))\to C^r(F)$. 
		Since $\kappa_{C\otimes C^r(G)}\colon C^r(\wt{G})=C^r(G)\otimes C(\Gamma)\to C^r(G)$ is the evaluation at $e\in \Gamma$ by its construction, we can check the commutativity of the following diagram 
		\begin{align*}
			&\xymatrix@C=4.5em{
				\Ind^{F}_{F^{\bi}} \Res^{\wt{F}^{\bi}}_{F^{\bi}} \bfF^{\bi} (C\otimes C^r(\wt{G}))
				\ar[r]^-{\theta'_{ C\otimes C^r(\wt{G}) }}
				\ar[d]_-{\Ind^{F}_{F^{\bi}}(\id_C\otimes \rho)}&
				\Res^{\wt{F}}_{F}\bfF\Ind^{\wt{G}}_{G^{\bi}_{\Gamma}} \Res^{\wt{G}^{\bi}}_{G^{\bi}_{\Gamma}} (C\otimes C^r(\wt{G})) 
				\ar[d]^-{ \Res\bfF\Ind(\kappa_{C\otimes C^r(G)}) }
				\\
				C 
				\ar[r]^-{\sim}&
				\Res^{\wt{F}}_{F}\bfF\Ind^{\wt{G}}_{G^{\bi}_{\Gamma}} (C\otimes C^r(G)) .
			}
		\end{align*}
		Now since the left $G^{\op}$-actions of $C^r(\wt{G})$ and $C^r(G)$ are free, we have the desired commutative diagram in $\KK^F$ by the argument after \cref{prf:lem:grtwitransfer1} and the definition of $\theta$. 
	\end{proof}
	
	\begin{proof}[Proof of {\cref{thm:grtwtransfer}}]
		We take an isomorphism $\bx\in\KK^{G^{\bi}_{\Gamma}}(A, C\otimes C^r(G))$ by assumption and put 
		$\by\in \KK^{F^{\bi}}(\bfF^{\bi}(A), C\otimes C^r(F))$ as the following composition, 
		\begin{align*}
			&
			\bfF^{\bi}(A)
			\xrightarrow{\bfF^{\bi}\eta_A} 
			\bfF^{\bi} \Ind^{{\wt{G}}^{\bi}}_{G^{\bi}_{\Gamma}} \Res^{{\wt{G}}^{\bi}}_{G^{\bi}_{\Gamma}} (A) 
			\\&
			\xrightarrow{\bfF^{\bi}\Ind^{{\wt{G}}^{\bi}}_{G^{\bi}_{\Gamma}}(\bx)} 
			\bfF^{\bi} \Ind^{{\wt{G}}^{\bi}}_{G^{\bi}_{\Gamma}} (C\otimes C(G)) 
			\\&
			\xrightarrow{\id_C\otimes \rho} 
			C\otimes C^r(F). 
		\end{align*}
		It suffices to show $\by$ is isomorphic, but 
		by $\KK^{F^{\bi}}= \bra\, \KK \,\ket_{\loc}^{F^{\bi}}$ and \cref{lem:gamma1}, 
		we only have to show $F^{\op}\redltimes\by \in \KK^F(F^{\op}\redltimes\Res^{\wt{F}^{\bi}}_{F^{\bi}} \bfF^{\bi}(A),C\otimes\cK(F))$ is isomorphic. 
		Thus it is enough to show the following diagram in $\KK^{F}$ commutes, 
		\begin{align*}
			&\xymatrix{
				\bfG(A) 
				\ar[d]_-{\bfG (\eta_A) }
				\ar[r]^-{\theta_A} &
				\bfH \Res^{\wt{G}^{\bi}}_{G^{\bi}_{\Gamma}} (A) 
				\ar[d]^-{ \bfH(\bx) }
				\\
				\bfG \Ind^{{\wt{G}}^{\bi}}_{G^{\bi}_{\Gamma}} \Res^{{\wt{G}}^{\bi}}_{G^{\bi}_{\Gamma}} (A) 
				\ar[d]_-{ \bfG \Ind (\bx) } &
				\bfH (C\otimes C^r(G)) 
				\\
				\bfG \Ind^{{\wt{G}}^{\bi}}_{G^{\bi}_{\Gamma}} (C\otimes C^r(G)) 
				\ar[r]^-{F^{\op}\ltimes(\id\otimes \rho)}
				&
				C\otimes \cK(F) 
				\ar[u]_-{\wr} . 
			}
		\end{align*}
		
		By \cref{lem:grtwitransfer2}, this reduces to the commutativity of the following diagram, 
		\begin{align*}
			\xymatrix@C=-3em@R=3em{
				\bfG(A) 
				\ar[d]_-{\bfG (\eta_A)}
				\ar[rrr]^-{\theta_A} &&&
				\bfH \Res^{\wt{G}^{\bi}}_{G^{\bi}_{\Gamma}} (A) 
				\ar[d]^-{\bfH(\bx)}
				\ar[dl]_-{\bfH \Res (\eta_{A}) \qquad}
				\\
				\bfG \Ind^{{\wt{G}}^{\bi}}_{G^{\bi}_{\Gamma}} \Res^{{\wt{G}}^{\bi}}_{G^{\bi}_{\Gamma}} (A) 
				\ar[d]_-{ \bfG \Ind (\bx) } 
				%\ar[rr]^-{\theta_{\Ind\Res A}}&\qquad\qquad\qquad\qquad\quad&
				&\qquad\qquad\qquad&
				\bfH \Res^{{\wt{G}}^{\bi}}_{G^{\bi}_{\Gamma}} \Ind^{{\wt{G}}^{\bi}}_{G^{\bi}_{\Gamma}} \Res^{{\wt{G}}^{\bi}}_{G^{\bi}_{\Gamma}} (A) 
				\quad
				\ar[dr]_-{\bfH \Res\Ind (\bx) \qquad }&
				\bfH (C\otimes C^r(G))
				\\
				\bfG \Ind^{{\wt{G}}^{\bi}}_{G^{\bi}_{\Gamma}} (C\otimes C^r(G)) 
				\ar[rrr]_-{\theta_{\Ind (C\otimes C^r(G))}}&&& 
				\bfH \Res^{\wt{G}^{\bi}}_{G^{\bi}_{\Gamma}} \Ind^{\wt{G}^{\bi}}_{G^{\bi}_{\Gamma}} (C\otimes C^r(G)) 
				\ar[u]_-{ \bfH(\kappa_{C\otimes C^r(G)}) }. 
			}
		\end{align*}
		The left part commutes by the naturality of $\theta$. 
		The right part commutes because of the equality 
		\begin{align*}
			&
			\kappa_{C\otimes C^r(G)} 
			\circ \bigl( \Res^{\wt{G}^{\bi}}_{G^{\bi}_{\Gamma}} \Ind^{\wt{G}^{\bi}}_{G^{\bi}_{\Gamma}} (\bx) \bigr)
			=
			\bx \circ \kappa_{\Res^{\wt{G}^{\bi}}_{G^{\bi}_{\Gamma}}A}
		\end{align*}
		by the naturality of $\kappa$ 
		and the unit-counit relation 
		\begin{align*}
			&
			\kappa_{\Res^{\wt{G}^{\bi}}_{G^{\bi}_{\Gamma}}A} 
			\circ \Res^{\wt{G}^{\bi}}_{G^{\bi}_{\Gamma}} (\eta_{A})
			= \id. 
		\qedhere\end{align*}
	\end{proof}
	
	Finally, we apply \cref{thm:grtwtransfer} to graded twisting of compact Lie groups considered in \cite{Neshveyev-Yamashita:twilie}. 
	For each connected, simply connected compact Lie group $G$, $0<q\leq 1$, and $\tau\in \cZ(G)^{\dim\fh}$, where $\fh$ is the Cartan subalgebra of the Lie algebra of $G$, 
	they constructed the graded twisting $G_q^{\tau}$ for the invariant cocentral action of a finite abelian group $\Gamma$ associated with $\tau$. 
	Its $\Gamma$-action is a restriction of the action of $T/\cZ(G)$ on $G$, given by the adjoint action of $T$. 
	See \cite[Theorem~3.1]{Neshveyev-Yamashita:twilie} for detail. 
	When $q=1$, this situation satisfies the assumption of the following result by putting $F:=G^{\tau}$. 
	
	\begin{thm}\label{thm:grtwLiehmg}
		Under \cref{cond:actionfinab}, we additionally assume that 
		$G$ is a connected compact Lie groups with a torsion-free fundamental group, 
		and there is a group homomorphism $\psi\colon \Gamma \to T/\cZ(G)$, where $T$ is a maximal torus of $G$, such that the action $\Gamma\curvearrowright G$ is given by $\Ad \psi$. 
		Then we have the following. 
		\begin{enumerate}
			\item
			$\KK^{F^{\bi}} = \bra \KK \ket_{\loc}^{F^{\bi}}$. 
			\item
			$C^r(F/T)\cong \bC^{\oplus n}$ in $\KK^{\sD(F)}$, for some $n\in\bZ_{>0}$. 
		\end{enumerate}
	\end{thm}
	
	Here, note that the restriction $f\colon C(G)\to C(T)$ gives a well-defined surjective $*$-homomorphism $\wt{f}\colon C(H)=\Gamma\redltimes C(G)\to C(T)$ preserving $\Delta_H$ 
	that sends $u^s x$ for $x\in C^r(G)$, $s\in\Gamma$ to $f(x)$ 
	since $\Gamma$ acts trivially on $T$ by $\Ad \psi$. 
	It follows $\wt{f}(C^r(F))=\wt{f}(C^r(H))=C(T)$ from 
	$C^r(H)=[u^t C^r(F) \,|\, t\in\Gamma]$ 
	and thus $\wt{f}$ identifies $T$ as a closed subgroup of $F$ by \cite[Theorem~6.1]{Daws-Kasprzak-Skalski-Soltan:closed}. 
	
	\begin{proof}
		We put $K:=T/\cZ(G)$ and define $G^{\bi}_{K}$ similarly to $G^{\bi}_{\Gamma}$. 
		Since $\pi_1(G^{\bi}_{K})\cong \pi_1(G\times K\times G)$ is torsion-free, it follows from \cite{Meyer-Nest:dualcpt} that $C(G)\in \KK^{G^{\bi}_{K}} = \bra \KK \ket_{\loc}^{G^{\bi}_{K}}$, 
		where we equip $G$ with the left and right translations of $G$ and the adjoint $K$-action. 
		It follows $\Res^{G^{\bi}_{K}}_{G^{\bi}_{\Gamma}} C(G)\in \bra \KK\ket_{\loc}^{G^{\bi}_{\Gamma}}$, and thus 
		$C(\wt{G})\cong \Ind^{\wt{G}^{\bi}}_{G^{\bi}_{\Gamma}} C(G) 
		\in \bra C(\Gamma)\otimes \KK\ket_{\loc}^{\wt{G}^{\bi}}$. 
		From this we can show $\KK^{\wt{G}^{\bi}} = \bra \KK^{\wt{G}\times \Gamma^{\op}} \ket_{\loc}^{\wt{G}^{\bi}} = \bra\KK^{\Gamma^{\bi}}\ket_{\loc}^{\wt{G}^{\bi}}$ 
		by replacing $G$ with $\wt{G}$ in the proof of (1')$\Rightarrow$(3')$\Rightarrow$(3) in \cref{prop:gamma1}. 
		By applying $\bfF^{\bi}$, we get $\KK^{\wt{F}^{\bi}}=\bra\KK^{\Gamma^{\bi}}\ket_{\loc}^{\wt{F}^{\bi}}$ 
		and (1) follows by applying $\Res^{\wt{F}^{\bi}}_{F^{\bi}}$ since 
		$\Res^{\wt{F}^{\bi}}_{F^{\bi}}\Ind^{\wt{F}^{\bi}}_{F^{\bi}}A$ contains $A$ as an equivariant direct summand for $A\in \Calg^{F^{\bi}}$. 
		
		By \cite{Rosenberg-Schochet:equivKunnethUCT,Meyer-Nest:dualcpt}, it holds 
		$C(G/T)=C((G\rtimes K)/(T\times K))\cong \bC^{\oplus n}$ in $\KK^{G\rtimes K}$ for some $n\in\bZ_{>0}$, 
		and thus in $\KK^{G^{\bi}_{K}}$ with the trivial $G^{\op}$-actions. 
		Note that $n$ must equal the order of the Weyl group of $G$ by \cite{Steinberg:Pittie,Goffeng:PVcptLie}. 
		Using well-defined endofunctor on $\KK^{G^{\bi}_{K}}$ that sends 
		each $A\in\KK^{G^{\bi}_{K}}$ to $A\otimes C(G)$ with the diagonal $G^{\bi}_{K}$-action, 
		we see $C(G/T)\otimes C(G)\cong C(G)^{\oplus n}$ in $\KK^{G^{\bi}_{K}}$, and thus in $\KK^{G^{\bi}_{\Gamma}}$. 
		Note that $C(G/T)\otimes C(G)\cong C(G)\cotensor{T}C(G)$ as left $G^{\bi}_{K}$-C*-algebras. 
		Here we equip $C(G)\cotensor{T}C(G)$ with a left $(G\rtimes K)^{\bi}$-action by 
		the left $G$-action $\Delta_G\otimes\id$, the right $G$-action $\id\otimes\Delta_G$, and taking adjoint on the left [resp.~right] component of $G$ as the left [resp.~right] $K$-action. 
		Now (1) and \cref{thm:grtwtransfer} shows $C(F)^{\oplus n} \cong \bfF^{\bi}(C(G)\cotensor{T}C(G))$ in $\KK^{F^{\bi}}$. 
		We can check the left $F^{\bi}$-$*$-isomorphism $\bfF^{\bi}(C(G)\cotensor{T}C(G))\cong C(F)\cotensor{T}C(F)$ because the right [resp.~left] translation of $T$ on $G$ commutes 
		with the left [resp.~right] action of $\wt{G}$ given by the restriction of the $G^{\bi}_\Gamma$-action on $G$. 
		Finally, \cref{prop:hmgvscotensor} concludes the proof (cf.~\cite[Theorem~4.7]{Kitamura:indhomlcqg}). 
	\end{proof}
	
	For example, a graded twisting of \cite{Neshveyev-Yamashita:twilie} gives rise to $SU_{-q}(2)$ from $SU_q(2)$. 
	In this case, $C(SU_{-q}(2)/T)$ is $\KK^{\sD(SU_{-q}(2))}$-equivalent to $\bC^{\oplus2}$ for $0<q<1$ as we have seen in \cref{rem:BCSU-q(2)}. We record the remaining case of $q=1$. 
	
	\begin{cor}
		It holds $C(SU_{-1}(2)/T)\cong \bC^{\oplus 2}$ in $\KK^{\sD(SU_{-1}(2))}$. 
	\qed\end{cor}
	
	Unfortunately, the statement of \cref{thm:grtwLiehmg} can no longer be hoped for $G_q$ with $0<q<1$ except $SU_q(2)$, because of the property (T) of $\sD(G_q)$ due to \cite{Arano:cpxssDrinfeld}. 
	This property is known to be an obstruction to the property $\gamma=1$ for groups, and at least the proof of the $\KK^{\sD(SU_q(2))}$-equivalence of $C(SU_q(2)/T)$ and $\bC^{\oplus2}$ by \cite{Voigt:bcfo} does not work for higher rank cases (cf.~\cite{Voigt-Yuncken:equivFredholm}). 
	Although we will not go further here, this situation also suggests the difficulty towards an analog of the property $\gamma=1$ for $\sD(G_q)$ and thus for $\wh{G}_q\times \wh{G}_q^{\op}$ in view of \cref{prop:gamma1}, despite the latter being amenable.

	\providecommand{\bysame}{\leavevmode\hbox to3em{\hrulefill}\thinspace}

\end{document}